\documentclass[a4paper,12pt]{amsart}
\usepackage[ansinew]{inputenc}
\usepackage{amscd}
\usepackage{amsmath}

\usepackage{amsthm}

\usepackage{graphicx}

\usepackage{txfonts}
\usepackage{dsfont}
\usepackage{mathabx}

\usepackage{tikz-cd}
\usepackage[titletoc]{appendix}

\usepackage{geometry}
\usepackage{tikz}

\newtheorem{Defn}{Definition}[section]
\newtheorem{Lemma}[Defn]{Lemma}
\newtheorem{prop}[Defn]{Proposition}
\newtheorem{theorem}[Defn]{Theorem}
\newtheorem{remark}[Defn]{Remark}
\newtheorem{Corollary}[Defn]{Corollary}
\newtheorem{Example}[Defn]{Example}

\title{Integrals along bimonoid homomorphisms}
\author{MINKYU KIM}
\date{}

\address{Graduate School of Mathematical Sciences \\ University of Tokyo}
\email{kim@ms.u-tokyo.ac.jp}

\begin{document}

\maketitle

\begin{abstract}
We introduce a notion of {\it an integral along a bimonoid homomorphism} as a simultaneous generalization of the integral and cointegral of bimonoids.
The purpose of this paper is to characterize an existence of a specific integral, called {\it a normalized generator integral}, along a bimonoid homomorphism in terms of the kernel and cokernel of the homomorphism.

We introduce a notion of {\it a volume on an abelian category} as a generalization of the dimension of vector spaces and the order of abelian groups.
In applications, we show that there exists a nontrivial volume partially defined on a category of bicommutative Hopf monoids.
The volume yields a notion of Fredholm homomorphisms between bicommutative Hopf monoids, which gives an analogue of the Fredholm index theory.
This paper gives a technical preliminary of our subsequent paper about a construction of TQFT's.
\end{abstract}

\setcounter{tocdepth}{2}
\tableofcontents

\section{Introduction}
\label{202002271058}

The notion of integrals of a bialgebra is introduced by Larson and Sweedler \cite{LarSwe}.
It is a generalization of the Haar measure of a group.
The integral theory has been used to study bialgebras or Hopf algebras \cite{LarSwe} \cite{sweedler1969integrals} \cite{radford1976order}.
The notion of bialgebras are generalized to bimonoids in a symmetric monoidal category $\mathcal{C}$ \cite{MS} \cite{mac}. 
The integral theory is generalized to the categorical settings and used to study bimonoids or Hopf monoids \cite{bespalov2000integrals}.

In this paper, we introduce a notion of an integral along a bimonoid homomorphism in a symmetric monoidal category $\mathcal{C}$.
For bimonoids $A$ and $B$, an integral along a bimonoid homomorphism $\xi : A \to B$ is a morphism $\mu : B \to A$ in $\mathcal{C}$ satisfying some axioms (see Definition \ref{201908021117}).
The integrals along bimonoid homomorphisms simultaneously generalize the notions of integral and cointegral of bimonoid : the notion of integral (cointegrals, resp.) of a bimonoid coincides with that of integrals along the counit (unit, resp.).

The purpose of this paper is to characterize the existence of a normalized generator integral along a bimonoid homomorphism in terms of the kernel and cokernel of the homomorphism.
The reader is referred to Definition \ref{201908021117} and \ref{201907311452} for the definition of normalized integrals and generator integrals respectively.
If $\mathcal{C}$ satisfies some assumptions (see (Assumption 0,1,2) in section \ref{202002201415}), then the existence of a normalized generator integral is characterized as follows.
Note that the assumptions on $\mathcal{C}$ automatically hold if $\mathcal{C} = \mathsf{Vec}^\otimes_{k}$, the tensor category of (possibly infinite-dimensional) vector spaces over a field ${k}$ :

\begin{theorem}
\label{202002211033}
Let $A,B$ be bicommutative Hopf monoids in $\mathcal{C}$ and $\xi : A \to B$ be a Hopf homomorphism.
Then there exists a normalized generator integral $\mu_\xi$ along $\xi$ if and only if the following conditions hold :
\begin{enumerate}
\item
the kernel Hopf monoid $Ker(\xi)$ has a normalized integral.
\item
the cokernel Hopf monoid $Cok (\xi)$ has a normalized cointegral.
\end{enumerate}
Moreover, if a normalized integral exists, then it is unique.
\end{theorem}

Note that even if $A,B$ are not bicommutative or it is not obvious whether $\mathcal{C}$ satisfies (Assumption 0,1,2), we have more general results under some assumptions on the homomorphism $\xi$.
The integral is constructed concretely.
We prove such a generalization in Corollary \ref{202002211035} which implies our main theorem.

In applications, we investigate the category $\mathsf{Hopf}^\mathsf{bc,\star} ( \mathcal{C} )$ of bicommutative Hopf monoids with a normalized integral and cointegral.
We prove that the category $\mathsf{Hopf}^\mathsf{bc,\star} ( \mathcal{C} )$ is an abelian subcategory of $\mathsf{Hopf}^\mathsf{bc} ( \mathcal{C})$ and closed under short exact sequences.
See section \ref{202002201415}\footnote{By Theorem \ref{small_integral_equiv}, the category $\mathsf{Hopf}^\mathsf{bc,\star} ( \mathcal{C} )$ coincides with $\mathsf{Hopf}^\mathsf{bc,bs} ( \mathcal{C} )$ in section \ref{202002201415}.}.

We introduce a notion of volume on an abelian category $\mathcal{A}$ as a generalization of the dimension of vector spaces and the order of abelian groups.
It is an invariant of objects in $\mathcal{A}$ compatible with short exact sequences (see Definition \ref{201911231935}).
As another application to $\mathsf{Hopf}^\mathsf{bc,\star} ( \mathcal{C} )$, we construct an $End_\mathcal{C} ( \mathds{1} )$-valued volume $vol^{-1}$ on the abelian category $\mathcal{A} = \mathsf{Hopf}^\mathsf{bc,\star} ( \mathcal{C} )$.
Here $\mathds{1}$ is the unit object of $\mathcal{C}$ and the endomorphism set $End_\mathcal{C} ( \mathds{1} )$ is an abelian monoid induced by the symmetric monoidal structure of $\mathcal{C}$.

By using the volume $vol^{-1}$, we introduce a notion of {\it Fredholm homomorphisms} between bicommutative Hopf monoids as an analogue of the Fredholm operator \cite{kato2013perturbation} (see Definition \ref{201911232050}).
We study its {\it index} which is robust to some finite perturbations (see Proposition \ref{202002271226}).
Furthermore, we construct a functorial assignment of integrals to Fredholm homomorphisms.

This paper gives a technical preliminary of our subsequent paper \cite{kim2020family}.
Indeed, we use the results in this paper to give a generalization of the untwisted abelian Dijkgraaf-Witten theory \cite{DW} \cite{Wakui} \cite{FQ} and the bicommutative Turaev-Viro TQFT \cite{TV} \cite{BW}.
We will give a systematic way to construct a sequence of TQFT's from (co)homology theory.
The TQFT's are constructed by using {\it path-integral} which is formulated by some integral along bimonoid homomorphisms.

As a corollary of our subsequent paper, if the volume $vol^{-1} (A)$ of an object $A$ in $\mathsf{Hopf}^\mathsf{bc,\star} ( \mathcal{C} )$ is invertible in $End_\mathcal{C} ( \mathds{1} )$, then the underlying object of $A$ is dualizable in $\mathcal{C}$ and its categorical dimension coincides with the inverse of $vol^{-1} (A)$.
If $\mathcal{C}$ is a rigid symmetric monoidal category with split idempotents, then the inverse volume of any Hopf monoid is invertible \cite{bespalov2000integrals}.
It is not obvious whether the inverse volume is invertible or not in general.
Note that we do not assume a duality on objects of $\mathcal{C}$.

There is another approach to a generalization of (co)integrals of bimonoids.
In \cite{bespalov2000integrals}, (co)integrals are defined by a universality.
It is not obvious whether our integrals could be generalized by universality.

We expect that the result in this paper could be applied to topology through another approach.
There is a topological invariant of 3-manifolds induced by a finite-dimensional Hopf algebra, called the Kuperberg invariant \cite{kuperberg1991involutory} \cite{kuperberg1997non}.
In particular, if the Hopf algebra is involutory, then it is defined by using the normalized integral and cointegral of the Hopf algebra.

\vspace{2mm}

The organization of this paper is as follows.
In section \ref{201912060935}, we give our convention for string diagrams and a brief review of monoids in a symmetric monoidal category.
In subsection \ref{201907220213}, we review integrals of bimonoids.
In subsection \ref{201908051600}, we introduce the notion of (normalized) integral along bimonoid homomorphisms and give some basic properties.
In subsection \ref{201908051612}, we introduce a notion of generator integral and give some basic properties.
In subsection \ref{201907220216}, \ref{201907220217}, we introduce the notion of invariant objects and stabilized objects respectively.
In subsection \ref{201907220219}, we introduce the notion of (co, bi) stable monoidal structure.
In section \ref{201908051552}, we introduce the notions of (co,bi)normality of bimonoid homomorphisms and give some basic properties.
In section \ref{201912060950}, we introduce the notion of (co, bi) small bimonoids and examine it in terms of an existence of normalized (co)integrals.
In subsection \ref{201908051601}, we prove the uniqueness of a normalized integral.
In subsection \ref{201908051602}, we prove some necessary conditions for existence of a normalized integral.
In section \ref{201908021404}, by using a normalized generator integral, we show an isomorphism between the set of endomorphisms on the unit object $\mathds{1}$ and the set of integrals.
In subsection \ref{201908051614}, we prove a key lemma to prove the main theorem.
In subsection \ref{201908051616},  we introduce two notions of (weakly) well-decomposable homomorphism and (weakly) Fredholm homomorphism, and prove the main theorem.
In section \ref{201908051617}, we investigate a commutativity of some homomorphisms and normalized integrals.
In subsection \ref{201908051620}, we introduce the inverse volume of some bimonoids.
In subsection \ref{201912060955}, we introduce the inverse volume of some bimonoid homomorphisms.
In subsection \ref{201908051621}, we study a commutativity of normalized integrals.
In subsection \ref{201908041309}, we give some conditions where $Ker(\xi)$, $Cok (\xi)$ inherits a (co)smallness from that of the domain and the target of $\xi$. 
In subsection \ref{201908051618}, we study bismallness of bimonoids in an exact sequence. 
In section \ref{201912022325}, we introduce the notion of volume on an abelian category and study basic notions related with it.
In subsection \ref{201911161328}, we prove that the inverse volume is a volume on the category of bicommutative Hopf monoids.
In subsection \ref{201908010930}, we construct functorial integrals for Fredholm homomorphisms.


\section*{Acknowledgements}
The author appreciates Christine Vespa who read this paper carefully and gave helpful comments.
The author was supported by FMSP, a JSPS Program for Leading Graduate Schools in the University of Tokyo, and JPSJ Grant-in-Aid for Scientific Research on Innovative Areas Grant Number JP17H06461.

\section*{Data availablity}

Data sharing not applicable to this article as no datasets were generated or analysed during the current study.


\section{Notations}
\label{201912060935}

This section gives our convention about notations.
The reader is referred to some introductory books for category theory or (Hopf) monoid theory \cite{mac} \cite{MS}.

We denote by $\mathds{1}$ the unit object of a monoidal category $\mathcal{C}$, by $\otimes$ the monoidal operation.
We often omit the coherence isomorphisms : the associator $\mathbf{a}_{x,y,z} : (x \otimes y ) \otimes z \to x \otimes ( y \otimes z )$, the right unitor $\mathbf{r}_x : x \otimes \mathds{1} \to x$ and the left unitor $\mathbf{l}_x : \mathds{1} \otimes x \to x$ ; and denote by $\cong$.

\vspace{0.1cm}
{\it String diagrams.}
We explain our convention to represent {\it string diagrams}.
It is convenient to use string diagrams to discuss equations of morphisms in a symmetric monoidal category $\mathcal{C}$.
It is based on finite graphs where for each vertex $v$ the set of edges passing through $v$ has a partition by, namely, {\it incoming edges} and {\it outcoming edges}.
For example, a morphism $f : x \to y$ in $\mathcal{C}$ is represented by (1) in Figure \ref{201912060940}.
In this example, the underlying graph has one 2-valent vertex and two edges.
If there is no confusion from the context, we abbreviate the objects as (2) in Figure \ref{201912060940}.
For another example, a morphism $g : a \otimes b \to x \otimes y \otimes z$ is represented by (3) in Figure \ref{201912060940}.

\begin{figure}[ht]
  \includegraphics[width=8cm]{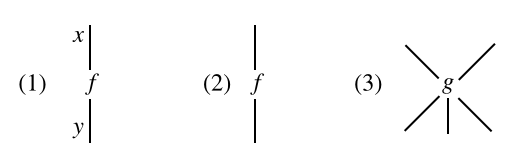}
  \caption{}
  \label{201912060940}
\end{figure}

We represent the tensor product of morphisms in a symmetric monoidal category $\mathcal{C}$ by gluing two string diagrams.
For example, if $h : x \to y$, $k : a \to b$ are morphisms, then we represent $h \otimes k : x \otimes a \to y \otimes b$ by (1) in Figure \ref{201912060941}.

We represent the composition of morphisms by connecting some edges of string diagrams.
For example, if $q : x \to y$ and $p : y \to z$ are morphisms, we represent their composition $p \circ q : x \to z$ by (2) in Figure \ref{201912060941}.

\begin{figure}[ht]
  \includegraphics[width=4.5cm]{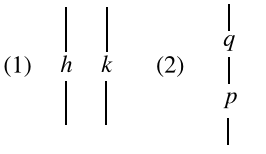}
  \caption{}
  \label{201912060941}
\end{figure}

The symmetry $\mathbf{s}_{x,y} : x \otimes y \to y \otimes x$ which is a natural isomorphism is denoted by (1) in Figure \ref{201912072118}.

The edge colored by the unit object $\mathds{1}$ of the symmetric monoidal category $\mathcal{C}$ is abbreviated.
For example, a morphism $u : \mathds{1} \to a$ is denoted by (2) in Figure \ref{201912072118} and a morphism $v : b \to  \mathds{1}$ is denoted by (3) in Figure \ref{201912072118}.

\begin{figure}[ht]
  \includegraphics[width=7.5cm]{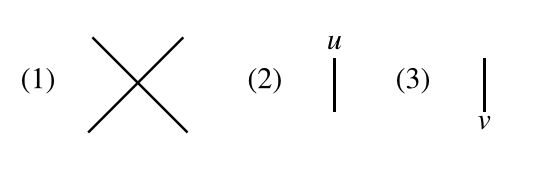}
  \caption{}
  \label{201912072118}
\end{figure}

\vspace{0.1cm}
{\it Monoid.}
The notion of monoid in a symmetric monoidal category is a generalization of the notion of {\it monoid} which is a set equipped with a unital and associative product.
Furthermore, it is a generalization of the notion of {\it algebra}.
We use the notations $\nabla : A \otimes A \to A$ and $\eta : \mathds{1} \to A$ to represent the multiplication and the unit.
On the one hand, the comonoid is a dual notion of the monoid.
We use the notations $\Delta : A  \to A\otimes A$ and $\epsilon : A \to \mathds{1} $ to represent the comultiplication and the counit.
Figure \ref{201912042210} denotes the structure morphisms as string diagrams.

The notions of bimonoid and Hopf monoid are defined as an object of $\mathcal{C}$ equipped with a monoid structure and a comonoid structure which are subject to some axioms.
We denote by $\mathsf{Bimon} (\mathcal{C}), \mathsf{Hopf} (\mathcal{C})$ the categories of bimonoids and Hopf monoids respectively.

\begin{figure}[ht]
  \includegraphics[width=6.5cm]{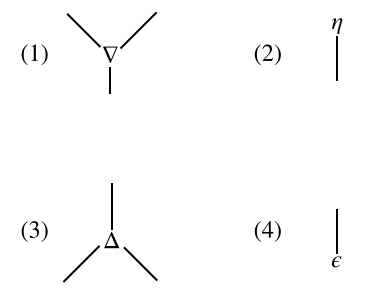}
  \caption{}
  \label{201912042210}
\end{figure}


{\it Action.}
We give some notations about actions in a symmetric monoidal category.
The notations related with coaction is defined similarly.

\begin{Defn}
\label{201911301531}
\rm
Let $X$ be an object of $\mathcal{C}$, $A$ be a bimonoid, and $\alpha : A \otimes X \to X$ be a morphism in $\mathcal{C}$.
A triple $(A,\alpha , X)$ is a {\it left action} in $\mathcal{C}$ if following diagrams commute :
\begin{equation}
\begin{tikzcd}
A\otimes A\otimes X \ar[r, "id_A \otimes \alpha"] \ar[d, "\nabla_A \otimes id_X"] & A\otimes X \ar[d, "\alpha"] \\
A\otimes X \ar[r, "\alpha"] & X
\end{tikzcd}
\end{equation}
\begin{equation}
\begin{tikzcd}
 \mathds{1} \otimes X  \ar[dr, "\cong"'] \ar[r, "\eta_A \otimes id_X"] & A \otimes X \ar[d, "\alpha"] \\
& X
\end{tikzcd}
\end{equation}
Let $(A, \alpha, X)$, $(A^\prime, \alpha^\prime , X^\prime)$ be left actions in a symmetric monoidal category $\mathcal{C}$.
A pair $(\xi_0 , \xi_1) : (A, \alpha, X) \to (A^\prime, \alpha^\prime , X^\prime)$ is a {\it morphism of left actions} if $\xi_0 : A \to A^\prime$ is a monoid homomorphism and $\xi_1 : X \to X^\prime$ is a morphism in $\mathcal{C}$ which intertwines the actions.

Left actions in $\mathcal{C}$ and morphisms of left actions form a category which we denote by $\mathsf{Act}_{l} (\mathcal{C})$.
The symmetric monoidal category structures of $\mathcal{C}$ and $\mathsf{Bimon}(\mathcal{C})$ induce a symmetric monoidal category on $\mathsf{Act}_{l} (\mathcal{C})$ by $(A,\alpha, X) \otimes (A^\prime, \alpha^\prime, X^\prime) \stackrel{\mathrm{def.}}{=} (A \otimes A^\prime, \alpha \tilde{\otimes} \alpha^\prime , X \otimes X^\prime)$.
Here, $\alpha \tilde{\otimes} \alpha^\prime  : (A \otimes A^\prime) \otimes (X \otimes X^\prime) \to  X \otimes X^\prime$ is defined by composing
\begin{align}
A \otimes A^\prime \otimes X \otimes X^\prime 
\stackrel{id_A \otimes \mathbf{s}_{A^\prime, X} \otimes id_{X^\prime}}{\longrightarrow} A \otimes  X \otimes A^\prime \otimes X^\prime 
\stackrel{\alpha \otimes \alpha^\prime}{\longrightarrow} X \otimes X^\prime .
\end{align}

We define a {\it right action in} a symmetric monoidal category $\mathcal{C}$ and its morphism similarly.
Note that for a right action, we use the notation $(X,\alpha, A)$ where $A$ is a bimonoid and $X$ is an object on which $A$ acts.
We denote by $\mathsf{Act}_{r} (\mathcal{C})$ the symmetric monoidal category of right actions.

Let $A$ be a bimonoid in a symmetric monoidal category $\mathcal{C}$ and $X$ be an object of $\mathcal{C}$.
A left action $(A,\tau_{A,X} , X)$ is {\it trivial} if
\begin{align}
\tau_{A,X} : A \otimes X \stackrel{\epsilon_A \otimes id_X}{\to} \mathds{1} \otimes X \stackrel{\cong}{\to} X . 
\end{align}
We also define a {trivial right action} analogously.
We abbreviate $\tau = \tau_{A,X}$ if there is no confusion.
\end{Defn}

\begin{remark}
The notion of action is usually defined for a monoid $A$, but we require that $A$ should be a bimonoid in Definition \ref{201911301531}.
In fact, the trivial action is well-defined since $A$ has a counit which is a structure of a comonoid.
\end{remark}


\section{Integrals}


\subsection{Integrals of bimonoids}
\label{201907220213}

In this subsection, we review the notion of integral of a bimonoid and its basic properties.

We give some remarks on terminology.
The integral in this paper is called a Haar integral \cite{BalKir}, \cite{BMCA}, \cite{meusburger}, an $Int (H)$-based integral \cite{bespalov2000integrals} or an integral-element \cite{KTLV}.
The cointegral in this paper is called an $Int(H)$-valued integral in \cite{bespalov2000integrals} or integral-functional \cite{KTLV}.
In fact, those notions introduced in \cite{bespalov2000integrals}, \cite{KTLV} are more general ones which are defined by a universality.

\begin{Defn}
\label{Defn_Haar}
\rm
Let $A$ be a bimonoid.
A morphism $\varphi : \mathds{1} \to A$ is a {\it left integral} of $A$ if it satisfy a commutative diagram (\ref{Haar_axiom2}).
A morphism $\varphi : \mathds{1} \to A$ is a {\it right integral} if it satisfy a commutative diagram (\ref{Haar_axiom3}).
A morphism $\varphi : \mathds{1} \to A$ is an {\it integral} if it is a left integral and a right integral.
A left (right) integral is {\it normalized} if it satisfies a commutative diagram (\ref{Haar_axiom1}).
For a bimonoid $A$, we denote by $\sigma_A : \mathds{1} \to A$ the normalized integral of $A$ if exists.
It is unique for $A$ as we will discuss in this section.
Denote by $Int_r (A), Int_l ( A), Int (A)$ the set of right integrals, left integrals and integrals of $A$.

We analogously define {\it cointegral} of a bimonoid as a morphism from $A$ to $\mathds{1}$.
Denote by $Cont_r (A), Coint_l ( A), Coint (A)$ the set of right cointegrals, left cointegrals and cointegrals of $A$.

\begin{equation}
\label{Haar_axiom2}
\begin{tikzcd}
\mathds{1} \otimes A \ar[r, "\varphi \otimes id_A"] \ar[d, "\varphi \otimes \epsilon_A"] & A \otimes A \ar[d, "\nabla_A"]  \\
 A \otimes \mathds{1} \ar[r, "\cong"] & A
\end{tikzcd}
\end{equation}
\begin{equation}
\label{Haar_axiom3}
\begin{tikzcd}
A \otimes \mathds{1} \ar[r, "id_A \otimes \varphi"] \ar[d, "\epsilon_A \otimes \varphi"] & A \otimes A \ar[d, "\nabla_A"] \\
\mathds{1} \otimes A \ar[r , "\cong"] & A
\end{tikzcd}
\end{equation}
\begin{equation}
\label{Haar_axiom1}
\begin{tikzcd}
\mathds{1}  \ar[r, "\varphi"] \ar[dr, equal] & A \ar[d, "\epsilon_A"] \\
& \mathds{1}
\end{tikzcd}
\end{equation}
\end{Defn}

\begin{remark}
The commutative diagrams in Definition \ref{Defn_Haar} can be understood by equations of some string diagrams in Figure \ref{201911231145} where the null diagram is the identity on the unit $\mathds{1}$.

\begin{figure}[ht]
\includegraphics[width=3.7cm]{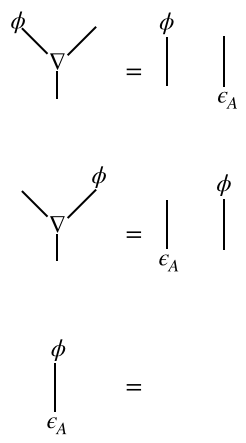}
  \caption{}
 \label{201911231145}
\end{figure}
\end{remark}

\begin{prop}
\label{norm_r_l_integral_is_integral}
Let $A$ be a bimonoid in a symmetric monoidal category, $\mathcal{C}$.
If the bimonoid $A$ has a normalized left integral $\sigma$ and a normalized right integral $\sigma^{\prime}$, then $\sigma = \sigma^\prime$ and it is a normalized integral of the bimonoid $A$.
In particular, if a normalized integral exists, then it is unique.
\end{prop}
\begin{proof}
It is proved by definitions.
It follows from more general results in Proposition \ref{r_integral_l_integra_coincide_along}.
In fact, a normalized left (right) integral of $A$ is a normalized left (right) integral along counit of $A$.
\end{proof}


\subsection{Integrals along bimonoid homomorphisms}
\label{201908051600}

In this subsection, we introduce the notion of {\it an integral along a homomorphism} and give its basic properties.
They are defined for bimonoid homomorphisms whereas the notion of (co)integrals is defined for bimonoids.
In fact, it is a generalization of (co)integrals.
See Proposition \ref{201907311133}.
We also give a typical example in Example \ref{202002271724}.

\begin{Defn}
\label{201908021117}
\rm
Let $A,B$ be bimonoids in a symmetric monoidal category $\mathcal{C}$ and $\xi : A\to B$ be a bimonoid homomorphism.
A morphism $\mu : B \to A$ in $\mathcal{C}$ is a {\it right integral along} $\xi$ if the diagrams (\ref{Haar_fam_axiom1}), (\ref{Haar_fam_axiom2}) commute.
A morphism $\mu : B \to A$ in $\mathcal{C}$ is a {\it left integral along} $\xi$ if the diagrams (\ref{Haar_fam_axiom3}), (\ref{Haar_fam_axiom4}) commute.
A morphism $\mu : B \to A$ in $\mathcal{C}$ is an {\it integral along} $\xi$ if it is a right integral along $\xi$ and a left integral along $\xi$.
An integral (or a right integral, a left integral) is {\it normalized} if the diagram (\ref{Haar_fam_axiom5}) commutes.

We denote by $Int_l ( \xi )$, $Int_r (\xi)$, $Int (\xi)$ the set of left integrals along $\xi$, the set of right integrals along $\xi$, the set of integrals along $\xi$ respectively.
\begin{equation}
\label{Haar_fam_axiom1}
\begin{tikzcd}
B\otimes A \ar[r, "\mu \otimes id_A"] \ar[d, "id_B \otimes \xi"] & A \otimes A \ar[r, "\nabla_A"] & A \\
B \otimes B \ar[r, "\nabla_B"] & B \ar[ur, "\mu"] & 
\end{tikzcd}
\end{equation}
\begin{equation}
\label{Haar_fam_axiom2}
\begin{tikzcd}
B \ar[r, "\Delta_B"] \ar[d, "\mu"] & B \otimes B \ar[r, "\mu \otimes id_B"] & A \otimes B \\
A \ar[r, "\Delta_A"] & A \otimes A \ar[ur, "id_A \otimes \xi"] 
\end{tikzcd}
\end{equation}
\begin{equation}
\label{Haar_fam_axiom3}
\begin{tikzcd}
A \otimes B \ar[r, "id_A \otimes \mu"] \ar[d, "\xi \otimes id_B"] & A \otimes A \ar[r, "\nabla_A"] & A \\
B \otimes B \ar[r, "\nabla_B"] & B \ar[ur, "\mu"'] & 
\end{tikzcd}
\end{equation}
\begin{equation}
\label{Haar_fam_axiom4}
\begin{tikzcd}
B \ar[r, "\Delta_B"] \ar[d, "\mu"] & B \otimes B \ar[r,  "id_B \otimes \mu"] & B \otimes A \\
A \ar[r, "\Delta_A"] & A \otimes A \ar[ur, "\xi \otimes id_A"'] 
\end{tikzcd}
\end{equation}
\begin{equation}
\label{Haar_fam_axiom5}
\begin{tikzcd}
A \ar[rrr, "\xi"] \ar[dr, "\xi"] & & & B \\
& B \ar[r, "\mu"] & A \ar[ur, "\xi"] &
\end{tikzcd}
\end{equation}
\end{Defn}

\begin{remark}
The commutative diagrams in Definition \ref{201908021117} can be understood by using some string diagrams in Figure \ref{201908021111}.
From now on, we freely use these string diagrams.
The string diagram is explained briefly in appendix \ref{201912060935}.
\begin{figure}[ht]
  \includegraphics[width=13cm]{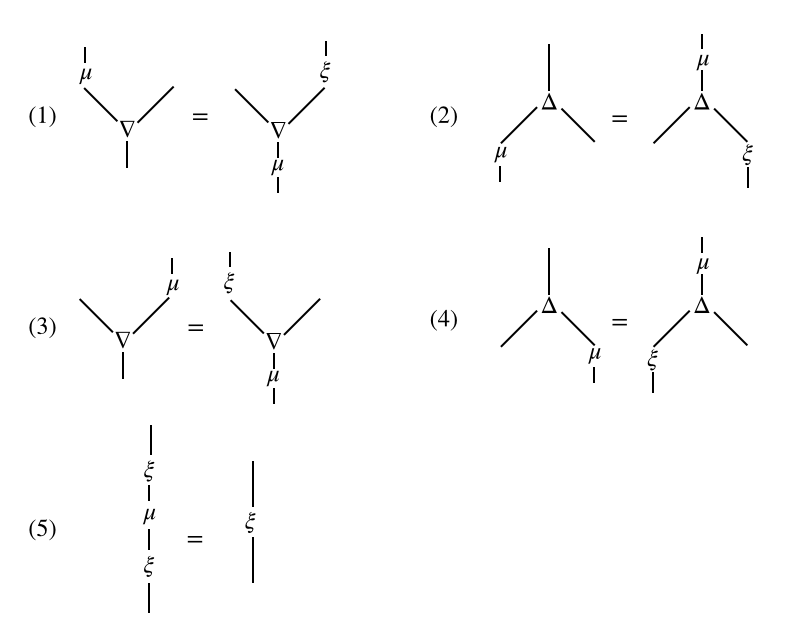}
  \caption{}
  \label{201908021111}
\end{figure}
\end{remark}

\begin{Example}
\label{202002271724}
Consider $\mathcal{C} = \mathsf{Vec}^\otimes_{k}$.
Let $G,H$ be (possibly infinite and non-abelian) groups and $\varrho : G \to H$ be a group homomorphism.
It induces a bialgebra homomorphism $\xi = \varrho_\ast : A \to B$ for the group Hopf algebras $A= {k} G$ and $B = {k} H$.
Suppose that the kernel $Ker (\varrho)$ is finite.
Put $\mu : B \to A$ by
\begin{align}
\mu ( h ) = \sum_{\varrho (g) = h} g .
\end{align}
Then $\mu$ is an integral along the homomorphism $\xi$.
If the order of $Ker (\varrho)$, say $N$, is coprime to the characteristic of ${k}$, then $N^{-1} \cdot \mu$ is a normalized integral.
\end{Example}

\begin{prop}
\label{201907311133}
Let $A$ be a bimonoid in a symmetric monoidal category $\mathcal{C}$.
For $\star = r,l$, we have
\begin{align}
& Int_\star (\epsilon_A ) = Int_\star (A),  \\
& Int_\star (\eta_A ) = Coint_\star (A) .  
\end{align}
In particular, we have
\begin{align}
& Int (\epsilon_A ) = Int (A), \\
& Int (\eta_A ) = Coint (A) .  
\end{align}
Under these equations, the normalized condition is preserved.
\end{prop}
\begin{proof}
We only prove that $Int_r (\epsilon_A) = Int_r (A)$ and leave the other parts to the readers.

Suppose that $\mu \in Int_r (\epsilon_A)$.
Then by (\ref{Haar_fam_axiom1}), we have $\nabla_A \circ (\mu \otimes id_A) = \mathbf{r}_A \circ (\mu \otimes \epsilon_A)$, i.e. $\mu$ is a right integral of the bimonoid $A$ where $\mathbf{r}_A$ is the right unitor.

Suppose that $\sigma \in Int_r(A)$.
Then $\sigma$ satisfies the commutative diagram (\ref{Haar_fam_axiom1}).
On the other hand, (\ref{Haar_fam_axiom2}) is automatic since $B= \mathds{1}$.

Note that $\mu \in Int_r (\epsilon_A)$ is normalized ,i.e. $\epsilon_A \circ \mu \circ \epsilon_A = \epsilon_A$, if and only if $\epsilon_A \circ \mu = id_{\mathds{1}}$.
\end{proof}

\begin{prop}
\label{201907311053}
If a bimonoid homomorphism $\xi : A \to B$ is an isomorphism, then we have $\xi^{-1} \in E (\xi )$.
Here, $E$ denotes either $Int_r$, $Int_l$ or $Int$.
In particular, $id_A \in E (id_A)$ for any bimonoid $A$.
\end{prop}
\begin{proof}
We only prove the case of $E= Int_r$ and leave the other parts to the readers.
The morphism $\xi^{-1}$ satisfies the axiom (\ref{Haar_fam_axiom1}) by the following equalitites.
\begin{align}
\nabla_A \circ ( \xi^{-1} \otimes id_A ) &= \nabla_A \circ (\xi^{-1} \otimes \xi^{-1} ) \circ (id_B \otimes \xi)  \\
&= \xi^{-1} \circ \nabla_B \circ (id_B \otimes \xi) . 
\end{align}
Here we use the assumption that $\xi$ is a bimonoid homomorphism.
Similarly, (\ref{Haar_fam_axiom2}) is verified.
Hence, $\xi^{-1} \in Int_r (\xi)$.
\end{proof}

\begin{prop}
\label{201907311054}
We have $E(id_{\mathds{1}}) = End_{\mathcal{C}}( \mathds{1}  )$.
Here, $E$ denotes either $Int_r$, $Int_l$ or $Int$.
\end{prop}
\begin{proof}
We only prove the case of $E= Int_r$ and leave the other parts to the readers.
For $\varphi \in End_{\mathcal{C}} ( \mathds{1}  )$, the morphism $\varphi$ satisfies the axiom (\ref{Haar_fam_axiom1}) with respect to $\xi = id_\mathds{1}$ :
\begin{align}
\nabla_{\mathds{1}} \circ ( \varphi \otimes id_{\mathds{1}} ) &= \mathbf{r}_{\mathds{1}} \circ (\varphi \otimes id_{\mathds{1}} )  \\
&= \varphi \circ \nabla_{\mathds{1}} .  
\end{align}
Here, $\mathbf{r}_\mathds{1}$ is the right unitor of $\mathds{1}$.
Similarly, the axiom (\ref{Haar_fam_axiom2}) is verified.
It implies that $\varphi \in Int_r ( id_{\mathds{1}})$.
\end{proof}

\begin{prop}
\label{201907311055}
The composition of morphisms induces a map,
\begin{align}
E ( \xi^\prime ) \times E (\xi) \to E (\xi^\prime \circ \xi) ; ( \mu^\prime, \mu) \mapsto \mu \circ \mu^\prime . 
\end{align}
Here, $E$ denotes either $Int_r$, $Int_l$ or $Int$.
\end{prop}
\begin{proof}
We only prove the case of $E= Int_r$.
Let $\xi : A \to B$, $\xi^\prime : B \to C$ be bimonoid homomorphisms and $\mu \in Int_r (\xi)$ and $\mu^\prime \in Int_r (\xi^\prime)$.
The composition $\mu\circ\mu^\prime$ satisfies he axiom (\ref{Haar_fam_axiom2}) as follows :
\begin{align}
\nabla_A \circ \left( (\mu \circ \mu^\prime ) \otimes id_A \right) &= \nabla_A \circ \left( \mu \otimes id_A \right) \circ \left( \mu^\prime \otimes id_A \right)  \\
&= \mu \circ \nabla_B \circ (\mu^\prime \otimes \xi)   \\
&= \mu \circ \mu^\prime \circ \nabla_C \circ \left( id_A \otimes (\xi^\prime \circ \xi) \right) . 
\end{align}
It is similarly verified that the composition $\mu\circ\mu^\prime$ satisfies the axiom (\ref{Haar_fam_axiom2}).
Hence, we obtain $\mu \circ \mu^\prime \in Int_r (\xi^\prime\circ \xi)$.
\end{proof}


\subsection{Generator integrals}
\label{201908051612}

In this subsection, we define a notion of {\it a generator integral}.
The terminology is motivated by Proposition \ref{201908021401},  which says that it plays a role of generator of (co)integrals of bimonoids.
It is named after the property that it generates the set of integrals under some conditions (see Theorem \ref{201906281559}).

\begin{Defn}
\rm
\label{201907311452}
Let $\mu$ be an integral along a bimonoid homomorphism $\xi : A \to B$.
The integral $\mu$ is a {\it generator} if the following two diagrams below commute for any $\mu^\prime \in Int_r (\xi) \cup Int_l (\xi)$ :
\begin{equation}
\label{201907311521}
\begin{tikzcd}
B \ar[rrr, "\mu^\prime"] \ar[dr, "\mu^\prime"] & & & A \\
& A \ar[r, "\xi"] & B \ar[ur, "\mu"] &
\end{tikzcd}
\end{equation}

\begin{equation}
\label{201907311522}
\begin{tikzcd}
B \ar[rrr, "\mu^\prime"] \ar[dr, "\mu"] & & & A \\
& A \ar[r, "\xi"] & B \ar[ur, "\mu^\prime"] &
\end{tikzcd}
\end{equation}
\end{Defn}

\begin{prop}
\label{201908021401}
Recall Proposition \ref{201907311133}.
Let $A$ be a bimonoid in a symmetric monoidal category $\mathcal{C}$.
Let $\sigma$ be an integral along the counit $\epsilon_A$.
The integral $\sigma$ is a generator if and only if for any $\sigma^\prime \in \left( Int_r (\epsilon_A) \cup Int_l (\epsilon_A) \right) = \left( Int_r (A) \cup Int_l (A) \right)$
\begin{align}
\sigma^\prime = \left( \epsilon_A \circ \sigma^\prime \right) \cdot \sigma . 
\end{align}
In particular, if an integral $\sigma$ is normalized, then $\sigma$ is a generator.
\end{prop}
\begin{proof}
Let $\sigma$ be a generator.
Then the commutative diagram (\ref{201907311521}) proves the claim.

Let $\sigma^\prime \in Int_l (\epsilon_A) = Int_l ( A)$.
Suppose that $\sigma^\prime = \left( \epsilon_A \circ \sigma^\prime \right) \cdot \sigma $.
Since $\sigma^\prime$ is a left integral of $A$, we have $\left( \epsilon_A \circ \sigma^\prime \right) \cdot \sigma = \nabla_A \circ ( \sigma \otimes \sigma^\prime ) = (\epsilon_A \circ \sigma) \cdot \sigma^\prime$.
Hence, we obtain $\sigma^\prime = (\epsilon_A \circ \sigma) \cdot \sigma^\prime$, which is equivalent with (\ref{201907311522}).
We leave the proof for a right integral $\sigma^\prime$ to the readers.

We prove that if $\sigma$ is normalized, then it is a generator.
Let $\sigma^\prime \in Int_r ( A)$.
Then $\sigma^\prime \ast \sigma = \left( \epsilon_A \circ \sigma \right) \cdot \sigma^\prime = \sigma^\prime$ since $\sigma$ is normalized.
We also have $\sigma^\prime \ast \sigma = \left( \epsilon_A \circ \sigma^\prime \right) \cdot \sigma$ since $\sigma$ is an integral.
Hence, we obtain $\sigma^\prime =\left( \epsilon_A \circ \sigma^\prime \right) \cdot \sigma$.
We leave the proof for $\sigma^\prime \in Int_l ( A)$ to the readers.
It completes the proof.
\end{proof}

We also have a dual statement for cointegrals.

\begin{remark}
There exists a bimonoid $A$ with a generator integral which is not normalized.
For example, finite-dimensional Hopf algebra which is not semi-simple is such an example.
\end{remark}

\begin{prop}
Let $\xi : A \to B$ be a bimonoid isomorphism.
Recall that $\xi^{-1}$ is an integral of $\xi$ by Proposition \ref{201907311053}.
The integral $\xi^{-1}$ is a generator.
\end{prop}
\begin{proof}
It is immediate from definitions.
\end{proof}


\section{Some objects associated with action}
\label{201912060945}

\subsection{Invariant object}
\label{201907220216}

In this subsection, we define a notion of an invariant object of a (co)action.
It is a generalization of the invariant subspace of a group action.

\begin{Defn}
\rm
Let $\mathcal{C}$ be a symmetric monoidal category.
Let $(A, \alpha, X)$ be a left action in $\mathcal{C}$.
A pair $(\alpha \backslash \backslash X , i )$ is an {\it invariant object} of the action $(A, \alpha, X)$ if it satisfies the following axioms :
\begin{itemize}
\item
$\alpha \backslash \backslash X$ is an object of $\mathcal{C}$.
\item
$ i : \alpha \backslash \backslash X \to X$ is a morphism in $\mathcal{C}$.
\item
The diagram commutes where $\tau$ is the trivial action :
\begin{equation}
\begin{tikzcd}
A \otimes X \ar[r, "\alpha"] & X \\
A \otimes (\alpha \backslash \backslash X) \ar[r, "\tau"] \ar[u, "i \otimes id_A"] &  \alpha \backslash \backslash X \ar[u, "i"]
\end{tikzcd}
\end{equation}
\item
It is {\it universal} :
If a morphism $\xi : Z \to X$ satisfies a commutative diagram,
\begin{equation}
\begin{tikzcd}
A \otimes X \ar[r, "\alpha"] & X \\
A \otimes Z \ar[r, "\tau"] \ar[u, "\xi \otimes id_A"] &  Z\ar[u, "\xi"]
\end{tikzcd}
\end{equation}
then there exists a unique morphism $\bar{\xi} : Z \to \alpha \backslash \backslash X$ such that $i \circ \bar{\xi} = \xi$.
\end{itemize}

In an analogous way, we define {\it invariant object} of a left (right) coactions.
\end{Defn}

\subsection{Stabilized object}
\label{201907220217}

In this subsection, we define a notion of a stabilized object of an action (coaction, resp.).
It is enhanced to a functor from the category of (co)actions if the symmetric monoidal category $\mathcal{C}$ has every coequalizer (equalizer, resp.).

\begin{Defn}
\rm
We define a {\it stabilized object of a left action} $(A,\alpha, X)$ in $\mathcal{C}$ by a coequalizer of following morphisms where $\tau_{A,X}$ is the trivial action in Definition \ref{201911301531}.
\begin{equation}
\begin{tikzcd}
A \otimes X \ar[r, shift left, "\alpha"] \ar[r, shift right, "\tau_{A,X}"'] & X
\end{tikzcd}
\end{equation}
We denote it by $\alpha \backslash X$.
Analogously, we define a {\it stabilized object of a right action} $(X, \alpha, A)$ by a coequalizer of $\alpha$ and $\tau_{X,A}$.
We denote it by $X / \alpha$.

We define a {\it stabilized object of a left coaction} $(B , \beta, Y)$ in $\mathcal{C}$ by an equalizer of following morphisms where $\tau^{A,X}$ is the trivial action in Definition \ref{201911301531}.
\begin{equation}
\begin{tikzcd}
Y \ar[r, shift left, "\beta"] \ar[r, shift right, "\tau^{B,Y}"'] & B \otimes Y
\end{tikzcd}
\end{equation}
We denote it by $\beta / Y$.
Analogously, we define a {\it stabilized object of a right coaction} $(Y, \beta, B)$ by an equalizer of $\alpha$ and $\tau^{Y,B}$.
We denote it by $Y \backslash \beta$.
\end{Defn}

\begin{prop}
\label{201911301554}
The assignments of stabilized objects to (co)actions have the following functoriality :
\begin{enumerate}
\item
Suppose that the category $\mathcal{C}$ has any coequalizers.
The assignment $(A,\alpha, X) \mapsto \alpha \backslash X$ gives a symmetric comonoidal functor (SCMF) from $\mathsf{Act}_l (\mathcal{C})$ to $\mathcal{C}$.
Analogouly, the assignment $(X,\alpha, A) \mapsto X / \alpha$ gives a SCMF from $\mathsf{Act}_r (\mathcal{C})$ to $\mathcal{C}$.
\item
Suppose that the category $\mathcal{C}$ has any equalizers.
The assignment $(A,\alpha, X) \mapsto \alpha / X$ gives a symmetric monoidal functor (SMF) from $\mathsf{Coact}_l (\mathcal{C})$ to $\mathcal{C}$.
Analogously, the assignment $(X ,\alpha, A) \mapsto X \backslash \alpha$ gives a SMF from $\mathsf{Coact}_r (\mathcal{C})$ to $\mathcal{C}$.
\end{enumerate}
\end{prop}
\begin{proof}
The functoriality follows from the universality of coequalizers and equalizers.
We only consider the first case.
It is necessary to construct structure maps of a symmetric monoidal functor.
Let us prove the first claim.

Let $(\mathds{1}, \tau, \mathds{1})$ be the unit object of the symmetric monoidal category, $\mathsf{Act}_l(\mathcal{C})$, i.e. the trivial action of the trivial bimonoid $\mathds{1}$ on the object $\mathds{1}$.
Then we have a canonical morphism $\Phi : \tau \backslash \mathds{1} \to \mathds{1}$, in particular an isomprhism.

Let $O = (A,\alpha, X),O^\prime = (A^\prime,\alpha^\prime,X^\prime) $ be left actions in $\mathcal{C}$, i.e. objects of $ \mathsf{Act}_l(\mathcal{C})$.
Denote by $(A\otimes A^\prime ,\beta, X \otimes X^\prime ) = (A,\alpha, X) \otimes (A^\prime,\alpha^\prime,X^\prime) \in \mathsf{Act}_l(\mathcal{C})$.
We construct a morphism $\Psi_{O,O^\prime} : \beta \backslash (X\otimes X^\prime ) \to (\alpha \backslash X) \otimes (\alpha^\prime \backslash X^\prime)$ :
The canonical projections induce a morphism $\xi : X \otimes X^\prime \to (\alpha \backslash X) \otimes (\alpha^\prime \backslash X^\prime)$.
The morphism $\xi$ coequalizes $\beta : (A\otimes A^\prime) \otimes (X\otimes X^\prime) \to X\otimes X^\prime$ and the trivial action of $A\otimes A^\prime$ due to the definitions of $\alpha \backslash X $ and $\alpha^\prime \backslash X^\prime$.
Thus, we obtain a canonical morphism $\Psi_{O,O^\prime} : \beta \backslash ( X\otimes X^\prime ) \to (\alpha \backslash X) \otimes (\alpha^\prime \backslash X^\prime)$.

Due to the universality of coequalizers and the symmetric monoidal structure of $\mathcal{C}$, $\Phi, \Psi_{O,O^\prime}$ give structure morphisms for a symmetric monoidal functor $(A,\alpha,X) \mapsto \alpha \backslash X$.

We leave it to the readers the proof of other part.
\end{proof}


\subsection{Stable monoidal structure}
\label{201907220219}

In this subsection, we define a (co)stability and bistability of the monoidal structure of a symmetric monoidal category.
We assume that $\mathcal{C}$ is a symmetric monoidal category with arbitrary equalizer and coequalizer.

\begin{Defn}
\label{201907230933}
\rm
\label{monstr_stable}
Recall that the assignments of stabilized objects to actions (coactions, resp.) are symmetric comonoidal functors (symmetric monoidal functors, resp.) by Proposition \ref{201911301554}.
The monoidal structure of $\mathcal{C}$ is {\it stable} if the assignments of stabilized objects to actions, $\mathsf{Act}_l(\mathcal{C}) \to \mathcal{C}$ and $\mathsf{Act}_r(\mathcal{C}) \to \mathcal{C}$, are strongly symmetric monoidal functors.
The monoidal structure of $\mathcal{C}$ is {\it costable} if the assignments of stabilized objects to coactions, $\mathsf{Coact}_l(\mathcal{C}) \to \mathcal{C}$ and $\mathsf{Coact}_r(\mathcal{C}) \to \mathcal{C}$, are SSMF's.
The monoidal structure of $\mathcal{C}$ is {\it bistable} if the monoidal structure is stable and costable.
\end{Defn}

\begin{Lemma}
\label{201911301627}
Let $\Lambda,\Lambda^\prime$ be small categories.
Let $F : \Lambda \to \mathcal{C}$, $F^\prime : \Lambda^\prime \to \mathcal{C}$ be functors with colimits $\varinjlim_{\Lambda} F$ and $\varinjlim_{\Lambda^\prime} F^\prime$ respectively.
Suppose that the functor $F(\lambda) \otimes (-)$ preserves small colimits for any object $\lambda$ of $\Lambda$ and so does the functor $(-) \otimes \varinjlim F^\prime$.
Then the exterior tensor product $F \boxtimes F^\prime : \Lambda \times \Lambda^\prime \to \mathcal{C}$ has a colimit $\varinjlim_{\Lambda \times \Lambda^\prime} F \boxtimes F^\prime$, and we have $\varinjlim_{\Lambda \times \Lambda^\prime} F \boxtimes F^\prime \cong \varinjlim_{\Lambda} F \otimes \varinjlim_{\Lambda^\prime} F^\prime$.
\end{Lemma}
\begin{proof}
Let $X$ be an object of $\mathcal{C}$ and $g_{\lambda, \lambda^\prime} : F(\lambda) \otimes F^\prime ( \lambda^\prime ) \to X$ be a family of morphisms for $\lambda \in \Lambda, \lambda^\prime \in \Lambda^\prime$ such that $g_{\lambda_1,\lambda^\prime_1} \circ \left( F(\xi) \otimes F(\xi^\prime) \right) = g_{\lambda_0,\lambda^\prime_0}$ where $\xi : \lambda_0 \to \lambda_1$, $\xi^\prime : \lambda^\prime_0 \to \lambda^\prime_1$ are morphisms in $\Lambda, \Lambda^\prime$ respectively.
By the first assumption, the object $F(\lambda) \otimes \varinjlim F^\prime$ is a colimit of $F(\lambda) \otimes F^\prime (-)$ for arbitrary object $\lambda \in \Lambda$.
We obtain a unique morphism $g_\lambda : F(\lambda) \otimes \varinjlim F^\prime \to X$ such that $g_\lambda \circ (id_{F(\lambda)} \otimes \pi_{\lambda^\prime}) = g_{\lambda, \lambda^\prime}$ for every object $\lambda \in \Lambda$.
By the universality of colimits, the family of morphisms $g_\lambda$ is, in fact, a natural transformation.
By the second assumption, $\varinjlim F \otimes \varinjlim F^\prime$ is a colimit of the functor $F(-) \otimes \varinjlim F^\prime$.
Hence, the family of morphisms $g_\lambda$ for $\lambda \in \Lambda$ induces a unique morphism $g  : \varinjlim F \otimes \varinjlim F^\prime \to X$ such that $g \circ (\pi_\lambda \otimes id_{\varinjlim F^\prime}) = g_\lambda$.
Above all, for objects $\lambda \in \Lambda, \lambda^\prime \in \Lambda^\prime$, we have $g \circ ( \pi_\lambda \otimes \pi_{\lambda^\prime} ) = g \circ (\pi_\lambda \otimes id_{\varinjlim F^\prime}) \circ (id_{F(\lambda)} \otimes \pi_{\lambda^\prime}) = g_\lambda \circ (id_{F(\lambda)} \otimes \pi_{\lambda^\prime}) = g_{\lambda, \lambda^\prime}$.

We prove that such a morphism $g$ that $g \circ ( \pi_\lambda \otimes \pi_{\lambda^\prime} ) = g_{\lambda, \lambda^\prime}$ is unique.
Let $g^\prime : \varinjlim F \otimes \varinjlim F^\prime \to X$ be a morphism such that $g^\prime \circ ( \pi_\lambda \otimes \pi_{\lambda^\prime} ) = g_{\lambda, \lambda^\prime}$.
Denote by $h = g \circ (\pi_\lambda \otimes id_{\varinjlim F^\prime})$ and  $h^\prime = g^\prime \circ (\pi_\lambda \otimes id_{\varinjlim F^\prime})$.
Then we have $h^\prime \circ (id_{F(\lambda)} \otimes \pi_{\lambda^\prime}) = g_{\lambda, \lambda^\prime} = h \circ (id_{F(\lambda)} \otimes \pi_{\lambda^\prime})$ by definitions.
Since $F(\lambda) \otimes \varinjlim F^\prime$ is a colimit of the functor $F(\lambda) \otimes F^\prime (-)$ by the first assumption, we see that $h^\prime = h$.
Equivalently, we have $g \circ (\pi_\lambda \otimes id_{\varinjlim F^\prime}) = g^\prime \circ (\pi_\lambda \otimes id_{\varinjlim F^\prime})$.
Since $\varinjlim F \otimes \varinjlim F^\prime$ is a colimit of the functor $F (-) \otimes \varinjlim F^\prime$ by the second assumption, we see that $g = g^\prime$ by the universality.
It completes the proof.
\end{proof}

\begin{prop}
\label{comp_monoidal_stab}
Suppose that the functor $Z \otimes (-)$ preserves coequalizers (equalizers resp.) for arbitrary object $Z \in \mathcal{C}$.
Then the monoidal structure of $\mathcal{C}$ is stable (costable, resp.).
\end{prop}
\begin{proof}
Note that since $\mathcal{C}$ is a symmetric monoidal category, the functor $(-) \otimes Z$ preserves coequalizers (equalizers resp.) for arbitrary object $Z \in \mathcal{C}$ by the assumption.
We prove the stability and leave the proof o the costability to the readers.

Let $(A,\alpha, X)$, $(B, \beta , Y)$ be left actions in $\mathcal{C}$.
Denote by $\alpha \backslash X , \beta \backslash Y$ their stabilized objects as before.
By the assumption, we can apply  Lemma \ref{201911301627}.
By Lemma \ref{201911301627}, $( \alpha \backslash X \otimes \beta \backslash Y)$ is a coequalizer of morphisms $\alpha \tilde{\otimes} \beta$, $\alpha \tilde{\otimes} \tau_{B}$, $\tau_{A} \tilde{\otimes} \beta$, $\tau_A \tilde{\otimes} \tau_B$.
Here, $\tilde{\otimes}$ is defined in Definition \ref{201911301531}.
It suffices to show that a coequalizer of $\alpha \tilde{\otimes} \beta$, $\alpha \tilde{\otimes} \tau_{B}$, $\tau_{A} \tilde{\otimes} \beta$, $\tau_A \tilde{\otimes} \tau_B$ coincides with the stabilized object $(\alpha \tilde{\otimes} \beta) \backslash ( X \otimes Y)$, i.e. a coequalizer of $\alpha \tilde{\otimes} \beta$, $\tau_A \tilde{\otimes} \tau_B$.

Let $\pi : X \otimes Y \to (\alpha \tilde{\otimes} \beta ) \backslash ( X\otimes Y )$ be the canonical projection.
The unit axiom of the action $\beta$ induces the following commutative diagram :
\begin{equation}
\begin{tikzcd}
A \otimes B \otimes X \otimes Y \ar[rr, "\alpha\tilde{\otimes}\tau_B"] \ar[dr, "id_{A} \otimes (\eta_B\circ\epsilon_B) \otimes id_{X\otimes Y}"'] & & X \otimes Y \\
& A \otimes B \otimes X \otimes Y \ar[ur, "\alpha \tilde{\otimes} \beta"'] &
\end{tikzcd}
\end{equation}
Hence, we have $\pi \circ (\alpha \tilde{\otimes} \tau_B) = \pi \circ (\alpha \tilde{\otimes} \beta) \circ (id_{A} \otimes (\eta_B\circ\epsilon_B) \otimes id_{X\otimes Y}) = \pi \circ (\tau_A \tilde{\otimes} \tau_B) \circ (id_{A} \otimes (\eta_B\circ\epsilon_B) \otimes id_{X\otimes Y}) = \pi \circ (\tau_A \tilde{\otimes} \tau_B)$.
We obtain $\pi \circ (\alpha \tilde{\otimes} \tau_B ) = \pi \circ (\tau_A \tilde{\otimes} \tau_B)$.
Likewise, we have  $\pi \circ (\tau_A \tilde{\otimes} \beta) = \pi \circ (\tau_A \tilde{\otimes} \tau_B)$.

Let $g  : X \otimes Y \to Z$ be a morphism which coequalizes $\alpha \tilde{\otimes} \beta$, $\alpha \tilde{\otimes} \tau_{B}$, $\tau_{A} \tilde{\otimes} \beta$, $\tau_A \tilde{\otimes} \tau_B$.
Since the morphism $g$ coequalizes $\alpha \tilde{\otimes} \beta$, $\tau_A \tilde{\otimes} \tau_B$, there exists a unique morphism $g^\prime : ( \alpha \tilde{\otimes} \beta ) \backslash ( X \otimes Y ) \to Z$ such that $g^\prime \circ \pi = g$.
Above all, $( \alpha \tilde{\otimes} \beta ) \backslash ( X \otimes Y )$ is a coequalizer of $\alpha \tilde{\otimes} \beta$, $\alpha \tilde{\otimes} \tau_{B}$, $\tau_{A} \tilde{\otimes} \beta$, $\tau_A \tilde{\otimes} \tau_B$.
\end{proof}

\begin{Example}
\label{201912022052}
Consider the symmetric monoidal category, $\mathsf{Vec}^{\otimes}_{{k}}$, the category of vector spaces over ${k}$ and linear homomorphisms.
Note that a coequalizer (an equalizer, resp.) of two morphisms in the category $\mathsf{Vec}_{{k}}$ is obtained via a cokernel (a kernel, resp.) of their difference morphism.
A functor $V \otimes (-)$ preserves coequazliers and equazliers since it is an exact functor for any linear space $V$.
Hence, by Proposition \ref{comp_monoidal_stab}, the monoidal structure of the symmetric monoidal category, $\mathsf{Vec}^{\otimes}_{{k}}$, is bistable.
\end{Example}


\section{Normal homomorphism}
\label{201908051552}

In this section, we define a notion of {\it normality}, {\it conormality} and {\it binormality} of bimonoid homomorphisms.
We prove that every homomorphism between bicommutative Hopf monoids is binormal under some assumptions on the symmetric monoidal category $\mathcal{C}$.

\begin{Defn}
\label{201908040955}
\rm
Let $\mathcal{D}$ be a category with a zero object, i.e. an initial object which is simultaneously a terminal object.
Let $A,B$ be objects of $\mathcal{D}$ and $\xi : A \to B$ be a morphism in $\mathcal{D}$.
A {\it cokernel} of $\xi$ is given by a pair $(Cok (\xi) , cok (\xi))$ of an object $Cok(\xi)$ and a morphism $cok (\xi ) : B \to Cok (\xi )$, which gives a coequalizer of $\xi : A \to B$ and $0 : A \to B$ in $\mathcal{D}$.

A {\it kernel} of $\xi$ is given by a pair $(Ker (\xi) , ker (\xi))$ of an object $Ker(\xi)$ and a morphism $ker (\xi ) : Ker (\xi ) \to A$, which gives an equalizer of $\xi : A \to B$ and $0 : A \to B$ in $\mathcal{D}$.
\end{Defn}

\begin{Defn}
\label{201912021024}
\rm
Let $A,B$ be bimonoids in a symmetric monoidal category $\mathcal{C}$ and $\xi : A \to B$ be a bimonoid homomorphism.
We define a left action $(A, \alpha^{\to}_\xi , B)$ and a right action $(B, \alpha^{\leftarrow}_\xi , A)$ by the following compositions :
\begin{align}
\alpha^{\to}_\xi : A \otimes B \stackrel{\xi\otimes id_B}{\to} B \otimes B \stackrel{\nabla_B}{\to} B , \\
\alpha^{\leftarrow}_\xi : B \otimes A \stackrel{id_B \otimes \xi}{\to} B \otimes B \stackrel{\nabla_B}{\to} B .
\end{align}
We define a left coaction $(A, \beta^{\to}_\xi , B)$ and a right coaction $(B, \beta^{\leftarrow}_\xi , A)$ by the following compositions : 
\begin{align}
\beta^{\to}_\xi : A \stackrel{\Delta_A}{\to} A \otimes A \stackrel{\xi\otimes id_A}{\to} B \otimes A , \\
\beta^{\leftarrow}_\xi : A \stackrel{\Delta_A}{\to} A \otimes A \stackrel{id_A\otimes \xi}{\to} A \otimes B  .
\end{align}
\end{Defn}

\begin{Defn}
\label{201908040928}
\rm
Let $A,B$ be bimonoids in a symmetric monoidal category $\mathcal{C}$.
A bimonoid homomorphism $\xi : A \to B$ is {\it normal} if there exists a bimonoid structure on the stabilized objects $\alpha^{\to}_\xi \backslash B$, $B / \alpha^{\leftarrow}_\xi$ such that the canonical morphisms $\pi : B\to \alpha^{\to}_\xi \backslash B$, $\tilde{\pi} : B \to B/\alpha^{\leftarrow}_\xi$ are bimonoid homomorphisms and the pairs $(\alpha^{\to}_\xi \backslash B, \pi)$, $(B / \alpha^{\leftarrow}_\xi , \tilde{\pi})$ give cokernels of $\xi$ in $\mathsf{Bimon}(\mathcal{C})$.

A {\it conormal} bimonoid homomorphism is defined in a dual way by using the coactions $\beta^{\leftarrow}_\xi, \beta^{\to}_\xi$ instead of $\alpha^{\to}_\xi, \alpha^{\leftarrow}_\xi$.
A bimonoid homomrphism $\xi : A \to B$ is {\it binormal} if it is normal and conormal in $\mathsf{Bimon}(\mathcal{C})$.
\end{Defn}

\begin{remark}
We use the terminology {\it normal} due to the following reason.
If $\mathcal{C} = \mathsf{Sets}^\times$, then a Hopf monoid in that symmetric monoidal category is given by a group.
For a group $H$ and its subgroup $G$, one can determine a set $H / G$ which is a candidate of a cokernel of the inclusion.
The set $H / G$ plays a role of cokernel group if and only if the image $G$ is a normal subgroup of $H$.
In this example, the {\it normality} defined in this paper means that the set $H / G$ is a cokernel group of the inclusion $G \to H$.
\end{remark}

\begin{remark}
We remark that our notion is implied by the Milnor-Moore's definition if $\mathcal{C} = \mathsf{Vec}^\otimes_{k}$.
Milnor and Moore defined the notion of normality of morphisms of augmented algebras over a ring and normality of morphisms of augmented coalgebras over a ring (Definition 3.3, 3.5 \cite{Mil}).
They are defined by using the additive structure of the category $\mathsf{Vec}_{k}$.
We introduce a weaker notion of normality and conormality of bimonoid homomorphisms without assuming an additive category structure on $\mathcal{C}$.
\end{remark}

\begin{prop}
\label{201907021116}
Let $A$ be a bimonoid.
The identity homomorphism $id_A : A \to A$ is binormal.
\end{prop}
\begin{proof}
We prove that the identity homomorphism $id_A$ is normal.
The counit $\epsilon_A : A \to \mathds{1}$ on $A$ induces gives a coequalizer of the regular action $\alpha^{\to}_{id_A} : A \otimes A \to A$ and the trivial action $\tau : A \otimes A \to A$.
In particular, we have a natural isomorphism $\alpha^{\to}_{id_A} \backslash A \cong \mathds{1}$.
We give a bimonoid structure on $\alpha^{\to}_{id_A}$ by the isomorphism.
Moreover the counit $\epsilon_A : A \to \mathds{1}$ is obviously a cokernel of the identity homomorphism $id_A$ in the category of bimonoids $\mathsf{Bimon}(\mathcal{C})$.
Thus, the identity homomorphsim $id_A$ is normal.
In a dual way, the identity homomorphsim $id_A$ is conormal, so that binormal.
\end{proof}

\begin{prop}
\label{201911302012}
Let $A,B$ be Hopf monoids in a symmetric monoidal category $\mathcal{C}$. 
Let $\xi : A \to B$ be a bimonoid homomorphism.
If the homomorphism $\xi$ is normal, then a cokernel $(Cok (\xi) , cok (\xi) )$ in the category of bimoniods $\mathsf{Bimon}(\mathcal{C})$ is a cokernel in the category of Hopf monoids $\mathsf{Hopf} ( \mathcal{C})$.
\end{prop}
\begin{proof}
Since $cok( \xi ) \circ S_B \circ \xi  = cok ( \xi ) \circ \xi \circ S_A$ is trivial, the anti-homomorphism $cok( \xi ) \circ S_B$ induces an anti-homomorphism $S : Cok ( \xi ) \to Cok ( \xi )$ such that $S \circ cok ( \xi ) = cok ( \xi ) \circ S_B$.
We claim that $S$ gives an antipode on the bimonoid $C= Cok ( \xi )$.
It suffices to prove that $\nabla_C \circ ( S \otimes id_C ) \circ \Delta_C = \eta_C \circ \epsilon_C = \nabla_C \circ (  id_C \otimes S ) \circ \Delta_C$.
Since $(\alpha^{\to}_\xi \backslash B, \pi)$, $(B / \alpha^{\leftarrow}_\xi , \tilde{\pi})$ give cokernels, the canonical morphism $cok ( \xi )$ is an epimorphism in $\mathcal{C}$ by the universality of stabilized objects.
Hence, it suffices to prove that $\nabla_C \circ ( S \otimes id_C ) \circ \Delta_C \circ cok ( \xi ) = \eta_C \circ \epsilon_C \circ cok ( \xi )  = \nabla_C \circ (  id_C \otimes S ) \circ \Delta_C \circ cok ( \xi ) $.
We prove the first equation by using the fact that $cok ( \xi ) : B \to Cok ( \xi ) =C$ is a bimonoid homomorphism.
\begin{align}
\nabla_C \circ ( S \otimes id_C ) \circ \Delta_C \circ cok ( \xi )
&=
\nabla_C \circ ( S \otimes id_C ) \circ ( cok ( \xi ) \otimes cok ( \xi )) \circ \Delta_B , \\
&=
\nabla_C \circ ( (S \circ cok ( \xi )) \otimes cok ( \xi ) ) \circ \Delta_B , \\
&=
\nabla_C \circ ( (cok ( \xi ) \circ S_B ) \otimes cok ( \xi ) ) \circ \Delta_B , \\
&=
\nabla_C \circ ( cok ( \xi ) \otimes cok ( \xi ) ) \circ ( S_B \otimes id_B ) \circ \Delta_B , \\
&=
cok ( \xi ) \circ \nabla_B \circ ( S_B \otimes id_B ) \circ \Delta_B , \\
&=
cok ( \xi ) \circ \eta_B \circ \epsilon_B , \\
&=
\eta_C \circ \epsilon_C \circ cok ( \xi ) .
\end{align}
The second equation is proved similarly.
It completes the proof.
\end{proof}

\begin{prop}
\label{stab_gives_cokernel}
Suppose that the monoidal structure of $\mathcal{C}$ is stable (costable, resp.).
Then every bimonoid homomorphism between bicommutative bimonoids is normal (conormal, resp.) and its cokernel (kernel, resp.) is a bicommutative bimonoid.
In particular, if the monoidal structure of $\mathcal{C}$ is bistable, then every bimonoid homomorphism between bicommutative bimonoids is binormal.
\end{prop}
\begin{proof}
We prove that if the monoidal structure of $\mathcal{C}$ is stable, then every bimonoid homomorphism between bicommutative bimonoids is normal and its cokernel is a bicommutative bimonoid.
Let $A,B$ be bicommutative bimonoids in a symmetric monoidal category $\mathcal{C}$ and $\xi : A \to B$ be a bimonoid homomorphism.
Note that the left action $(A, \alpha^{\to}_\xi , B)$ has a natural bicommutative bimonoid structure in the symmetric monoidal category $\mathsf{Act}_{l} (\mathcal{C})$, the category of left actions in $\mathcal{C}$.
The symmetric monoidal category structure on $\mathsf{Act}_{l} (\mathcal{C})$ is described in Definition \ref{201911301531}.
In fact, it is due to the commutativity of $B$ :
We explain the monoid structure of $(A, \alpha^{\to}_\xi , B)$ here.
Since $B$ is a bicommutative bimonoid, $\nabla_B : B \otimes B \to B$ is a bimonoid homomorphism.
In particular, $\nabla_B$ is compatible with the action $\alpha^{\to}_\xi$, i.e. the following diagram commutes.
\begin{equation}
\begin{tikzcd}
(A \otimes A) \otimes (B \otimes B) \ar[r, "\alpha^{\to}_\xi \tilde{\otimes} \alpha^{\to}_\xi"] \ar[d, "\nabla_A \otimes \nabla_B"] & B \otimes B \ar[d, "\nabla_B"] \\
A \otimes B \ar[r, "\alpha^{\to}_\xi"] & B
\end{tikzcd}
\end{equation}
Since $\eta_B : \mathds{1} \to B$ is a bimonoid homomorphism, the following diagram commutes.
\begin{equation}
\begin{tikzcd}
\mathds{1} \otimes \mathds{1} \ar[r, "\cong"] \ar[d, "\eta_A \otimes \eta_B"] & \mathds{1} \ar[d, "\eta_B"] \\
A \otimes B \ar[r, "\alpha^{\to}_\xi"] & B
\end{tikzcd}
\end{equation}
Hence, they induce a monoid structure on $(A, \alpha^{\to}_\xi , B)$ in the symmetric monoidal category $\mathsf{Act}_{l} (\mathcal{C})$.
Likewise, $(A, \alpha^{\to}_\xi , B)$ has a comonoid structure in $\mathsf{Act}_l ( \mathcal{C})$ :
The comultiplications on $A,B$ induces a comultiplication on $(A, \alpha^{\to}_\xi , B)$ due to following diagram commutes.
\begin{equation}
\begin{tikzcd}
(A \otimes A) \otimes (B \otimes B) \ar[r, "\alpha^{\to}_\xi \tilde{\otimes} \alpha^{\to}_\xi"]  & B \otimes B  \\
A \otimes B \ar[r, "\alpha^{\to}_\xi"] \ar[u, "\Delta_A \otimes \Delta_B"] & B \ar[u, "\Delta_B"]
\end{tikzcd}
\end{equation}
In fact, we do not need any commutativity or cocommutativity of $A,B$ to prove the commutativity of the diagram.
The counits on $A,B$ induce a counit on $(A, \alpha^{\to}_\xi , B)$ due to the following commutativity diagram.
\begin{equation}
\begin{tikzcd}
\mathds{1} \otimes \mathds{1} \ar[r, "\cong"]  & \mathds{1}  \\
A \otimes B \ar[u, "\epsilon_A \otimes \epsilon_B"] \ar[r, "\alpha^{\to}_\xi"] & B \ar[u, "\epsilon_B"]
\end{tikzcd}
\end{equation}
Since the morphisms $\Delta_A,\nabla_A, \epsilon_A,\eta_A$ and the morphisms $\Delta_B,\nabla_B, \epsilon_B,\eta_B$ give bicommutative bimonoid structure on $A,B$ respectively, the above monoid structure and comonoid structure on $(A, \alpha^{\to}_\xi , B)$ give a bicommutative bimonoid structure on $(A, \alpha^{\to}_\xi , B)$.

Since the monoidal structure of $\mathcal{C}$ is stable by the assumption, the assignment of stabilized objects to actions is a strongly symmetric monoidal functor by definition.
The bicommutative bimonoid structure on $(A, \alpha^{\to}_\xi , B)$ is inherited to its stabilized object $\alpha^{\to}_\xi \backslash B$.
We consider $\alpha^{\to}_\xi \backslash B$ as a bicommutative bimonoid by the inherited structure.

The canonical morphism $\pi : B \to \alpha^{\to}_\xi \backslash B$ is a bimonoid homomorphism with respect to the bimonoid structure on $\alpha^{\to}_\xi \backslash B$ described above.
In fact, the commutative diagram (\ref{201908040659}) induces a bimonoid homomorphism $(\mathds{1} , \alpha^{\to}_{\eta_B} , B) \to (A , \alpha^{\to}_\xi , B)$ between bicommutative bimonoids in the symmetric monoidal category $\mathsf{Act}_l ( \mathcal{C} )$.
\begin{equation}
\label{201908040659}
\begin{tikzcd}
\mathds{1} \ar[r, "\eta_B"] \ar[d, "\eta_A"] & B \ar[d, "id_B"] \\
A \ar[r, "\xi"] & B
\end{tikzcd}
\end{equation}
By the stability of the monoidal structure of $\mathcal{C}$ again, we obtain a bimonoid homomorphism, 
\begin{align}
B \cong \alpha^{\to}_{\eta_B} \backslash B \to \alpha^{\to}_\xi \backslash B . 
\end{align}
It coincides with the canonical projection $\pi : B  \to \alpha^{\to}_\xi \backslash B$ by definitions.

All that remain is to show that the pair $(\alpha^{\to}_\xi \backslash B , \pi)$ is a cokernel of the bimonoid homomorphism $\xi$ in $\mathsf{Bimon}(\mathcal{C})$ in the sense of Definition \ref{201908040955}.
Let $C$ be another bimonoid and $\varphi : B \to C$ be a bimonoid homomorphism such that $\varphi \circ \xi = \eta_C \circ \epsilon_A$.
It coequazlies the action $\alpha^{\to}_\xi : A \otimes B \to B$ and the trivial action $\tau_{A,B} : A \otimes B \to B$ so that it induces a unique morphism $\bar{\varphi} : \alpha^{\to}_\xi \backslash B \to C$ such that $\bar{\varphi}\circ\pi = \varphi$.
We prove that $\bar{\varphi}$ is a bimonoid homomorphism.
Note that the counit $\epsilon_A : A \to \mathds{1}$ and the homomorphism $\varphi : B \to C$ induces a bimonoid homomorphism $(A , \alpha^{\to}_\xi , B) \to ( \mathds{1} , \alpha^{\to}_{\eta_C} , C )$.
By the stability of the monoidal structure of $\mathcal{C}$ again, it induces a bimonoid homomorphism $\alpha^{\to}_\xi \backslash B \to \alpha^{\to}_{\eta_C} \backslash C \cong C$ which coincides with $\bar{\varphi}$.
It completes the proof.
\end{proof}

\begin{Corollary}
\label{201908040952}
Suppose that the monoidal structure of $\mathcal{C}$ is stable (costable, resp.).
Let $A,B$ be bicommutative Hopf monoids and $\xi : A \to B$ be a bimonoid homomorphism.
Then a cokernel (kernel, resp.) of $\xi$ in $\mathsf{Bimon} (\mathcal{C})$ is a cokernel (kernel, resp.) of $\xi$ in $\mathsf{Hopf}^\mathsf{bc}(\mathcal{C})$.
\end{Corollary}
\begin{proof}
Suppose that the monoidal structure of $\mathcal{C}$ is stable.
Let $A,B$ be bicommutative Hopf monoids and $\xi : A \to B$ be a bimonoid homomorphism.
By Proposition \ref{stab_gives_cokernel}, the homomorphism $\xi$ is normal and its cokernel is a bicommutative bimonoid.
By Proposition \ref{201911302012}, the cokernel of $\xi$ is a bicommutative Hopf monoid.
\end{proof}

\section{Small bimonoid and integral}
\label{201912060950}

In this section, we introduce a notion of {\it (co,bi)small bimonoids}.
We study the relationship between existence of normalized (co)integrals and (co)smallness of bimonoids.

\begin{Defn}
\label{201911301753}
\rm
Let $\mathcal{C}$ be a symmetric monoidal category.
Let $(A,\alpha,X)$ be a left action in the symmetric monoidal category $\mathcal{C}$.
Recall the invariant object $\alpha \backslash\backslash  X$ and the stabilized object $\alpha \backslash X$ of the left action $(A,\alpha,X)$.
We define a morphism $_\alpha\gamma: \alpha \backslash\backslash  X \to \alpha \backslash X$ in $\mathcal{C}$ by composing the canonical morphisms $X \to \alpha \backslash X$ and $\alpha \backslash\backslash X \to X$.
Likewise, we define $\gamma_\alpha : X // \alpha \to X / \alpha$ for a right action $(X, \alpha, A)$, $^\beta\gamma : \beta / Y \to \beta // Y$ for a left coaction $(B , \beta, Y)$, $\gamma^\beta : Y \backslash \beta \to Y \backslash \backslash \beta$ for a right coaction $(Y, \beta, B)$.
\end{Defn}

\begin{Defn}
\label{202002271041}
\rm
Let $\mathcal{C}$ be a symmetric monoidal category.
Recall Definition \ref{201912021024}.
A bimonoid $A$ in the symmetric monoidal category $\mathcal{C}$ is {\it small} if 
\begin{itemize}
\item
For every left action $(A, \alpha , X)$, an invariant object $\alpha \backslash \backslash X$ and a stabilized object $\alpha \backslash X$ exist.
Furthermore, the canonical morphism $_\alpha\gamma : \alpha \backslash \backslash X \to \alpha \backslash X$ is an isomorphism.
\item
For every right action $(X, \alpha , A)$, an invariant object $X // \alpha$ and a stabilized object $X / \alpha$ exist.
Furthermore, the canonical morphism $\gamma_\alpha : X // \alpha \to X / \alpha$ is an isomorphism.
\end{itemize}

A bimonoid $A$ in the symmetric monoidal category $\mathcal{C}$ is {\it cosmall} if
\begin{itemize}
\item
For every left coaction $(B, \beta , Y)$, an invariant object $\beta // Y$ and a stabilized object $\beta / Y$ exist.
Furthermore, the canonical morphism $^\beta\gamma : \beta \backslash  Y \to \beta \backslash \backslash Y$ is an isomorphism.
\item
For every right coaction $(Y , \beta , B)$, an invariant object $Y \backslash \backslash \beta$ and a stabilized object $Y \backslash \beta$ exist.
Furthermore, the canonical morphism $\gamma^\beta : Y / \beta \to Y // \beta$ is an isomorphism.
\end{itemize}
A bimonoid $A$ is {\it bismall} if the bimonoid $A$ is small and cosmall.

We use subscript `bs' to denote `bismall'.
For example, $\mathsf{Hopf}^\mathsf{bs} ( \mathcal{C})$ is a full subcategory of $\mathsf{Hopf}(\mathcal{C})$ formed by bismall Hopf monoids.
\end{Defn}

\begin{remark}
In general, the morphism $_\alpha\gamma : \alpha \backslash \backslash X \to \alpha \backslash X$ (also, $^\beta\gamma , \gamma_\alpha, \gamma^\beta$) in Definition \ref{201911301753} is not an isomorphism.
We give three examples as follows.
\end{remark}

\begin{Example}
Let $(A, \alpha , X)$ be a left action where $A=X = {k} G$ and $\alpha$ is the multiplication where $G$ is a finite group.
There exists an invariant object $\alpha \backslash \backslash {k} G$ and a stabilized object $\alpha \backslash {k} G$ given by
\begin{align}
\alpha \backslash \backslash {k} G &= \{ \lambda \sum_{g\in G} g ~;~ \lambda \in{k} \}  \\
\alpha \backslash {k} G &= {k}G / \left( g \sim e \right)  
\end{align}
Here, $e\in G$ denotes the unit of $G$ and ${k}G / \left( g \sim e \right)$ means the quotient space of ${k}G$ by the given relation.
Then if the characteristic of $G$ divides the order $|G|$ we see that the morphism $_\alpha\gamma$ is zero while $\alpha \backslash \backslash {k} G$, $\alpha \backslash {k} G$ are 1-dimensional.
\end{Example}

\begin{Defn}
\label{201907130830}
\rm
Let $\mathcal{C}$ be a category.
A morphism $p : X \to X$ is an {\it idempotent} if $p \circ p = p$.
A {\it retract} of an idempotent $p$ is given by $(X^p , \iota  , \pi )$ where $\iota : X^p \to X$, $\pi  : X \to X^p$ are morphisms in $\mathcal{C}$ such that $\pi \circ \iota = id_{X^p}$ and $\iota \circ \pi = p$.
If an idempotent $p$ has a retract, then $p$ is called a {\it split idempotent}.
\end{Defn}

\begin{prop}
\label{201912021623}
Let $\mathcal{C}$ be a category and $p : X \to X$ be an idempotent.
Suppose that there exists an equalizer of the identity $id_X$ and $p$ and a coequalizer of the identity $id_X$ and $p$.
Then the idempotent $p$ is a split idempotent.
\end{prop}
\begin{proof}
Denote by $e : E \to X$ an equalizer of the identity $id_X$ and the morphism $p : X \to X$.
Denote by $c : X \to C$ a coequalizer of the identity $id_X$ and the morphism $p : X \to X$.
We claim that $c \circ e : K \to E$ is an isomorphism and $(E , e , (c \circ e)^{-1} \circ c)$ is a retract of the idempotent $p$.

Note that the morphism $p$ equalizes the identity $id_X$ and the morphism $p$ due to $p \circ p = p$.
The morphism $p$ induces a unique morphism $p^\prime : X \to E$ such that $e \circ p^\prime = p$.
Note that the morphism $p^\prime$ coequalizes the identity $id_X$ and the morphism $p$ due to $p^\prime \circ p = p^\prime$.
The morphism $p^\prime$ induces a unique morphism $p^{\prime\prime} : C \to E$ such that $p^{\prime\prime} \circ c = p^\prime$.
Then $p^{\prime\prime}$ is an inverse of the composition $c \circ e$ so that $c \circ e$ is an isomorphism.

We prove that $(E , e , (c \circ e)^{-1} \circ c)$ is a retract of the idempotent $p$.
It follows from $\left( (c \circ e)^{-1} \circ c \right) \circ e = id_K$ and $e \circ \left( (c \circ e)^{-1} \circ c \right) = p$.
The latter one follows from the above discussion that $(c \circ e)^{-1} = p^{\prime\prime}$ and $e \circ p^{\prime\prime} \circ c = e \circ p^\prime = p$.
\end{proof}

\begin{prop}
\label{201912021539}
Let $(A, \alpha , X)$ be a left action in a symmetric monoidal category $\mathcal{C}$ with an invariant object $\alpha \backslash \backslash X$ and a stabilized object $\alpha \backslash X$.
Suppose that the morphism $_\alpha \gamma : \alpha \backslash \backslash X \to \alpha \backslash X$ is an isomorphism.
Then the endomorphism $p : X \to X$ defined by following composition is a split idempotent.
\begin{align}
_\alpha p = \left( X \stackrel{\pi}{\to} \alpha \backslash X \stackrel{_\alpha\gamma^{-1}}{\to} \alpha \backslash \backslash X \stackrel{\iota}{\to} X \right) .
\end{align}
Here, $\iota,\pi$ are the canonical morphisms.
\end{prop}
\begin{proof}
We prove that $p$ is an idempotent on $X$.
It follows from $p \circ p = \iota \circ _\alpha\gamma^{-1} \circ \pi \circ \iota \circ _\alpha\gamma^{-1} \circ \pi = \iota \circ _\alpha\gamma^{-1} \circ _\alpha\gamma \circ _\alpha\gamma^{-1} \circ \pi = \iota \circ _\alpha\gamma^{-1} \circ \pi = p$.

We prove that $(\alpha \backslash X , \iota \circ _\alpha\gamma^{-1} , \pi)$ give a retract of the idempotent $p$.
By definition, we have $\iota \circ _\alpha\gamma^{-1} \circ \pi = p$.
Moreover, we have $\pi \circ \iota \circ _\alpha\gamma^{-1} = _\alpha\gamma \circ _\alpha\gamma^{-1} = id_{\alpha \backslash X}$.
\end{proof}

\begin{Lemma}
\label{small_has_nor_integral}
Let $A$ be a bimonoid in a symmetric monoidal category $\mathcal{C}$.
Suppose that for the regular left action $(A, \alpha^{\to}_{id_A} , A)$, an invariant object $\alpha^{\to}_{id_A} \backslash \backslash A$ and a stabilized object $\alpha^{\to}_{id_A} \backslash A$ exist and the canonical morphism $_{\alpha^{\to}_{id_A}}\gamma : \alpha^{\to}_{id_A} \backslash \backslash A \to \alpha^{\to}_{id_A} \backslash A$ is an isomorphism.
Then the bimonoid $A$ has a normalized left integral.
\end{Lemma}
\begin{proof}
Let $A$ be a bimonoid.
Suppose that the bimonoid $A$ is small.
Consider a left action $(A, \alpha , A)$ in $\mathcal{C}$ where $\alpha = \alpha^{\to}_{id_A} = \nabla_A : A \otimes A \to A$ is the regular left action.
Since $A$ is small, the invariant object $ \alpha \backslash\backslash A$ and the stabilized object $\alpha \backslash A$ exist and the morphism $_\alpha\gamma : \alpha \backslash\backslash A \to \alpha \backslash A$ is an isomorphism.
Let $p : A \to A$ be a composition of $A \stackrel{\pi}{\to} \alpha \backslash A \stackrel{_\alpha\gamma^{-1}}{\to} \alpha \backslash \backslash A \stackrel{\iota}{\to} A$ where $\pi$, $\iota$ are canonical morphisms.
We prove that $\sigma = p \circ \eta_A : \mathds{1} \to A$ is a normalized right integral.

We claim that $\epsilon_A \circ p = \epsilon$.
Then $\epsilon_A \circ \sigma =\epsilon_A \circ \eta_A = id_{\mathds{1}}$ which is the axiom (\ref{Haar_axiom1}) :
Note that the canonical morphism $\pi : A \to \alpha \backslash A$ coequalizes the regular left action $\alpha$ and the trivial left action.
The counit morphism $\epsilon_A$ induces a unique morphism $\bar{\epsilon_A} : \alpha \backslash A \to \mathds{1}$ such that $\bar{\epsilon_A} \circ \pi = \epsilon_A$.
We obtain following commutative diagram so that $\epsilon_A \circ p = \epsilon$.
\begin{equation}
\begin{tikzcd}
A \ar[rrrr, "p", bend left] \ar[r, "\pi"] \ar[ddrr, "\epsilon_A"'] & \alpha \backslash A \ar[rr, "_\alpha\gamma^{-1}"] \ar[ddr, "\bar{\epsilon_A}"'] & & \alpha \backslash\backslash A \ar[r, "\iota"] \ar[dl, "\iota"] & A \ar[ddll, "\epsilon_A"] \\
& & A \ar[d, "\epsilon_A"] \ar[ul, "\pi"] & & \\
& & \mathds{1} & & 
\end{tikzcd}
\end{equation}

We claim that $\nabla_A \circ (id_A \otimes p) = \mathbf{r}_A \circ (\epsilon_A \otimes p ) : A \otimes A \to A$.
Then by composing $id_A \otimes \eta_A : A \otimes \mathds{1}   \to A \otimes A$ we see that $\sigma = p \circ \eta_A$ satisfies the axiom (\ref{Haar_axiom3}) :
In fact, we have $\nabla_A \circ ( id_A \otimes \iota) = \epsilon_A \otimes \iota : A \otimes (\alpha \backslash\backslash A) \to A$ by definition of $\iota : \alpha \backslash\backslash A \to A$.
Thus, we have $\nabla_A \circ (id_A \otimes p) = \nabla_A \circ (id_A \otimes \iota) \circ (id_A \otimes (_\alpha\gamma^{-1} \circ \pi)) = (\epsilon_A \otimes \iota)  \circ (id_A \otimes (_\alpha\gamma^{-1} \circ \pi)) = \mathbf{r}_A \otimes (\epsilon_A \otimes p)$.

Above all, the morphism $\sigma = p \circ \eta_A : \mathds{1} \to A$ is a normalized right integral of $A$.
\end{proof}

\begin{remark}
\label{201912021229}
In Lemma \ref{small_has_nor_integral}, we show that a bimonoid $A$ has a normalized left integral under some assumptions on the bimonoid $A$.
Similarly, a bimonoid has a normalized right integral if $A$ satisfies similar assumptions on the regular right action.
Especially, if the bimonoid $A$ is small, then the bimonoid $A$ has a normalized left integral and a normalized right integral.
We also have a dual statement.
\end{remark}

\begin{Defn}
\label{201907132333}
\rm
Let $(A , \alpha , X)$ be a left action in a symmetric monoidal category $\mathcal{C}$.
For a morphism $a : \mathds{1} \to A$ in $\mathcal{C}$, we define an endomorphism $L_\alpha (a ) : X \to X$ by a composition,
\begin{align}
X \stackrel{\cong}{\to} \mathds{1} \otimes X \stackrel{a \otimes id_X}{\to} A \otimes X \stackrel{\alpha}{\to} X  . 
\end{align}

Let $(Y, \beta, B)$ be a right coaction in $\mathcal{C}$.
For a morphism $b : B \to \mathds{1}$ in $\mathcal{C}$, we define an endomorphism $R^\beta ( b) : Y \to Y$ by a composition,
\begin{align}
Y \stackrel{\beta}{\to} Y \otimes B \stackrel{id_Y \otimes b}{\to} Y \otimes \mathds{1} \stackrel{\mathrm{r}_Y}{\to} Y . 
\end{align}
\end{Defn}

\begin{prop}
\label{201907122224}
Let $(A,\alpha, X)$ be a left action in $\mathcal{C}$.
Then $a \in Mor_{\mathcal{C}}( \mathds{1} , A) \mapsto L_\alpha ( a) \in End_\mathcal{C} ( X)$ is a homomorphism.
Here, the monoid $End_\mathcal{C} ( X)$ consists of endomorphisms on $X$ :
\begin{align}
L_\alpha ( a \ast a^\prime) &= L_\alpha ( a) \circ L_\alpha ( a^\prime),~a,a^\prime \in Mor_{\mathcal{C}}( \mathds{1} , A) .  
\end{align}
Likewise, for a right coaction $(Y,\beta , B)$, the assignment $b \in Mor_{\mathcal{C}}( B, \mathds{1}) \mapsto R^\beta (b) \in End_\mathcal{C} ( Y)$ is a homomorphism :
\begin{align}
R^\beta ( b \ast b^\prime) &= R^\beta  (b) \circ R^\beta (b^\prime ) , ~b,b^\prime \in Mor_{\mathcal{C}}( B, \mathds{1})
\end{align}
\end{prop}
\begin{proof}
It follows from the associativity of an action and a coaction.
\end{proof}

\begin{prop}
\label{201912021501}
Let $A$ be a small bimonoid in a symmetric monoidal category $\mathcal{C}$.
Let $(A, \alpha , X)$ be a left action in $\mathcal{C}$.
Recall Lemma \ref{small_has_nor_integral}, then we have a normalized integral $\sigma_A$ of $A$.
The induced morphism $L_\alpha ( \sigma_A)$ is a split idempotent.
Moreover we have $_\alpha p =  L_\alpha ( \sigma_A)$ where $_\alpha p$ is given in Proposition \ref{201912021539}.
\end{prop}
\begin{proof}
The morphsim $L_\alpha ( \sigma_A)$ is an idempotent by Proposition \ref{201907122224} and $\sigma_A \ast \sigma_A = \sigma_A$.
$\sigma_A \ast \sigma_A = \sigma_A$ follows from the normality of $\sigma_A$.

Let $ \alpha \backslash\backslash X$ be an invariant object and $\alpha \backslash X$ be a stabilized object of the left action $(A,\alpha, X)$.
Denote by $\iota : \alpha \backslash\backslash X \to X$ and $\pi :X \to \alpha \backslash X$ the canonical morphisms.
We claim that the morphism $\iota$ gives an equalizer of $L_\alpha ( \sigma_A )$ and $id_X$, and the morphism $\pi$ gives a coequalizer of $L_\alpha ( \sigma_A )$ and $id_X$.
Then the idempotent $L_\alpha ( \sigma_A)$ is a split idempotent by Proposition \ref{201912021623}.

We prove that the morphism $\iota$ gives an equalizer of $L_\alpha ( \sigma_A )$ and $id_X$.
Note that $L_\alpha ( \sigma_A ) \circ \iota = id_X \circ \iota$ since the integral $\sigma_A$ is normalized.
We prove the universality.
Suppose that $f : Z \to X$ equalizes $L_\alpha ( \sigma_A )$ and $id_X$, i.e. $L_\alpha ( \sigma_A ) \circ f = f$.
Then $\alpha \circ ( id_A \otimes f) =\tau_{A,X} \circ ( id_A \otimes f)$ by Figure \ref{201912021633}.
By definition of the invariant object $\alpha \backslash \backslash X$, $f$ induces a unique morphism $f^\prime : Z \to \alpha \backslash \backslash X$ such that $\iota \circ f^\prime = f$.

\begin{figure}[ht]
  \includegraphics[width=\linewidth]{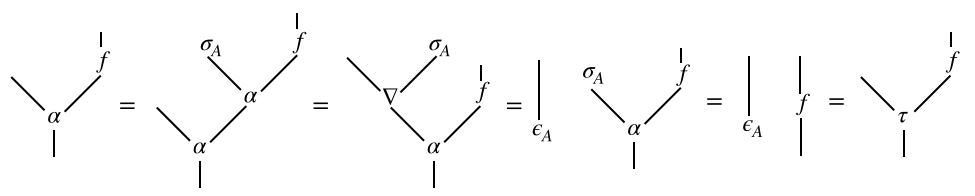}
  \caption{}
  \label{201912021633}
\end{figure}

We prove that the morphism $\pi$ gives a coequalizer of $L_\alpha ( \sigma_A )$ and $id_X$.
Note that $\pi \circ L_\alpha ( \sigma_A )$ and $\pi \circ id_X$ since the integral $\sigma_A$ is normalized.
We prove the universality.
Suppose that $g : X \to Z$ coequalizes $L_\alpha ( \sigma_A )$ and $id_X$, i.e. $g \circ L_\alpha ( \sigma_A ) = g$.
Then $g \circ \alpha = g \circ \tau_{A,X}$ by Figure \ref{201912021644}.
By definition of the stabilzed object $\alpha \backslash X$, the morphism $g$ induces a unique morphism $g^\prime : \alpha \backslash X \to Z$ such that $g^\prime \circ \pi = g$.

\begin{figure}[ht]
  \includegraphics[width=\linewidth]{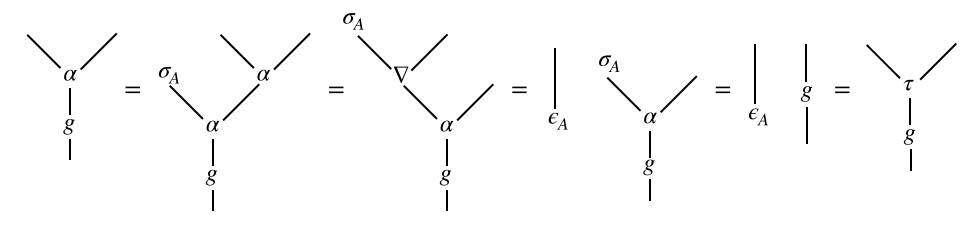}
  \caption{}
  \label{201912021644}
\end{figure}

All that remain is to prove that $_\alpha p = L_\alpha ( \sigma_A)$.
Note that $( \alpha \backslash \backslash X , \iota , _\alpha \gamma^{-1} \circ \pi)$ gives a retract of the idempotent of $L_\alpha ( \sigma_A)$.
See the proof of Proposition \ref{201912021623}.
Hence, $L_\alpha ( \sigma_A) = \iota \circ (_\alpha \gamma^{-1} \circ \pi) = _\alpha p$.
It completes the proof.
\end{proof}

\begin{theorem}
\label{small_integral_equiv}
Let $\mathcal{C}$ be a symmetric monoidal category.
Suppose that every idempotent in $\mathcal{C}$ is a split idempotent.
A bimonoid $A$ in $\mathcal{C}$ is small if and only if the bimonoid $A$ has a normalized integral.
\end{theorem}
\begin{proof}
By Proposition \ref{norm_r_l_integral_is_integral}, Lemma \ref{small_has_nor_integral}, and Remark \ref{201912021229}, if a bimonoid $A$ is small, then $A$ has a normalized integral.

Suppose that a bimonoid $A$ has a normalized integral $\sigma_A$.
Let $(A,\alpha,X)$ be a left action in $\mathcal{C}$.
Let us write $p = L_\alpha ( \sigma_A) : X \to X$.
By Proposition \ref{201907122224}, we have $p \circ p = L_\alpha ( \sigma_A) \circ L_\alpha ( \sigma_A) = L_\alpha ( \sigma_A \ast \sigma_A ) = L_\alpha ( \sigma_A) = p$ since $\sigma_A$ is a normalized integral of $A$.
In other words, the morphsim $p$ is an idempotent on $X$.
By the assumption, there exists a retract $(X^p , \iota , \pi )$ of the idempotent $p : X \to X$.
We claim that,
\begin{enumerate}
\item
The morphism $\pi : X \to X^p$ gives a stabilized object $\alpha \backslash X$ of the left action $(A , \alpha , X)$.
\item
The morphism $\iota : X^p \to X$ gives an invariant object $\alpha \backslash \backslash X$ of the left action $(A , \alpha , X)$.
\end{enumerate}
Then the canonical morphism $_\alpha \gamma : \alpha \backslash \backslash X \to \alpha \backslash X$ coincides with $\pi \circ \iota = id_{X^p}$ so that $_\alpha \gamma$ is an isomorphism.
It completes the proof.

We prove the first claim.
Suppose that a morphism $f : X \to Y$ coequalizes the action $\alpha : A \otimes X \to X$ and the trivial action $\tau_{A,X} : A \otimes X \to X$, i.e. $f \circ \alpha = f \circ \tau_{A,X}$.
We set $f^\prime = f \circ \iota : X^p \to Y$.
Then we have $f^\prime \circ \pi = f \circ \iota \circ \pi = f \circ p = f \circ L_\alpha ( \sigma_ A) = f \circ \alpha \circ ( \sigma_A \otimes id_X)$.
By $f \circ \alpha = f \circ \tau_{A,X}$, we obtain $f^\prime \circ \pi = f \circ \tau_{A,X} \circ ( \sigma_A \otimes id_X) = f$ since $\sigma_A$ is a normalized integral.
Moreover, if we have $f^{\prime\prime} \circ \pi = f$ for a morphism $f^{\prime\prime} : X^p \to Y$, then $f^{\prime\prime} = f^{\prime\prime} \circ \pi \circ \iota = f \circ \iota = f^\prime$.
Above all, the morphism $\pi : X \to X^p$ gives a stabilized object $\alpha \backslash X$ of the left action $(A , \alpha , X)$.

We prove the second claim.
The following diagram commutes :
\begin{equation}
\begin{tikzcd}
A \otimes X \ar[r, "\alpha"] & X \\
A \otimes X^p \ar[r, "\tau_{A,X^p}"] \ar[u, "id_A \otimes \iota"] & X^p \ar[u, "\iota"]
\end{tikzcd}
\end{equation}
It follows from Figure \ref{201912021338}.
\begin{figure}[ht]
  \includegraphics[width=14.5cm]{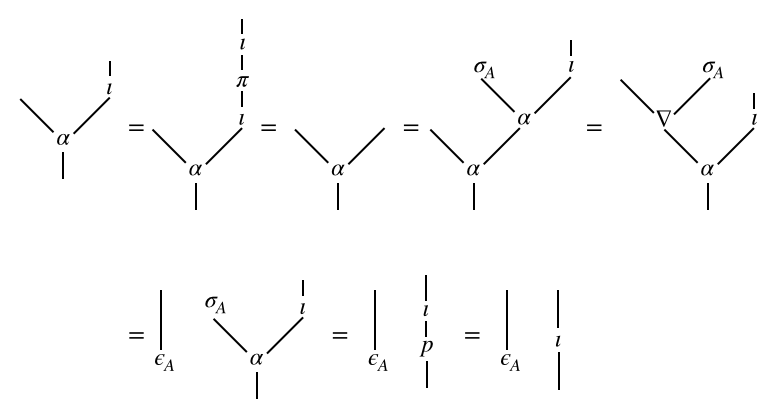}
  \caption{}
  \label{201912021338}
\end{figure}
We prove the universality of an invariant object.
Suppose that a morphism $g : Z \to X$ satisfies $\alpha \circ (id_A \otimes g) = \tau_{A,X} \circ (id_A \otimes g) : A \otimes Z \to X$.
Put $g^\prime = \pi \circ g : Z \to X^p : Z \to X^p$.
We have $\iota \circ g^\prime = \iota \circ \pi \circ g = p \circ g = \alpha \circ ( \sigma_A \otimes id_X ) \circ g = \tau_{A,X}  \circ ( \sigma_A \otimes id_X ) \circ g = g$ since $\sigma_A$ is the normalized integral.
If for a morphism $g^{\prime\prime} : Z \to X^p$ we have $\iota \circ g^{\prime\prime} = g$, then we have $g^{\prime\prime} = \pi \circ \iota \circ g^{\prime\prime} = \pi \circ g = g^\prime$.
It proves the universality of an invariant object $\iota : X^p \to X$.
\end{proof}

\begin{Corollary}
\label{201912021909}
Let $\mathcal{C}$ be a symmetric monoidal category.
Suppose that every idempotent in $\mathcal{C}$ is a split idempotent.
A bimonoid $A$ in $\mathcal{C}$ is bismall if and only if $A$ has a normalized integral and a normalized cointegral.
\end{Corollary}
\begin{proof}
We have a dual statement of Theorem \ref{small_integral_equiv}.
The dual statement and Theorem \ref{small_integral_equiv} complete the proof.
\end{proof}

\begin{Corollary}
\label{201907101753}
Suppose that every idempotent in $\mathcal{C}$ is a split idempotent.
The full subcategory of (co)small bimonoids in a symmetric monoidal category $\mathcal{C}$ forms a sub symmetric monoidal category of $\mathsf{Bimon}(\mathcal{C})$.
In particular, the full subcategory of bismall bimonoids in a symmetric monoidal category $\mathcal{C}$ forms a sub symmetric monoidal category of $\mathsf{Bimon}(\mathcal{C})$.
\end{Corollary}
\begin{proof}
We prove the claim for small cases and leave the second claim to the readers.
By Theorem \ref{small_integral_equiv}, small bimonoids $A,B$ have nomalized integrals $\sigma_A,\sigma_B$.
Then a morphism $\sigma_A \otimes \sigma_B : \mathds{1} \cong \mathds{1} \otimes \mathds{1} \to A \otimes B$ is verified to give a morphism of the bimonoid $A\otimes B$ by direct calculation.
Hence the bimonoid $A\otimes B$ possesses a normalized integral so that $A \otimes B$ is small by Theorem \ref{small_integral_equiv}.
It completes the proof.
\end{proof}

\section{Integral along bimonoid homomorphism}


\subsection{Uniqueness of normalized integral}
\label{201908051601}

In this subsection, we prove the uniqueness of normalized integrals along homomorphisms.
It is a generalization of the uniqueness of normalized (co)integrals of bimonoids in Proposition \ref{norm_r_l_integral_is_integral}.

\begin{prop}
\label{r_integral_l_integra_coincide_along}
Let $\xi : A \to B$ be a bimonoid homomorphism.
Suppose that $\mu \in Int_r (\xi), \mu^\prime \in Int_l (\xi)$ are normalized.
Then we have
\begin{align}
\mu = \mu^\prime \in Int (\xi). 
\end{align}
In particular, a normalized integral along $\xi$ is unique if exists.
\end{prop}
\begin{proof}
It is proved by two equations $\mu = \mu \circ \xi \circ \mu^\prime$  and $\mu^\prime = \mu \circ \xi \circ \mu^\prime $.
The former claim follows from (Figure \ref{unique_nor_int_along}) and the latter claim follows from (Figure \ref{unique_nor_int_along(2)}).
It completes the proof.

\begin{figure}[ht]
  \includegraphics[width=\linewidth]{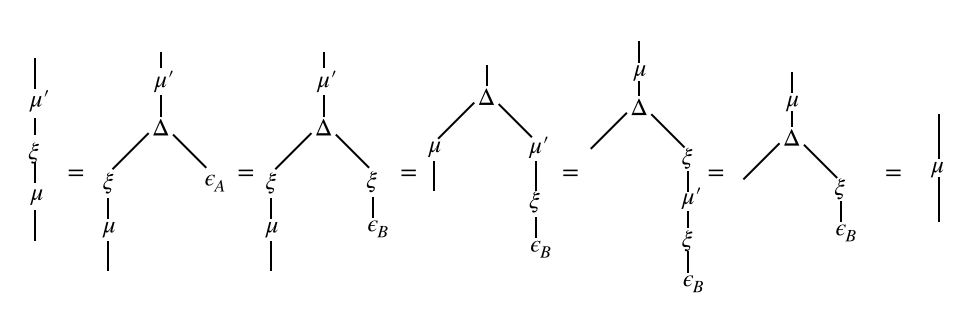}
  \caption{}
  \label{unique_nor_int_along}
\end{figure}
\begin{figure}[ht]
  \includegraphics[width=\linewidth]{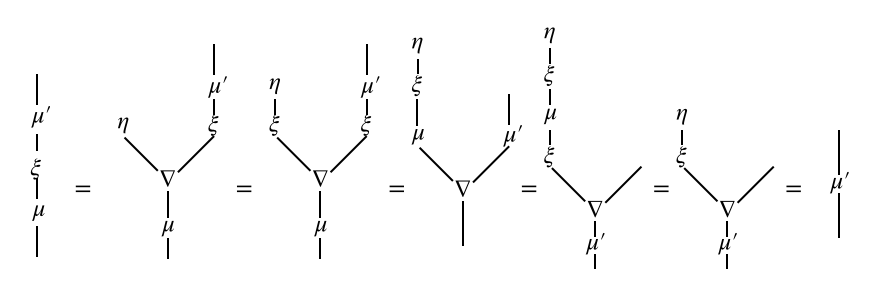}
  \caption{}
  \label{unique_nor_int_along(2)}
\end{figure}
\end{proof}

\begin{Corollary}
Let $\xi : A \to B$ a bimonoid homomorphism.
If $\mu \in Int (\xi)$ is normalized, then we have
\begin{itemize}
\item $\mu \circ \xi \circ \mu = \mu$.
\item $\mu \circ \xi : A \to A$ is an idempotent on $A$.
\item $\xi \circ \mu : B \to B$ is an idempotent on $B$.
\end{itemize}
\end{Corollary}
\begin{proof}
By direct verification, $\mu^\prime = \mu \circ \xi \circ \mu$ is an integral along $\xi$.
Also, $\mu^\prime$ is normalized since $\xi \circ \mu^\prime \circ \xi = \xi \circ \mu \circ \xi \circ \mu \circ \xi = \xi$ by the normality of $\mu$.
By Proposition \ref{r_integral_l_integra_coincide_along}, we have $\mu^\prime = \mu$.
It completes the proof of the first claim.
The other claims are immediate from the first claim.
\end{proof}


\subsection{Necessary conditions for existence of a normalized integral}
\label{201908051602}

An existence of a normalized integral along a homomorphism $\xi$ is related with an existence of a normlaized integral of $Ker (\xi)$ and a cointegral $Cok (\xi)$.
In this subsection, we prove Theorem \ref{norm_inte_along_induces_norm_inte} which implies Theorem \ref{201912051457}.
We define an integral $\check{F} (\mu)$ of $Ker (\xi)$ from an integral $\mu$ along $\xi$ when $\xi$ is conormal.
Furthermore, if the integral $\mu$ is normalized, then the integral $\check{F} ( \mu)$ is normalized.

\begin{Lemma}
\label{201907292200}
Let $\mu \in Int_r (\xi)$.
Then $\mu \circ \eta_B : \mathds{1} \to A$ equalizes the homomorphism $\xi $ and the trivial homomorphism, i.e. $\xi\circ ( \mu\circ\eta_B) = \eta_B\circ\epsilon_A\circ ( \mu\circ\eta_B )$.
\end{Lemma}
\begin{proof}
It is verified by Figure \ref{mu_eta_equalize_hom}.
\begin{figure}[ht]
  \includegraphics[width=11.3cm]{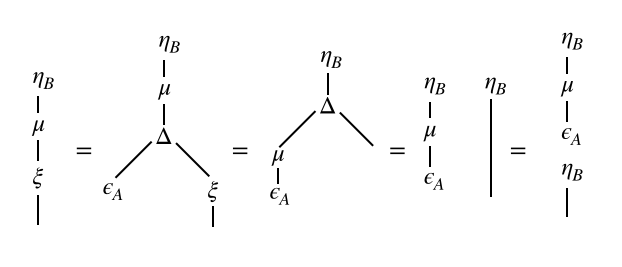}
  \caption{}
  \label{mu_eta_equalize_hom}
\end{figure}
\end{proof}

\begin{Defn}
\label{201907311644}
\rm
Let $\xi : A \to B$ be a bimonoid homomorphism and $\mu \in Int_r (\xi)$.
If $\xi$ is conormal, a morphism $\check{F} ( \mu ) : \mathds{1} \to Ker (\xi)$ is defined as follows.
By Lemma \ref{201907292200}, $\mu \circ \eta_B$ is decomposed into 
\begin{equation}
\mathds{1} \stackrel{\varphi}{\to} 
A \backslash \beta^{\leftarrow}_\xi \to
A . 
\end{equation}
Since $\xi$ is conormal, $A \backslash \beta^{\leftarrow}_\xi$ gives a kernel bimonoid of $\xi$, $Ker (\xi)$ so that the morphism $\varphi$ defines $\check{F} (\mu) : \mathds{1} \to Ker (\xi)$.

If $\xi$ is normal, we define a morphism $\hat{F} ( \mu ) : Cok (\xi ) \to \mathds{1}$ in an analogous way, i.e. $\epsilon_A \circ \mu$ is decomposed into 
\begin{equation}
B \to
Cok (\xi) \stackrel{\hat{F}(\mu)}{\to}
\mathds{1} . 
\end{equation}
\end{Defn}

\begin{theorem}
\label{norm_inte_along_induces_norm_inte}
Let $A,B$ be bimonoids and  $\xi : A \to B$ be a bimonoid homomorphism
Let $\mu \in Int_r (\xi)$.
\begin{enumerate}
\item
Suppose that $\xi$ is conormal.
Then the morphism $\check{F}(\mu) : \mathds{1} \to Ker ( \xi )$ is defined and it is a right integral of $Ker (\xi)$.
If the integral $\mu$ along $\xi$ is normalized, then the integral $\check{F}(\mu)$  is normalized.
\item
Suppose that $\xi : A \to B$ is normal.
Then the morphism $\hat{F}(\mu) : Cok ( \xi) \to \mathds{1}$ is defined and it is a right cointegral of $Cok (\xi)$.
If the integral $\mu$ along $\xi$ is normalized, then the cointegral $\hat{F}(\mu)$ is normalized.
\end{enumerate}
\end{theorem}
\begin{proof}
We only prove the first part.
For simplicity, let us write $j= ker (\xi) : Ker (\xi) \to A$.
We prove that $\nabla_{Ker(\xi)} \circ ( \check{F}(\mu) \otimes id_{Ker(\xi)}) = \check{F}(\mu) \otimes \epsilon_{Ker (\xi)}$.
Due to the universality of kernels, it suffices to show that $j \circ \nabla_{Ker(\xi)} \circ ( \check{F}(\mu) \otimes id_{Ker(\xi)}) = j \circ (\check{F}(\mu) \otimes \epsilon_{Ker (\xi)})$.
See Figure \ref{pf induced integral of kernel}.

Let us prove that $\check{F}(\mu)$ is normalized if $\mu$ is normalized.
It is shown by the following direct calculation :
\begin{align}
\epsilon_{Ker(\xi)} \circ \check{F}(\mu) 
&= \epsilon_{A} \circ ker(\xi) \circ \check{F}(\mu)  \\
&= \epsilon_{A} \circ \mu \circ \eta_B \\
&= \epsilon_{B} \circ \xi \circ \mu \circ \xi \circ \eta_A  \\
&= \epsilon_{B} \circ \xi \circ \eta_A ~~(\because \mu : \mathrm{normalized})  \\
&= id_{\mathds{1}} 
\end{align}

\begin{figure}[ht]
  \includegraphics[width=\linewidth]{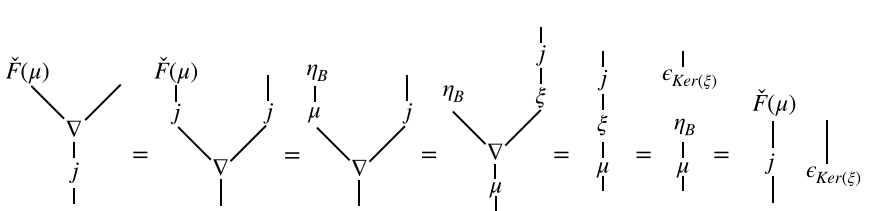}
  \caption{}
  \label{pf induced integral of kernel}
\end{figure}
\end{proof}

\begin{Corollary}
\label{201912051457}
Let $\xi : A \to B$ be a bimonoid homomorphism with a normalized integral along $\xi$.
If the homomorphism $\xi$ is conormal, then the kernel bimonoid $Ker (\xi)$ has a normalized integral.

We have a dual claim : if the homomorphism $\xi$ is normal, then the cokernel bimonoid $Cok (\xi)$ has a normalized cointegral.
\end{Corollary}


\section{Computation of $Int (\xi)$}
\label{201908021404}

In this section, we compute $Int ( \xi)$ by using $\check{F}$, $\hat{F}$ in Definition \ref{201907311644}.
The main result in this subsection is that if $\xi$ has a normalized generator integral, then $Int (\xi)$ is isomorphic to $End_{\mathcal{C}}(\mathds{1})$, the endomorphism set of the unit $\mathds{1} \in \mathcal{C}$.

\begin{Defn}
\rm
Let $A,B$ be bimonoids and $\xi : A \to B$ be a bimonoid homomorphism with a kernel bimonoid $Ker(\xi)$.
Let $\varphi \in Mor_{\mathcal{C}} (\mathds{1} , Ker(\xi))$ and $\mu \in Int_r ( \xi)$.
We define $\varphi \ltimes \mu \in Mor_{\mathcal{C}} ( B, A)$ by 
\begin{align}
\varphi \ltimes \mu &\stackrel{\mathrm{def.}}{=} \left( B \stackrel{\cong}{\to} \mathds{1} \otimes B \stackrel{\varphi \otimes id_B}{\to} Ker (\xi) \otimes B \stackrel{ker(\xi)\otimes \mu}{\to} A \otimes A \stackrel{\nabla_A}{\to} A \right)  \\
\mu \rtimes \varphi &\stackrel{\mathrm{def.}}{=} \left( B \stackrel{\cong}{\to} B \otimes \mathds{1}  \stackrel{id_B \otimes \varphi}{\to} B \otimes Ker (\xi) \stackrel{\mu \otimes ker(\xi)}{\to} A \otimes A \stackrel{\nabla_A}{\to} A \right) 
\end{align}
\end{Defn}

\begin{remark}
The definitions of $\varphi \ltimes \mu$ and $\mu \rtimes \varphi$ can be understood via some string diagrams in Figure \ref{201908021116}.
\begin{figure}[ht]
  \includegraphics[width=11.3cm]{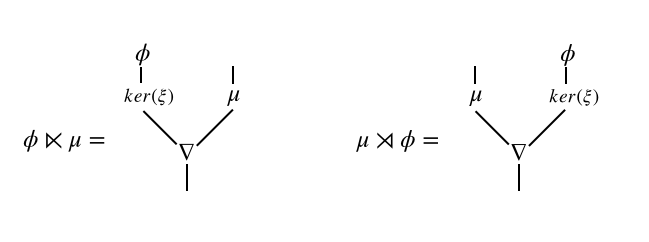}
  \caption{}
  \label{201908021116}
\end{figure}
\end{remark}

\begin{prop}
Let $\mu \in Int_r (\xi)$.
Then we have
\begin{itemize}
\item
$\varphi \ltimes \mu \in Int_r (\xi)$.
\item
$\mu \rtimes \varphi = (\epsilon_{Ker(\xi)} \circ \varphi ) \cdot \mu \in Int_r (\xi)$.
\end{itemize}
\end{prop}
\begin{proof}
For simplicity we denote $j= ker (\xi) : Ker (\xi) \to A$.
We show that $\varphi \ltimes \mu \in Int_r (\xi)$.
The axiom (\ref{Haar_fam_axiom1}) is verified by Figure \ref{phi_mu_is_integral}.
The axiom (\ref{Haar_fam_axiom2}) is verified by Figure \ref{phi_mu_is_integral(2)}.
Note that the target of $\varphi$ needs to be $Ker (\xi)$ to verify Figure \ref{phi_mu_is_integral(2)}.

We show that $\mu \rtimes \varphi = (\epsilon_{Ker(\xi)} \circ \varphi ) \cdot \mu \in Int_r (\xi)$.
The equation is verified by Figure \ref{mu_phi_is_integral}.
Since $\mu \in Int_r (\xi)$,  $\mu \rtimes \varphi $ lives in $Int_r (\xi)$.

\begin{figure}[ht]
  \includegraphics[width=\linewidth]{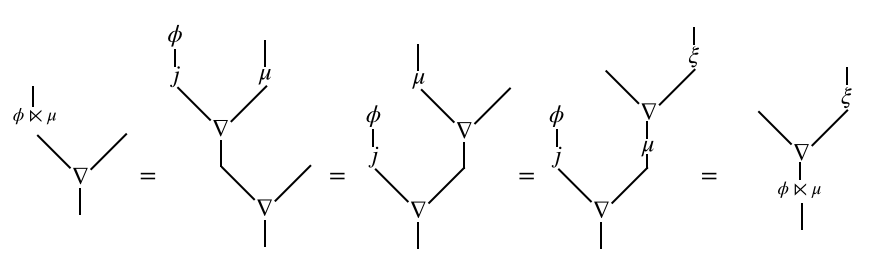}
  \caption{}
  \label{phi_mu_is_integral}
\end{figure}

\begin{figure}[ht]
  \includegraphics[width=\linewidth]{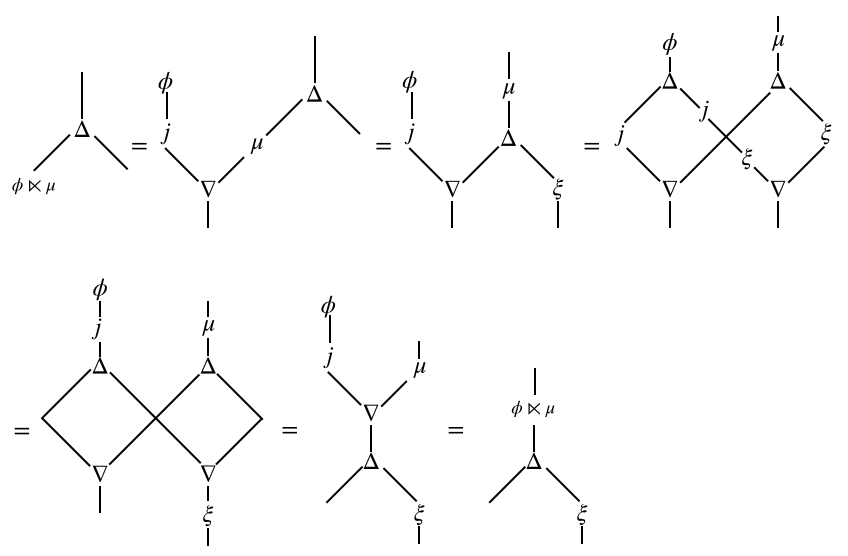}
  \caption{}
  \label{phi_mu_is_integral(2)}
\end{figure}

\begin{figure}[ht]
  \includegraphics[width=12cm]{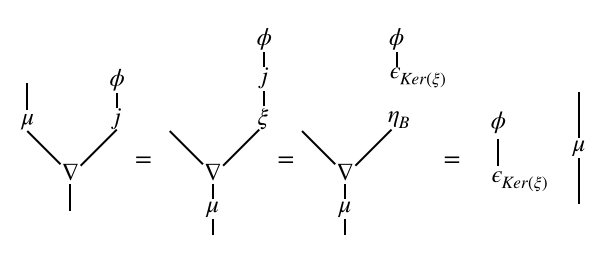}
  \caption{}
  \label{mu_phi_is_integral}
\end{figure}
\end{proof}

\begin{Lemma}
\label{201907311125}
Let $\xi : A \to B$ be a bimonoid homomorphism which is conormal.
Let $\mu$ be a generator integral along $\xi$.
For an integral $\mu^\prime \in Int ( \xi) $, we have
\begin{align}
\check{F}(\mu^\prime) \ltimes \mu 
= \mu^\prime . 
\end{align}

In particular, if a bimonoid homomorphism $\xi$ has a generator integral, then $\check{F} : Int (\xi ) \to Int (Ker (\xi))$ is injective.
\end{Lemma}
\begin{proof}
It follows from Figure \ref{201912072155}.
\begin{figure}[ht]
  \includegraphics[width=10.5cm]{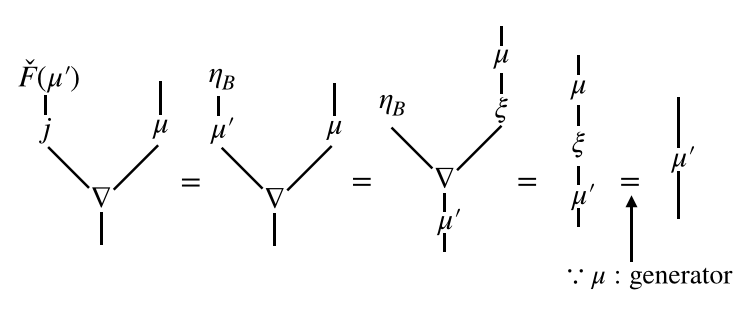}
  \caption{}
  \label{201912072155}
\end{figure}
\end{proof}

\begin{theorem}
\label{201906281559}
Let $\xi : A \to B$ be a bimonoid homomorphism which is either conormal or normal.
Let $\mu$ be a normalized generator integral along $\xi$.
Then the map $End_{\mathcal{C}}(\mathds{1}) \to Int (\xi) ~;~ \lambda \mapsto \lambda \cdot \mu$ is a bijection.
\end{theorem}
\begin{proof}
We only prove the statement for conormal $\xi$.
It suffices to replace $\check{F}(\mu)$ with $\hat{F}(\mu)$ for normal $\xi$ and other discussion with a dual one.

We claim that $Int (\xi) \to End_{\mathcal{C}}(\mathds{1}) ; \mu^\prime \mapsto \epsilon_{Ker(\xi)} \circ \check{F}(\mu^\prime)$ gives an inverse map.
It suffices to prove that $\mu^\prime = \left( \epsilon_{Ker(\xi)} \circ \check{F}(\mu^\prime) \right) \cdot \mu$ and $\epsilon_{Ker(\xi)} \circ \check{F}(\lambda \cdot \mu) = \lambda$.
The latter one follows from $\epsilon_{Ker(\xi)} \circ \check{F}(\mu) = id_{\mathds{1}}$ which is nothing but the normality of $\check{F}(\mu)$ by Theorem \ref{norm_inte_along_induces_norm_inte}.
We show the former one by calculating $\check{F}(\mu^\prime) \ltimes \mu$ in a different way.
It follows from Figure \ref{201912072203}.
\begin{figure}[ht]
  \includegraphics[width=10.5cm]{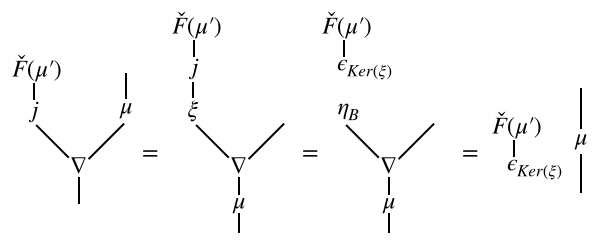}
  \caption{}
  \label{201912072203}
\end{figure}
By Lemma \ref{201907311125}, $\check{F}(\mu^\prime ) \ltimes \mu  = \mu^\prime$, so that $\mu^\prime = \left( \epsilon_{Ker(\xi)} \circ \check{F}(\mu^\prime ) \right) \cdot \mu$.
\end{proof}


\section{Existence of a normalized generator integral}

In this section, we give sufficient conditions for a normalized generator integral along a homomorphism exists.


\subsection{Key Lemma}
\label{201908051614}

\begin{Lemma}
\label{ker_cok_integral_exist}
Let $A,B$ be bimonoids.
Let $\xi : A \to B$ be a bimonoid homomorphism.
\begin{enumerate}
\item
Suppose that $A$ is small.
In particular, the canonical morphism $_\xi\gamma : \alpha^{\to}_\xi \backslash\backslash B \to \alpha^{\to}_\xi \backslash B$ is an isomorphism.
Here, the left action $\alpha^{\to}_\xi$ is defined in Definition \ref{201912021024}.
Let 
\begin{align}
\mu_{0} = \left( \alpha^{\to}_\xi \backslash B \stackrel{(_{\xi}\gamma)^{-1}}{\to}  \alpha^{\to}_\xi \backslash \backslash B \to B \right) . 
\end{align}
If $\alpha^{\to}_\xi \backslash B$ has a bimonoid structure such that the canonical morphism $\pi : B \to \alpha^{\to}_\xi \backslash B$ is a bimonoid homomorphism, then we have
\begin{itemize}
\item 
$\mu_0  \in Int_r (\pi)$. 
In particular, $Int_r (\pi) \neq \emptyset$.
\item
$\pi \circ \mu_0 = id_{\alpha^{\to}_\xi \backslash B}$.
In particular, the right integral $\mu_0$ is normalized.
\item
By Remark \ref{201912021229}, the bimonoid $A$ has a normalized integral $\sigma_A$.
We have,
\begin{align}
\mu_0 \circ \pi = L_{\alpha^{\to}_{\xi}} ( \sigma_A ) . 
\end{align}
\end{itemize}
If $B$ is commutative, then $\mu_0 \in Int_l (\pi)$, in particular, $\mu_{0} \in Int (\pi) \neq \emptyset$.
We have an analogous statement for the right action $(B, \alpha^{\leftarrow}_\xi, A)$.
\item
Suppose that $B$ is cosmall.
In particular, the canonical morphism $\gamma^\xi : A \backslash \beta^{\leftarrow}_\xi \to A \backslash \backslash \beta^{\leftarrow}_\xi$ is an isomorphism.
Here, the right coaction $\beta^{\leftarrow}_\xi$ is defined in Definition \ref{201912021024}.
Let 
\begin{align}
\mu_{1} = \left(A \to A \backslash \backslash \beta_{\xi} \stackrel{(\gamma^{\xi})^{-1}}{\to} A \backslash \beta_{\xi}  \right) . 
\end{align}
If $A \backslash \beta_{\xi}$ has a bimonoid structure such that the canonical morphism $\iota : A \backslash \beta \to A$ is a bimonoid homomorphism, then we have
\begin{itemize}
\item 
$\mu_1 \in Int_l (\iota)$.
In particular, $Int_l (\iota) \neq \emptyset$.
\item
$\mu_1 \circ \iota = id_{A \backslash \beta^{\leftarrow}_\xi }$.
In particular, the left integral $\mu_1$ is normalized.
\item
By Remark \ref{201912021229}, the bimonoid $B$ has a normalized cointegral $\sigma^B$.
We have,
\begin{align}
\iota \circ \mu_1 = R^{\beta^{\leftarrow}_{\xi}} ( \sigma^B ) . 
\end{align} 
\end{itemize}
If $A$ is cocommutative, then $\mu_{1} \in Int_r (\iota)$, in particular, $\mu_{1} \in Int (\iota) \neq \emptyset$.
We have an analogous statement for the left coaction $(B, \beta^{\to}_\xi , A)$.
\end{enumerate}
\end{Lemma}
\begin{proof}
We prove the first claim here and leave the second claim to the readers.
Recall Lemma \ref{small_has_nor_integral} that a small bimonoid $A$ has a normalized integral.
We denote the normalized integral by $\sigma_A : \mathds{1} \to A$.

We prove that $\mu_0$ satisfies the axiom (\ref{Haar_fam_axiom1}).
Denote by $j : \alpha^{\to}_\xi \backslash \backslash B \to B$ the canonical morphism.
Since $\gamma = _\xi\gamma$ is an isomorphism, it suffices to show that $\nabla_B \circ ((\mu_0 \circ \gamma) \otimes id_B) = \mu_0 \circ \nabla_{\alpha^{\to}_\xi \backslash B} \circ (\gamma \otimes \pi)$.
It is verified by Figure \ref{mu_0_integral(2)}.

\begin{figure}[ht]
  \includegraphics[width=\linewidth]{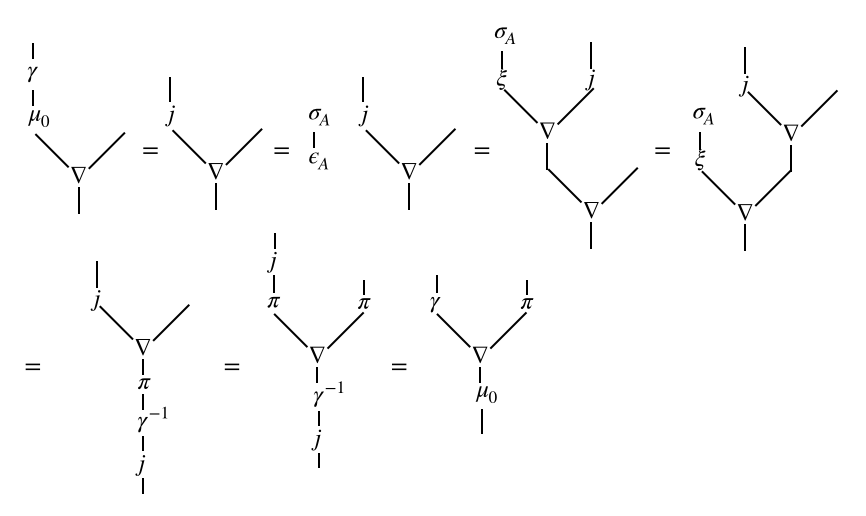}
  \caption{}
  \label{mu_0_integral(2)}
\end{figure}
We prove that $\mu_0$ satisfies the axiom (\ref{Haar_fam_axiom2}).
Due to the universality of $\pi : B \to \alpha^{\to}_\xi \backslash B$, it suffices to show that $(\mu_0 \otimes id_{\alpha^{\to}_\xi \backslash B}) \circ \Delta_{\alpha^{\to}_\xi \backslash B} \circ \pi = (id_B \otimes \pi) \circ \Delta_B \circ \mu_0 \circ \pi$.
It is verified by Figure \ref{mu_0_integral}.
Thus, we obtain $\mu_0  \in Int_r (\pi)$.

\begin{figure}[ht]
  \includegraphics[width=\linewidth]{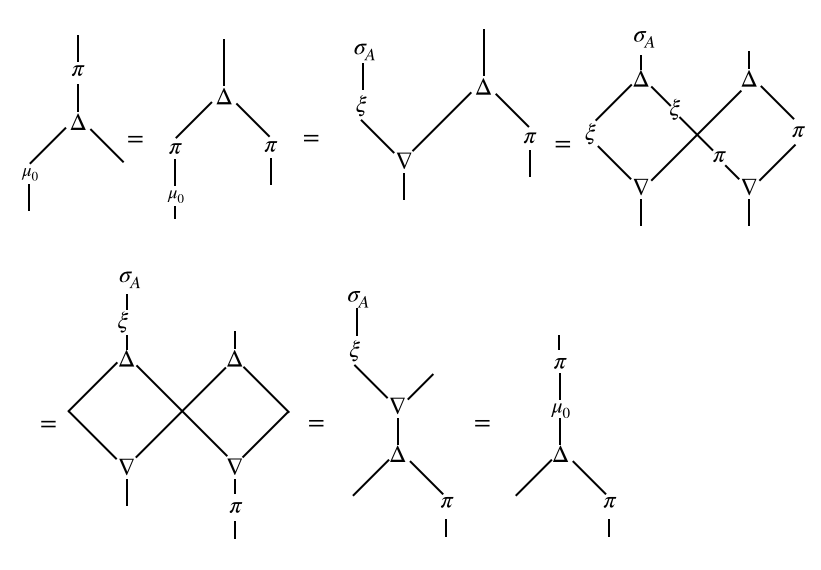}
  \caption{}
  \label{mu_0_integral}
\end{figure}

The claim $\pi \circ \mu_0 id_{\alpha^{\to}_\xi\backslash B}$ follows from $\pi \circ \mu_0 = _\xi\gamma \circ  (_\xi\gamma)^{-1} = id_{\alpha^{\to}_\xi\backslash B}$.

The claim $\mu_0 \circ \pi = L_{\alpha^{\to}_{\xi}} ( \sigma_A)$ follows from the definition of $\alpha^{\to}_\xi$ and Proposition \ref{201912021501}.

From now on, we suppose that $B$ is commutative and show that $\mu \in Int_l (\pi)$.
We prove that $\mu_0$ satisfies the axiom (\ref{Haar_fam_axiom3}).
Since $\gamma = _\xi \gamma$ is an isomorphism, it suffices to show that $\nabla_B \circ (id_B \otimes (\mu \circ \gamma)) = \mu \circ \nabla_{\alpha^{\to}_\xi\backslash B} \circ (\pi \otimes \gamma)$.
It is verified by Figure \ref{mu_0_integral(3)}.
We need the commutativity of $B$ here.

\begin{figure}[ht]
  \includegraphics[width=\linewidth]{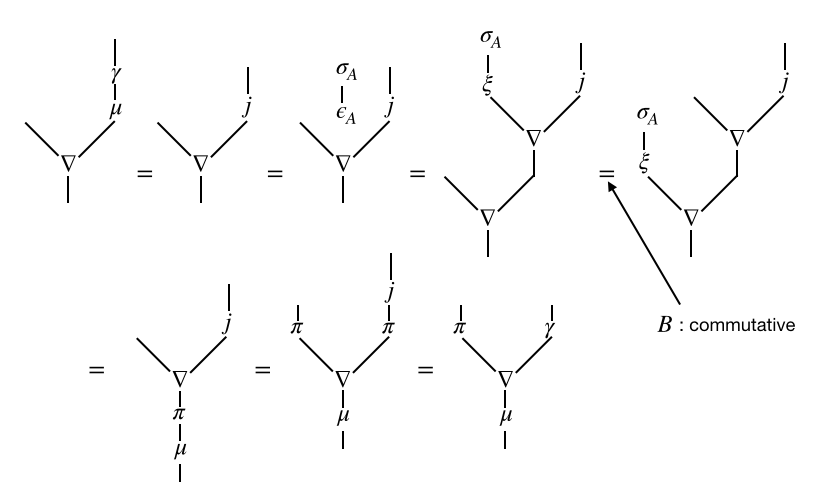}
  \caption{}
  \label{mu_0_integral(3)}
\end{figure}

We prove that $\mu_0$ satisfies the axiom (\ref{Haar_fam_axiom4}).
Due to the universality of $\pi : B \to \alpha^{\to}_\xi \backslash B$, it suffices to show that $(id_{\alpha^{\to}_\xi \backslash B} \otimes \mu) \circ \Delta_{\alpha^{\to}_\xi \backslash B} \circ \pi = (\pi \otimes id ) \circ \Delta_B \circ \mu_0 \circ \pi$.
It is verified by Figure \ref{mu_0_integral(4)}.

\begin{figure}[ht]
  \includegraphics[width=\linewidth]{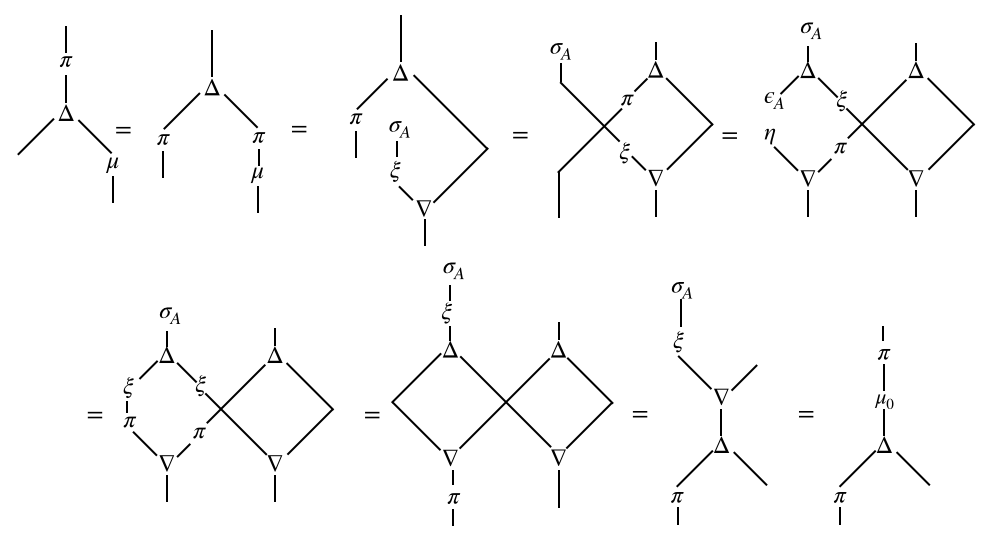}
  \caption{}
  \label{mu_0_integral(4)}
\end{figure}
\end{proof}

\begin{Defn}
\rm
\label{defn_integral_cok_ker}
Let $A,B$ be bimonoids in a symmetric monoidal category $\mathcal{C}$ and $\xi : A \to B$ be a bimonoid homomorphism.
Suppose that the bimonoid $A$ is small and $\xi$ is normal.
By Lemma \ref{ker_cok_integral_exist}, there exists a normalized right integral along the homomorphism $cok(\xi) : B \to Cok ( \xi)$.
Analogously, there also exists a normalized left integral along $cok(\xi)$ since the homomorphism $\xi$ is normal.
By Proposition \ref{r_integral_l_integra_coincide_along}, these coincide to each other.
Denote the normalized integral by $\tilde{\mu}_{cok(\xi)} \in Int ( cok (\xi))$.

Suppose that $B$ is cosmall and $\xi$ is conormal.
Analogously, by Lemma \ref{ker_cok_integral_exist}, we define a normalized integral $\tilde{\mu}_{ker(\xi)} \in Int (ker (\xi))$.
\end{Defn}

\begin{Lemma}
\label{201907311900}
Let $A,B$ be bimonoids and $\xi : A \to B$ be a bimonoid homomorphism.
Suppose that $A$ is small and the homomorphism $\xi$ is normal.
Then we have
\begin{align}
cok (\xi ) \circ \tilde{\mu}_{cok(\xi)} &= id_{Cok(\xi)}  \\
\tilde{\mu}_{cok(\xi)} \circ cok (\xi) 
&=  L_{\alpha^{\to}_{\xi}} (\sigma_A)  \\
&=  R_{\alpha^{\leftarrow}_{\xi}} (\sigma_A) 
\end{align}
In particular, $cok (\xi)$ has a section in $\mathcal{C}$.

Suppose that $B$ is cosmall and the canonical morphism $\xi$ is conormal.
Then we have, 
\begin{align}
\tilde{\mu}_{ker(\xi)} \circ ker(\xi) &= id_{Ker(\xi)}  \\
ker(\xi) \circ \tilde{\mu}_{ker(\xi)} 
&=R^{\beta^{\leftarrow}_{\xi}} ( \sigma^B )  \\
&= L^{\beta^{\to}_{\xi}} ( \sigma^B )  
\end{align}
In particular, $ker (\xi)$ has a retract in $\mathcal{C}$.
\end{Lemma}
\begin{proof}
It follows from the definitions of $\tilde{\mu}_{cok(\xi)}$, $\tilde{\mu}_{ker(\xi)}$ and Lemma \ref{ker_cok_integral_exist}.
\end{proof}


\subsection{Sufficient conditions for existence of a normalized generator integral}
\label{201908051616}

In this subsection, we give some sufficient conditions for a normalized generator integral to exist.
Furthermore, we prove Corollary \ref{202002211035} which implies our main theorem.

\begin{Defn}
\label{201907311908}
\rm
Let $A,B$ be bimonoids and $\xi : A \to B$ be a bimonoid homomorphism with a kernel bimonoid $Ker ( \xi)$.
Suppose that $Ker (\xi)$ is small and the canonical morphism $ker (\xi) : Ker (\xi) \to A$ is normal.
We define a normalized integral along $coim(\xi) = cok (ker (\xi)): A \to Coim (\xi)$ by $\tilde{\mu}_{cok( \zeta)}$ in Definition \ref{defn_integral_cok_ker} where $\zeta = ker ( \xi)$.
We denote $\tilde{\mu}_{cok( \zeta)}$ by $\tilde{\mu}_{coim(\xi)} \in Int ( coim (\xi))$.

Analogously we define $\tilde{\mu}_{im(\xi)}$ :
Let $A,B$ be bimonoids and $\xi : A \to B$ be a bimonoid homomorphism with a cokernel bimonoid $Cok ( \xi)$.
Suppose that $Cok (\xi)$ is cosmall and the canonical morphism $ker (\xi) : Ker (\xi) \to A$ is conormal.
We define a normalized integral along $im(\xi) = ker (cok (\xi)): A \to Im (\xi)$ by $\tilde{\mu}_{ker( \zeta)}$ in Definition \ref{defn_integral_cok_ker} where $\zeta = cok ( \xi)$.
We denote $\tilde{\mu}_{ker( \zeta)}$ by $\tilde{\mu}_{im(\xi)} \in Int ( im (\xi))$.
\end{Defn}

\begin{Lemma}
\label{coim_im_norm_integral_exist}
Let $A,B$ be bimonoids and $\xi : A \to B$ be a bimonoid homomorphism with a kernel $Ker(\xi)$.
Suppose that the kernel bimonoid $Ker (\xi)$ is small and the canonical morphism $ker (\xi) : Ker (\xi) \to A$ is normal.
Then we have
\begin{align}
coim (\xi ) \circ \tilde{\mu}_{coim(\xi)} &= id_{Coim(\xi)}  \\
\tilde{\mu}_{coim(\xi)} \circ coim (\xi) 
&=  L_{\alpha^{\to}_{ker(\xi)}} (\sigma_{Ker(\xi)})  \\
&=  R_{\alpha^{\leftarrow}_{ker(\xi)}} (\sigma_{Ker(\xi)}) 
\end{align}
In particular, $coim (\xi)$ has a section in $\mathcal{C}$.

An analogous statement for $Im (\xi)$ holds :
Let $A,B$ be bimonoids and $\xi : A \to B$ be a bimonoid homomorphism with a cokernel bimonoid $Cok(\xi)$.
Suppose that $Cok (\xi)$ is cosmall and the canonical morphism $cok (\xi) : B \to Cok (\xi)$ is conormal.
Then we have, 
\begin{align}
\tilde{\mu}_{im(\xi)} \circ im(\xi) &= id_{Im (\xi)}  \\
im(\xi) \circ \tilde{\mu}_{im(\xi)} 
&=R^{\beta^{\leftarrow}_{cok(\xi)}} ( \sigma^{Cok(\xi)} )  \\
&= L^{\beta^{\to}_{cok(\xi)}} ( \sigma^{Cok(\xi)} )  
\end{align}
In particular, $im (\xi)$ has a retract in $\mathcal{C}$.
\end{Lemma}
\begin{proof}
It follows from Lemma \ref{201907311900}.
\end{proof}

\begin{Defn}
\label{201908041300}
\rm
Let $A,B$ be bimonoids.
A bimonoid homomorphism $\xi : A \to B$ is {\it weakly well-decomposable} if the following conditions hold :
\begin{itemize}
\item
$Ker ( \xi )$, $Cok (\xi)$, $Coim (\xi)$, $Im (\xi)$ exist in $\mathsf{Bimon}(\mathcal{C})$.
\item
$ker (\xi) : Ker (\xi) \to A$ is normal and $cok (\xi) : B \to Cok (\xi)$ is conormal.
\item
$\bar{\xi} : Coim (\xi) \to Im (\xi)$ is an isomorphism.
\end{itemize}

A bimonoid homomorphism $\xi : A \to B$ is {\it well-decomposable} if following conditions hold :
\begin{itemize}
\item
$\xi$ is binormal.
In particular, $Ker ( \xi )$, $Cok (\xi)$ exist in $\mathsf{Bimon}(\mathcal{C})$.
\item
$ker (\xi) : Ker (\xi) \to A$ is normal and $cok (\xi) : B \to Cok (\xi)$ is conormal.
In particular, $Coim (\xi)$, $Im (\xi)$ exist.
\item
$\bar{\xi} : Coim (\xi) \to Im (\xi)$ is an isomorphism.
\end{itemize}
\end{Defn}

\begin{Defn}
\label{201907312115}
\rm
Let $\xi : A \to B$ be a weakly well-decomposable homomorphism.
The homomorphism $\xi$ is {\it weakly pre-Fredholm} if the kernel bimonoid $Ker (\xi)$ is small and the cokernel bimonoid $Cok (\xi)$ is cosmall.
Recall Definition \ref{201907311908}.
For a weakly pre-Fredholm homomorphism $\xi : A \to B$, we define
\begin{align}
\mu_\xi \stackrel{\mathrm{def.}}{=} \tilde{\mu}_{coim(\xi)} \circ \bar{\xi}^{-1} \circ \tilde{\mu}_{im(\xi)} : B \to A . 
\end{align}
The homomorphism $\xi$ is {\it pre-Fredholm} if both of the kernel bimonoid $Ker (\xi)$ and the cokernel bimonoid $Cok (\xi)$ are bismall.
\end{Defn}

\begin{prop}
\label{201907311130}
Let $A$ be a bimonoid.
\begin{enumerate}
\item
The unit $\eta_A : \mathds{1} \to A$ and the counit $\epsilon_A : A \to \mathds{1}$ are well-decomposable.
\item
The unit $\eta_A$ is weakly pre-Fredholm if and only if $A$ is cosmall.
Then $\mu_{\eta_A}$ in Definition \ref{201907312115} is well-defined and we have $\mu_{\eta_A} = \sigma^A$.
\item
The counit $\epsilon_A$ is weakly pre-Fredholm if and only if $A$ is small.
Then $\mu_{\epsilon_A}$in Definition \ref{201907312115} is well-defined and we have $\mu_{\epsilon_A} = \sigma_A$.
\end{enumerate}
\end{prop}
\begin{proof}
We prove that $\eta_A$ is well-decomposable and leave the proof of $\epsilon_A$ to the readers.
Note that the unit bimonoid $\mathds{1}$ is bismall since it has a normalized (co)integral.
The bimonoid homomorphism $\eta_A$ is normal due to the canonical isomorphism $\alpha_{\eta_A} \backslash A \leftarrow A = Cok (\eta_A)$.
The bimonoid homomorphism $\eta_A$ is conormal due to the canonical isomorphism $\mathds{1} \backslash \beta_{\eta_A} \to \mathds{1} = Ker (\eta_A)$.
Moreover, $ker (\eta_A ) : Ker (\eta_A) = \mathds{1} \to \mathds{1}$ and $cok (\eta_A) : A \to Cok (\eta_A) = A$ are normal and conormal due to Proposition \ref{201907021116}.
The final axiom is verified since $\bar{\eta}_A : \mathds{1} = Coim (\eta_A) \to Im (\eta_A) = \mathds{1}$ is the identity.

The morphism $\mu_{\eta_A}$ is a normalized integral by the following Theorem \ref{201907292156}.
By Proposition \ref{r_integral_l_integra_coincide_along}, we obtain $\mu_{\eta_A} = \sigma^A$.
\end{proof}

\begin{theorem}
\label{201907292156}
Let $A,B$ be bimonoids and $\xi : A \to B$ be a weakly well-decomposable homomorphism.
If the homomorphism $\xi$ is weakly pre-Fredholm, then the morphism $\mu_\xi$ is a normalized generator integral along $\xi$.
\end{theorem}
\begin{proof}
Recall that $\tilde{\mu}_{coim(\xi)} \in Int (coim(\xi)),\tilde{\mu}_{im(\xi)} \in Int (im (\xi))$ by Definition \ref{201907311908}.
By Proposition \ref{201907311053}, $\bar{\xi}^{-1} \in Int (\bar{\xi})$.
By Proposition \ref{201907311055}, $\mu_\xi$ is an integral along $\xi$ since $\mu_\xi$ is defined to be a composition of $\tilde{\mu}_{coim(\xi)},\tilde{\mu}_{im(\xi)},\bar{\xi}^{-1}$.

Note that $\mu_\xi \circ \xi = \tilde{\mu}_{coim(\xi)} \circ coim(\xi)$.
In fact, by Lemma \ref{coim_im_norm_integral_exist}, we have
\begin{align}
 \mu_\xi \circ \xi
&=  \left( \tilde{\mu}_{coim(\xi)} \circ \bar{\xi}^{-1} \circ \tilde{\mu}_{im(\xi)} \right) \circ \left( im(\xi) \circ \bar{\xi} \circ coim(\xi) \right) \\
&=
\tilde{\mu}_{coim(\xi)} \circ \bar{\xi}^{-1} \circ \bar{\xi} \circ coim(\xi) \\
&= 
\tilde{\mu}_{coim(\xi)} \circ coim(\xi) 
\end{align}

We prove that the integral $\mu_\xi$ is normalized, i.e. $\xi \circ \mu_\xi \circ \xi = \xi$.
By Lemma \ref{coim_im_norm_integral_exist}, we have $\tilde{\mu}_{coim(\xi)} \circ coim(\xi)  = L_{\alpha^{\to}_{ker(\xi)}} ( \sigma_{Ker(\xi)} )$.
Then the claim $\xi \circ \mu_\xi \circ \xi = \xi$ follows from Figure \ref{201912072240} where we put $j = ker(\xi)$.

\begin{figure}[ht]
  \includegraphics[width=10cm]{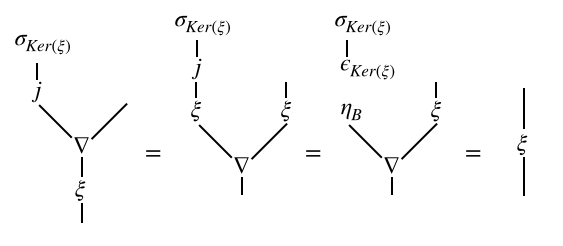}
  \caption{}
  \label{201912072240}
\end{figure}

We prove that the integral $\mu_\xi$ is a generator.
We first prove that $\mu_\xi \circ \xi \circ \mu = \mu$ for any $\mu \in Int_l (\xi) \cup Int_r (\xi)$.
By Lemma \ref{coim_im_norm_integral_exist}, we have $\tilde{\mu}_{coim(\xi)} \circ coim(\xi)  = R_{\alpha^{\leftarrow}_{ker(\xi)}} ( \sigma_{Ker(\xi)} )$.
We obtain $\mu_\xi \circ \xi \circ \mu = \mu$ for arbitrary $\mu \in Int_l (\xi)$ from Figure \ref{201912072248} where we put $j = ker(\xi)$.
Analogously, we prove that $\mu_\xi \circ \xi \circ \mu = \mu$ for arbitrary $\mu \in Int_r (\xi)$ by using $\tilde{\mu}_{coim(\xi)} \circ coim(\xi)  = L_{\alpha^{\to}_{ker(\xi)}} ( \sigma_{Ker(\xi)} )$ in Lemma \ref{coim_im_norm_integral_exist}.

\begin{figure}[ht]
  \includegraphics[width=10cm]{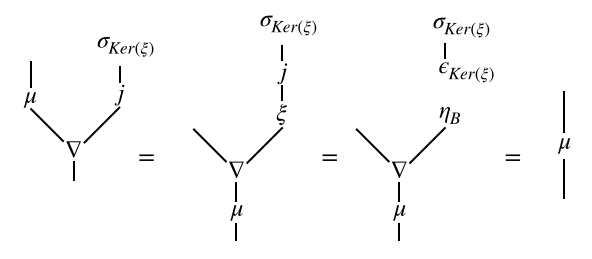}
  \caption{}
  \label{201912072248}
\end{figure}

All that remain is to prove that $\mu \circ \xi \circ \mu_\xi = \mu$ for any $\mu \in Int_l (\xi) \cup Int_r (\xi)$.
Note that we have $\xi \circ \mu_\xi = im(\xi) \circ \tilde{\mu}_{im(\xi)}$ by Lemma \ref{coim_im_norm_integral_exist}.
We prove that $\mu \circ im(\xi) \circ \tilde{\mu}_{im(\xi)}= \mu$ for arbitrary $\mu \in Int_l (\xi)$.
By Lemma \ref{coim_im_norm_integral_exist}, we have $im(\xi) \circ \tilde{\mu}_{im(\xi)} = R^{\beta^{\leftarrow}_{cok(\xi)}} ( \sigma^{Cok(\xi)} )$.
Then the claim $\mu \circ im(\xi) \circ \tilde{\mu}_{im(\xi)} = \mu$ follows from Figure \ref{201912072305}.
Analogously, we prove that $\mu \circ im(\xi) \circ \tilde{\mu}_{im(\xi)}= \mu$ for arbitrary $\mu \in Int_r (\xi)$ by using $im(\xi) \circ \tilde{\mu}_{im(\xi)} = L^{\beta^{\to}_{cok(\xi)}} ( \sigma^{Cok(\xi)} )$ in Lemma \ref{coim_im_norm_integral_exist}.
It completes the proof.

\begin{figure}[ht]
  \includegraphics[width=10cm]{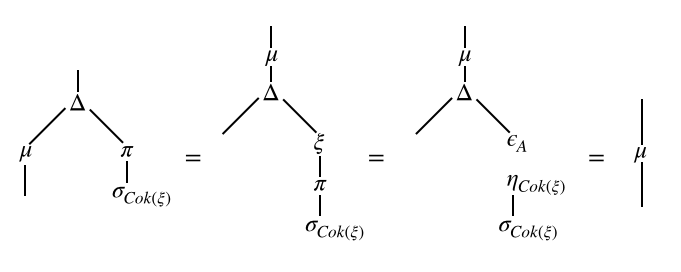}
  \caption{}
  \label{201912072305}
\end{figure}
\end{proof}

\begin{Corollary}
\label{201912050852}
Let $A,B$ be bimonoids in a symmetric monoidal category $\mathcal{C}$ and $\xi : A \to B$ be a weakly well-decomposable homomorphism.
If the homomorphism $\xi$ is weakly pre-Fredholm, then there exists a unique normalized generator integral $\mu_\xi : B \to A$ along $\xi$.
\end{Corollary}
\begin{proof}
The existence follows from Theorem \ref{201907292156} and the uniqueness follows from Proposition \ref{r_integral_l_integra_coincide_along}.
\end{proof}

\begin{Corollary}
\label{202002211035}
Suppose that every idempotent in $\mathcal{C}$ is a split idempotent.
Let $\xi$ be a well-decomposable bimonoid homomorphism.
There exists a normalized generator integral $\mu_\xi$ along $\xi$ if and only if the homomorphism $\xi$ is weakly pre-Fredholm.
Note that if a normalized integral exists, then it is unique.
\end{Corollary}
\begin{proof}
It is immediate from Theorem \ref{small_integral_equiv}, \ref{201912051457}, and Corollary \ref{201912050852}.
\end{proof}

Corollary \ref{202002211035} implies Theorem \ref{202002211033} by definitions of pre-Fredholmness and Corollary \ref{201912021909}.
Note that (Assumption 0) implies that every idempotent in $\mathcal{C}$ is a split idempotent.

On the other hand, Corollary \ref{202002211035} is a strict generalization of Theorem \ref{202002211033} as the following example describes.

\begin{Example}
\label{202010121400}
Let $G,H$ be (possibly infinite and non-abelian) groups and $\varrho : G \to H$ be a group homomorphism.
Put $A = k G$, $B = k H$, $\xi = \varrho_\ast$.
We give a condition for $\xi$ to have a normalized generator integral.
Note that since $G,H$ could not be abelian, $A,B$ are not necessarily bicommutative so that we can not apply Theorem \ref{202002211033}.
If the image of $\varrho$ is a normal subgroup in $H$., then $\xi$ is a well-decomposable bialgebra homomorphism by definitions.
The homomorphism $\xi$ is weakly pre-Fredholm if and only if the kernel and cokernel of $\varrho$ are finite groups whose orders are coprime to the characteristic of the field $k$.
By Corollary \ref{202002211035}, it is equivalent with the existence of a normalized generator integral along $\xi$.
Concretely, if such assumptions are satisfied, the normalized generator integral $\mu_\xi$ is given by $\mu_\xi ( h ) = |Ker (\varrho ) |^{-1} \sum_{\varrho ( g ) = h} g$.
\end{Example}


\section{Commutativity of homomorphisms and integrals}
\label{201908051617}

In this subsection, we study a commutativity of some bimonoid homomorphisms and integrals.
We prove the following theorem in this section.

\begin{theorem}
\label{key_theorem_existence_integral}
Let $A,B, C, D$ be bimonoids.
Consider a commutative diagram (\ref{202002271136}) of bimonoid homomorphisms.
Suppose that the bimonoid homomorphisms $\varphi$, $\psi$ are weakly well-decomposable and weakly pre-Fredholm.
Recall that there exist normalized generator integrals $\mu_\varphi, \mu_\psi$ along $\varphi,\psi$ respectively by Corollary \ref{201912050852}.
If the following conditions hold, then we have $\mu_{\psi} \circ \psi^\prime = \varphi^\prime \circ \mu_\varphi$.
\begin{itemize}
\item[(a)]
The induced bimonoid homomorphism $\varphi^\prime_0 : Ker ( \varphi ) \to Ker ( \psi )$ has a section in $\mathcal{C}$.
In other words, there is a morphism $s : Ker ( \psi ) \to Ker ( \varphi )$ in $\mathcal{C}$ such that $\varphi^\prime_0 \circ s = id$.
\item[(b)]
The induced bimonoid homomorphism $\psi^\prime_0 : Cok ( \varphi ) \to Cok ( \psi )$ has a retract in $\mathcal{C}$.
In other words, there exists a morphism $r : Cok ( \psi ) \to Cok ( \varphi )$ in $\mathcal{C}$ such that $r \circ \psi^\prime_0 = id$.
\end{itemize}
\begin{equation}
\label{202002271136}
\begin{tikzcd}
A \ar[r, "\varphi^\prime"] \ar[d, "\varphi"] & C \ar[d, "\psi"] \\
B \ar[r, "\psi^\prime"] & D
\end{tikzcd}
\end{equation}
\end{theorem}

\begin{remark}
It is sufficient that the section $s$ and the retract $r$ in Theorem \ref{key_theorem_existence_integral} are morphisms in $\mathcal{C}$, not bimonoid homomorphisms.
\end{remark}

\begin{remark}
We give a remark about assumptions (a), (b) in Theorem \ref{key_theorem_existence_integral}.
Suppose that the symmetric monoidal category $\mathcal{C}$ satisfies (Assumption 0,1,2) in section \ref{202002201415}.
Consider bicommutative Hopf monoids $A,B,C,D$ and pre-Fredholm homomorphisms $\varphi,\psi$.
In particular, $Ker(\varphi),Ker(\psi),Cok(\varphi),Cok(\psi)$ are small and cosmall.
If the induced bimonoid homomorphism $\varphi^\prime_0$ is an epimorphism in $\mathsf{Hopf}^\mathsf{bc}(\mathcal{C})$, then the assumption (a) is immediate.
In fact, the normalized generator integral along the homomorphism $\varphi^\prime_0$, which exists due to Corollary \ref{201912050852}, is a section of $\varphi^\prime_0$.
See Lemma.
Dually, if the induced bimonoid homomorphism $\psi^\prime_0$ is a monomorphism in $\mathsf{Hopf}^\mathsf{bc}(\mathcal{C})$, then the assumption (b) is immediate.
Especially, by (Assumption 2), the conditions (a), (b) are equivalent with an exactness of the induced chain complex below where $(\varphi,\varphi^\prime ) = ( \varphi \otimes \varphi^\prime) \circ \Delta_A$ and $\psi^\prime - \psi = \nabla_D \circ \left(\psi^\prime\otimes (S_C \circ \psi ) \right)$ :
\begin{align}
A
\stackrel{(\varphi , \varphi^\prime)}{\longrightarrow} 
B \otimes C
\stackrel{\psi^\prime - \psi}{\longrightarrow} 
D
\end{align}
\end{remark}

\begin{Lemma}
\label{201907290857}
Consider the following commutative diagram of bimonoid homomorphisms.
Suppose that $\varphi, \psi$ are weakly well-decomposable and weakly pre-Fredholm.
$$
\begin{tikzcd}
A \ar[r, "\varphi^\prime"] \ar[d, "\varphi"] & C \ar[d, "\psi"] \\
B \ar[r, "\psi^\prime"] & D
\end{tikzcd}
$$
Then we have $\psi \circ \left( \varphi^\prime \circ \mu_\varphi \right) \circ \varphi  = \psi \circ \left( \mu_\psi \circ \psi^\prime \right) \circ \varphi$.
In particular, if $\varphi$ is an epimorphism in $\mathcal{C}$ and $\psi$ is a monomorphism in $\mathcal{C}$, then $\varphi^\prime \circ \mu_\varphi =\mu_\psi \circ \psi^\prime$.
\end{Lemma}
\begin{proof}
Since $\mu_\varphi$ is normalized, we have,
\begin{align}
\psi \circ  \varphi^\prime \circ \mu_\varphi  \circ \varphi 
&= 
\psi^\prime \circ \varphi \circ \mu_\varphi \circ \varphi  \\
&=
\psi^\prime \circ \varphi . 
\end{align}
Since $\mu_\psi$ is normalized, we have
\begin{align}
\psi \circ  \mu_\psi \circ \psi^\prime  \circ \varphi
&=
\psi \circ \mu_\psi \circ \psi \circ \varphi^\prime \\
&=
\psi \circ \varphi^\prime . 
\end{align}
It completes the proof.
\end{proof}

{\it Proof of Theorem \ref{key_theorem_existence_integral}}
By Theorem \ref{201907292156}, the morphisms $\mu_\varphi, \mu_\psi$ in Definition \ref{201907312115} are the normalized generator integrals.
Note that the homomorphisms in the above diagram are decomposed into the following diagram.
$$
\begin{tikzcd}
A \ar[r, "\varphi^\prime"] \ar[d, shift left, "coim(\varphi)"] \ar[dr, "\varphi^{\prime\prime}"] & C \ar[d, shift left, "coim(\psi)"] \\
Coim (\varphi )  \ar[d, "\bar{\varphi}"] \ar[u, shift left, "\tilde{\mu}_{coim(\varphi)}"] & Coim (\psi) \ar[d, "\bar{\psi}"] \ar[u, shift left, "\tilde{\mu}_{coim(\psi)}"] \\
Im (\varphi) \ar[dr, "\psi^{\prime\prime}"'] \ar[d, shift left, "im(\varphi)"] & Im (\psi) \ar[d, shift left, "im(\psi)"] \\
B \ar[r, "\psi^\prime"'] \ar[u, shift left, "\tilde{\mu}_{im(\varphi)}"] & D \ar[u, shift left,  "\tilde{\mu}_{im(\psi)}"]
\end{tikzcd}
$$
By Lemma \ref{201907290857}, we have $\varphi^{\prime\prime} \circ \tilde{\mu}_{coim (\varphi)} \circ \bar{\varphi}^{-1} = \bar{\psi}^{-1} \circ \tilde{\mu}_{im (\psi)} \circ \psi^{\prime\prime}$.
Here, we use the fact that $coim(\varphi)$ is an epimorphism in $\mathcal{C}$ and $im(\psi)$ is a monomorphism in $\mathcal{C}$ by Lemma \ref{coim_im_norm_integral_exist}.
Thus, we have $coim (\psi) \circ \varphi^\prime \circ \tilde{\mu}_{coim (\varphi)} \circ \bar{\varphi}^{-1} =\bar{\psi}^{-1} \circ \tilde{\mu}_{im (\psi)} \circ \psi^{\prime} \circ im(\varphi)$.

We claim that 
\begin{enumerate}
\item
$\tilde{\mu}_{coim(\psi)} \circ coim (\psi) \circ \varphi^\prime \circ \tilde{\mu}_{coim(\varphi)} = \varphi^\prime \circ \tilde{\mu}_{coim(\varphi)} $.
\item
$\tilde{\mu}_{im(\psi)} \circ \psi^\prime \circ im(\varphi) \circ \tilde{\mu}_{im(\varphi)} = \tilde{\mu}_{im(\psi)} \circ \psi^\prime $. 
\end{enumerate}
By these claims, we have
\begin{align}
\mu_{\psi} \circ \psi^\prime
&=
\tilde{\mu}_{coim(\psi)} \circ \bar{\psi}^{-1} \circ \tilde{\mu}_{im(\psi)} \circ \psi^\prime  \\
&=
\tilde{\mu}_{coim(\psi)} \circ \bar{\psi}^{-1} \circ \tilde{\mu}_{im(\psi)} \circ \psi^\prime \circ im(\varphi) \circ \tilde{\mu}_{im(\varphi)}  \\
&=
\tilde{\mu}_{coim(\psi)} \circ coim (\psi) \circ \varphi^\prime \circ \tilde{\mu}_{coim (\varphi)} \circ \bar{\varphi}^{-1} \circ \tilde{\mu}_{im(\varphi)}  \\
&=
\varphi^\prime \circ \tilde{\mu}_{coim (\varphi)} \circ \bar{\varphi}^{-1} \circ \tilde{\mu}_{im(\varphi)}  \\
&=
\varphi^\prime \circ \mu_{\varphi} . 
\end{align}
It suffices to prove the above claims.

From now on, we show the first claim.
We use the hypothesis to prove $\varphi^\prime \circ ker (\varphi) \circ \sigma_{Ker(\varphi)} = ker (\psi ) \circ \sigma_{Ker(\psi)}$.
Since $\varphi^\prime_0 = \varphi^\prime|_{Ker ( \varphi )} : Ker ( \varphi ) \to Ker (\psi)$ has a section in $\mathcal{C}$, we have $\varphi^\prime_0 \circ \sigma_{Ker(\varphi)} =\sigma_{Ker(\psi)}$ by Lemma \ref{201908020749}.
Hence, we obtain $\varphi^\prime \circ ker (\varphi) \circ \sigma_{Ker(\varphi)} = ker (\psi) \circ \varphi^\prime_0 \circ \sigma_{Ker(\varphi)} = ker (\psi ) \circ \sigma_{Ker(\psi)}$.

Recall that $\tilde{\mu}_{coim(\psi)} \circ coim (\psi) : C \to C$ coincides with the action by $ker (\psi ) \circ \sigma_{Ker(\psi)} : \mathds{1} \to C$ by Lemma \ref{coim_im_norm_integral_exist}.
Then Figure \ref{201907021012} completes the proof of the first claim.
\begin{figure}[ht]
  \includegraphics[width=\linewidth]{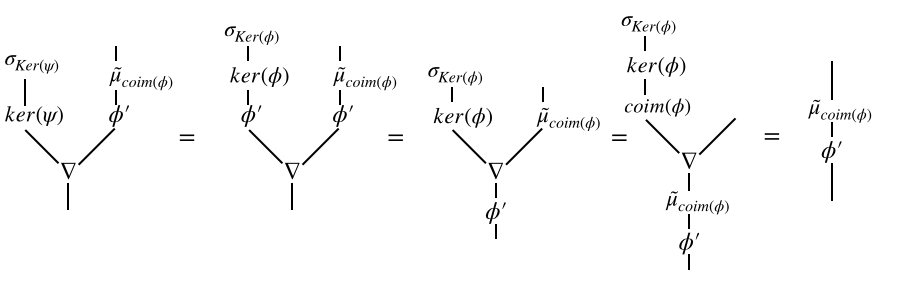}
  \caption{}
  \label{201907021012}
\end{figure}

Dually we can prove the second claim.
Here, we use the section of $\psi^\prime_0 : Cok ( \varphi ) \to Cok ( \psi )$ and apply Lemma \ref{201908020749} again.
It completes the proof.


\section{Inverse volume}

\subsection{Inverse volume of bimonoid}
\label{201908051620}

In this subsection, we introduce a notion of {\it inverse volume} $vol^{-1} (A)$ of a bimonoid $A$ with a normalized integral and a normalized cointegral.
It gives an invariant of such bimonoids by Proposition \ref{201908010934}.
By Remark \ref{201912021229}, it  defines an invariant of bismall bimonoids.

\begin{Defn}
\label{201907311409}
\rm
Let $A$ be a bimonoid with a normalized integral $\sigma_A : \mathds{1} \to A$ and a normalized cointegral $\sigma^A : A \to \mathds{1}$.
An {\it inverse volume} of the bimonoid $A$ is an endomorphism $vol^{-1} (A) : \mathds{1} \to \mathds{1}$ in $\mathcal{C}$, defined by a compostiion,
\begin{align}
vol^{-1} (A)  \stackrel{\mathrm{def.}}{=} \sigma^A \circ \sigma_A . 
\end{align}
\end{Defn}

\begin{Defn}
\label{201912041843}
\rm
A bimonoid $A$ {\it has a finite volume} if $A$ has a normalized integral and a normalized cointegral, and its inverse volume $vol^{-1} ( A) : \mathds{1} \to \mathds{1}$ is invertible.
\end{Defn}

\begin{Example}
Consider the symmetric monoidal category, $\mathcal{C} = \mathsf{Vec}^{\otimes}_{{k}}$.
Let $G$ be a finite group.
Suppose that the characteristic of ${k}$ is not a divisor of the order $\sharp G$ of $G$.
Then the induced Hopf monoid $A= {k} G$ in $\mathsf{Vec}^{\otimes}_{{k}}$ has a normalized integral $\sigma_A$ and a normalized cointegral $\sigma^A$.
In particular, 
\begin{align}
\sigma_A  &: {k} \to {k} G ~;~ 1 \mapsto (\sharp G)^{-1} \sum_{g \in G} g ,  \\
\sigma^A &: {k}G  \to {k}  ~;~ g \mapsto \delta_e (g) , 
\end{align}
give a normalized integral and a normalized cointegral of $A={k}G$ respectively..
Then we have 
\begin{align}
vol^{-1} ( {k}(G)) : {k} \to {k} ~;~ 1 \mapsto (\sharp G)^{-1}. 
\end{align}
\end{Example}

\begin{prop}
\label{201908010934}
Let $A,B$ be bimonoids with a normalized integral and a normalized cointegral.
\begin{itemize}
\item
For the unit bimonoid, we have $vol^{-1} ( \mathds{1}) = id_{\mathds{1}}$.
\item
A bimonoid isomorphism $A \cong B$ implies $vol^{-1} (A) = vol^{-1} (B)$.
\item
$vol^{-1} (A \otimes B) = vol^{-1} (A) \circ vol^{-1} (B) = vol^{-1}(B) \circ vol^{-1} (A)$.
\item
If $A^\vee$ is a dual bimonoid of the bimonoid $A$, then the bimonoid $A^\vee$ has a normalized integral and a normalized cointegral and we have
\begin{align}
vol^{-1}(A^\vee) = vol^{-1} (A) .
\end{align}
\end{itemize}
\end{prop}
\begin{proof}
Since $\sigma_{\mathds{1}} = \sigma^{\mathds{1}} = id_{\mathds{1}}$, we have $vol^{-1} (\mathds{1}) = id_{\mathds{1}}$.

If $A \cong B$ as bimonoids, then their normalized (co)integrals coincide via that isomorphism due to their uniqueness.
Hence, we have $vol^{-1} (A) = \sigma^A \circ \sigma_A = \sigma^B \circ \sigma_B = vol^{-1} (B)$.

Since $\sigma_{A\otimes B} = \sigma_{A} \otimes \sigma_{B} : \mathds{1} \to A \otimes B$ and $\sigma^{A \otimes B} : \sigma^{A} \otimes \sigma^{B} : A \otimes B \to \mathds{1}$, we have $vol^{-1} (A \otimes B) = vol^{-1} (A) \ast vol^{-1} (B) = vol^{-1}(A) \circ vol^{-1} (B) = vol^{-1}(B) \circ vol^{-1} (A)$.

By direct calculations, the following morphisms give a normalized integral and a normalized cointegral on the dual bimonoid $A^\vee$ :
\begin{align}
\sigma_{A^\vee} &= \left( \mathds{1} \stackrel{coev_{A}}{\to} A^\vee \otimes A \stackrel{id_{A^\vee}\otimes \sigma^A}{\to} A^\vee \otimes \mathds{1} \cong A^\vee \right)  \\
\sigma^{A^\vee} &= \left( A^\vee \cong \mathds{1} \otimes A^\vee \stackrel{\sigma_A \otimes id_{A^\vee}}{\to} A \otimes A^\vee \stackrel{ev_A}{\to} \mathds{1} \right) 
\end{align}
It implies that $\sigma^{A^\vee} \circ \sigma_{A^\vee} = \sigma^A \circ \sigma_{A}$ since $\mathbf{l}_{A} \circ (ev_A \otimes id_A) \circ (id_A \otimes coev_A) \circ \mathbf{r}_{A} = id_A$ where $\mathbf{l}_A$ denotes the left unitor.
\end{proof}


\subsection{Inverse volume of homomorphisms}
\label{201912060955}

\begin{Defn}
\label{202002271106}
\rm
Let $A$ be a bimonoid with a normalized integral $\sigma_A$ and $B$ be a bimnoid with a normalized cointegral $\sigma^B$.
For a bimonoid homomorphism $\xi : A \to B$, we define a morphism $\langle \xi \rangle : \mathds{1} \to \mathds{1}$ by
\begin{align}
\langle \xi \rangle \stackrel{\mathrm{def.}}{=} \sigma^B \circ \xi \circ \sigma_A . 
\end{align}
\end{Defn}

\begin{remark}
Since $\langle id_A \rangle = vol^{-1} (A)$ by definitions, $\langle - \rangle$ is an extended notion of the inverse volume in Definition \ref{201907311409}.
On the other hand, for some special $\xi$, we can compute $\langle \xi \rangle$ from an inverse volume.
See Proposition \ref{201907312242}.
\end{remark}

\begin{Lemma}
\label{201908020749}
Let $A,B$ be bimonoids.
Let $\sigma_A$ be a normalized integral of $A$.
Let $\xi : A \to B$ be a bimonoid homomorphism.
If there exists a morphism $\xi^\prime : B \to A$ in $\mathcal{C}$ such that $\xi \circ \xi^\prime = id_A$, then $\xi \circ \sigma_A$ is a normalized integral of $B$.
\end{Lemma}
\begin{proof}
The morphism $\xi \circ \sigma_A : \mathds{1} \to B$ is a right integral due to Figure \ref{201908021108}.
It can be verified to be a left integral in a similar way.
Moreover, it is normalized since we have $\epsilon_{\xi} \circ \xi \circ \sigma_A = \epsilon_A \circ \sigma_A = id_{\mathds{1}}$.

\begin{figure}[ht]
  \includegraphics[width=\linewidth]{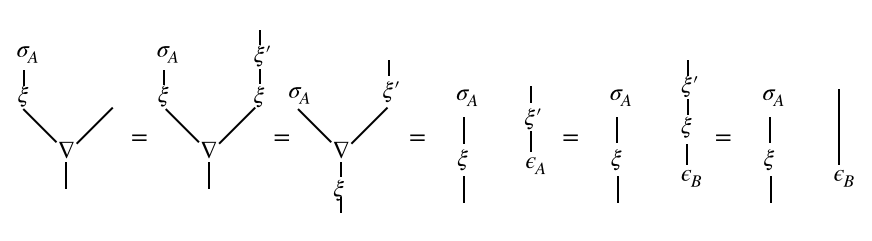}
  \caption{}
  \label{201908021108}
\end{figure}
\end{proof}

\begin{prop}
Let $\xi : A \to B$ be a bimonoid homomorphism.
Suppose that every idempotent in the symmetric monoidal category $\mathcal{C}$ is a split idempotent.
If the bimonoid $A$ is small and there exists a morphism $\xi^\prime : B \to A$ in $\mathcal{C}$ such that $\xi \circ \xi^\prime = id_A$, then the bimonoid $B$ is small.
\end{prop}
\begin{proof}
It is immediate from Lemma \ref{201908020749} and 
Theorem \ref{small_integral_equiv}.
\end{proof}

\begin{prop}
\label{201907312242}
Let $\xi :  A \to B$ be a bimonoid homomorphism.
Suppose that a kernel bimoniod $Ker (\xi)$, a cokernel bimonoid $Cok (\xi)$, a coimage bimonoid $Coim (\xi)$, an image bimonoid $Im (\xi)$ exist.
Suppose that $Ker (\xi)$ is small and $Cok (\xi)$ is cosmall.
Suppose that the canonical homomorphism $ker (\xi) : Ker (\xi) \to A$ is normal and $cok (\xi) : B \to Cok (\xi)$ is conormal.
Then for the canonical homomorphism $\bar{\xi}: Coim (\xi) \to Im (\xi)$, we have,
\begin{align}
\langle \xi \rangle = \langle \bar{\xi} \rangle .
\end{align}
In particular, if $\bar{\xi}$ is an isomorphism, then we have $\langle \xi \rangle = \langle \bar{\xi} \rangle = vol^{-1} ( Coim (\xi)) = vol^{-1}( Im (\xi))$.
\end{prop}
\begin{proof}
It suffices to prove that $\langle \xi \rangle = \langle \bar{\xi} \rangle$.
Since $\langle \xi \rangle = \sigma^B \circ \xi \circ \sigma_A = \sigma^B \circ im (\xi) \circ \bar{\xi} \circ coim (\xi) \circ \sigma_A$, it suffices to show that $coim (\xi) \circ \sigma_A = \sigma_{Coim (\xi)}$ and $\sigma^B \circ im (\xi) = \sigma^{Im(\xi)}$.
The morphism $coim(\xi)$ ($im(\xi)$, resp.) has a section (retract, resp.) in $\mathcal{C}$ by Lemma \ref{coim_im_norm_integral_exist}.
Hence, the compositions $coim (\xi) \circ \sigma_A$ ( $\sigma^B \circ im ( \xi )$, resp.) are normalized integrals by Lemma \ref{201908020749}.
It completes the proof.
\end{proof}


\section{Commutativity of integrals}
\label{201908051621}

In this section, we discuss a relation between two composable integrals and their composition.
Recall that integrals are preserved under compositions by Proposition \ref{201907311055}.
Nevertheless, such a composition does not preserve normalized integrals.
By considering normalized generator integrals rather than normalized integrals, one can deduce that they are preserved up to a {\it scalar}.
Here, a scalar formally means an endomorphism on the unit object $\mathds{1}$.

\begin{theorem}
\label{201907311424}
Let $A,B,C$ be bimonoids.
Let $\xi : A \to B$, $\xi^\prime : B \to C$ be bimonoid homomorphism.
Suppose that 
\begin{itemize}
\item
$\xi$ is normal, $\xi^\prime$ is conormal.
The composition $\xi^\prime\circ \xi$ is either conormal or normal.
\item
$\mu, \mu^\prime$ are normalized integrals along $\xi, \xi^\prime$ respectively.
$\mu^{\prime\prime}$ is a normalized integral along $\xi^\prime \circ \xi$, which is a generator.
\end{itemize}
Recall that the cokernel bimonoid $Cok (\xi)$ has a normalized cointegral and the kernel bimonoid $Ker (\xi^\prime )$ has a normalized integral by Theorem \ref{norm_inte_along_induces_norm_inte}.
Then we have,
\begin{align}
\mu \circ \mu^\prime = \langle cok (\xi) \circ ker(\xi^\prime) \rangle \cdot \mu^{\prime\prime} . 
\end{align}
\end{theorem}
\begin{proof}
By Proposition \ref{201907311055}, $\mu \circ \mu^\prime$ is an integral along the composition $\xi^\prime \circ \xi$.
By Theorem \ref{201906281559}, there exists a unique $\lambda\in End_{\mathcal{C}}(\mathds{1} )$ such that $\mu \circ \mu^\prime = \lambda \cdot \mu^{\prime\prime}$ since $\xi^\prime\circ \xi$ is either conormal or normal.

We have $\epsilon_A \circ \mu^{\prime\prime} \circ \eta_C = id_{\mathds{1}}$ due to the following computation :
\begin{align}
\epsilon_A \circ \mu^{\prime\prime} \circ \eta_C
&=
 \left( \epsilon_C \circ \xi^\prime \circ \xi \right) \circ \mu^{\prime\prime} \circ \left(  \xi^\prime \circ \xi \eta_A \right)  \\
&=
\epsilon_C \circ \left(  \xi^\prime \circ \xi  \circ \mu^{\prime\prime} \circ \xi^\prime \circ \xi \right) \circ \eta_A  \\
&=
\epsilon_C \circ \left(  \xi^\prime \circ \xi  \right) \circ \eta_A ~(\because \mu^{\prime\prime} : \mathrm{normalized} )   \\
&=
 id_{\mathds{1}}  
\end{align}
Hence it suffices to calculate $\epsilon_A \circ \mu \circ \mu^\prime \circ \eta_C$ to know $\lambda$.
Since $\xi^\prime$ is conormal, we have a morphism $\check{F}(\mu^\prime)$ such that $\mu^\prime \circ \eta_C = ker (\xi^\prime) \circ \check{F}(\mu^\prime)$ (see Definition \ref{201907311644}).
Since $\xi$ is normal, we have a morphism $\hat{F}(\mu)$ such that $\epsilon_A \circ \mu = \hat{F}(\mu) \circ cok (\xi)$.
Since the integrals $\mu, \mu^\prime$ are normalized, $\check{F}(\mu^\prime)$ and $\hat{F}(\mu) $ are normalized integrals by Theorem \ref{norm_inte_along_induces_norm_inte}.
By using our notations, $\check{F}(\mu^\prime) = \sigma_{Ker(\xi^\prime)}$ and $\hat{F}(\mu) = \sigma^{Cok(\xi)}$.
Therefore, we have  $\epsilon_A \circ \mu \circ \mu^\prime \circ \eta_C = \sigma^{Cok(\xi)} \circ cok (\xi) \circ ker (\xi^\prime) \circ \sigma_{Ker(\xi^\prime)} = \langle cok (\xi) \circ ker (\xi^\prime) \rangle$ by definitions.
It completes the proof.
\end{proof}

\begin{Corollary}
\label{201907311426}
Let $A,B$ be bimonoids and $\xi : A \to B$ be a bimonoid homomorphism.
Suppose that 
\begin{itemize}
\item
$\xi$ is normal.
\item
$\mu$ is a normalized integral along $\xi$, $\sigma_B$ is a normalized integral of $B$, and $\sigma_A$ is a normalized integral of $A$ which is a generator.
\end{itemize}
Then we have
\begin{align}
\mu \circ \sigma_B = vol^{-1} ( Cok (\xi)) \cdot \sigma_A .  
\end{align}

We have an analogous statement.
Suppose that 
\begin{itemize}
\item
$\xi$ is conormal.
\item
$\mu$ is a normalized integral along $\xi$, $\sigma^A$ is a normalized cointegral of $A$, and $\sigma^B$ is a normalized cointegral of $B$ which is a generator.
\end{itemize}
Then we have
\begin{align}
\sigma^A \circ \mu = vol^{-1} ( Ker(\xi)) \cdot \sigma^B .
\end{align}
\end{Corollary}
\begin{proof}
We prove the first claim.
We replace $\xi, \xi^\prime$ in Theorem \ref{201907311424} with $\xi, \epsilon_B$ in the above assumption.
Then the assumption in Theorem \ref{201907311424} is satisfied.

We prove the second claim.
We replace $\xi, \xi^\prime$ in Theorem \ref{201907311424} with $\eta_A, \xi$ in the above assumption.
Then the assumption in Theorem \ref{201907311424} is satisfied.
\end{proof}

\begin{Corollary}
\label{201907312239}
Let $A,B$ be bimonoids and $\xi : A \to B$ be a bimonoid homomorphism.
Suppose that
\begin{itemize}
\item
$\xi$ is binormal.
\item
There exists a normalized integral along $\xi$.
\item
$A$, $B$ are bismall
\item 
The normalized integral $\sigma_A$ of $A$ is a generator.
The normalized cointegral $\sigma^B$ of $B$ is a generator.
\end{itemize} 
Then we have
\begin{align}
\label{201907311428}
vol^{-1}(Cok (\xi)) \circ vol^{-1} (A) = vol^{-1} (Ker (\xi)) \circ vol^{-1} (B) . 
\end{align}
\end{Corollary}
\begin{proof}
Since $A,B$ are bismall, the counit $\epsilon_A$ and the unit $\eta_B$ are pre-Fredholm.
Since the counit $\epsilon_A$ and the unit $\eta_B$ are well-decomposable, the normalized integral $\sigma_A$ of $A$ and te normalized cointegral $\sigma^B$ of $B$ are generators by Theorem \ref{201907292156}.
Hence, the assumptions in Corollary \ref{201907311426} are satisfied.
By Corollary \ref{201907311426}, we obtain 
\begin{align}
\mu_\xi \circ \sigma_B &= vol^{-1} ( Cok (\xi)) \cdot \sigma_A ,  \\
\sigma^A \circ \mu_\xi &= vol^{-1} ( Ker(\xi)) \cdot \sigma^B .  
\end{align}
Hence, we obtain $vol^{-1} ( Cok (\xi)) \cdot \sigma^A \circ \sigma_A = vol^{-1} ( Ker (\xi)) \cdot \sigma^B \circ \sigma_B$, which is equivalent with (\ref{201907311428}).
\end{proof}

\begin{Corollary}
\label{201912050859}
Let $A,B,C$ be bimonoids.
Let $\xi : A \to B$, $\xi^\prime : B \to C$ be bimonoid homomorphisms.
Suppose that the homomorphisms $\xi,\xi^\prime, \xi^\prime\circ\xi$ are well-decomposable and weakly pre-Fredholm.
Recall that there exist normalized generator integrals $\mu_\xi , \mu_{\xi^\prime}, \mu_{\xi^\prime\circ\xi}$ along the bimonoid homomorphisms $\xi,\xi^\prime,\xi^\prime\circ\xi$ respectively by Corollary \ref{201912050852}.
Then there exists a unique $\lambda \in End_\mathcal{C} ( \mathds{1} )$ such that 
\begin{align}
\mu_\xi \circ \mu_{\xi^\prime} =  \lambda \cdot \mu_{\xi^\prime\circ\xi} . 
\end{align}
\end{Corollary}
\begin{proof}
It is a corollary of Theorem \ref{201907311424}.
Since $\xi, \xi^\prime , \xi^\prime \circ \xi$ are well-decomposable, in particular weakly well-decomposable, and weakly pre-Fredholm, we obtain normalized generator integrals $\mu_\xi , \mu_{\xi^\prime}, \mu_{\xi^\prime\circ\xi}$ by Theorem \ref{201907292156}.
Since $\xi, \xi^\prime , \xi^\prime \circ \xi$ are well-decomposable, they satisfy the first assumption in Theorem \ref{201907311424}.
By Theorem Theorem \ref{201907292156}, the integrals $\mu= \mu_\xi ,\mu^\prime = \mu_{\xi^\prime}, \mu^{\prime\prime} = \mu_{\xi^\prime\circ\xi}$ satisfy the second assumption in Theorem \ref{201907311424}.
\end{proof}

The existence follows from Theorem \ref{201907311424} and the uniqueness follows form Theorem \ref{201906281559}.
In fact, the endomorphism $\lambda$ coincides with $\langle cok (\xi) \circ ker(\xi^\prime) \rangle \in End_\mathcal{C} ( \mathds{1} )$.
The symbol $\langle - \rangle$ represents an invariant of bimonoid homomorphisms from a bimonoid with a normalized integral to a bimonoid with a normalized cointegral (see Definition \ref{202002271106}).
In Corollary $\ref{201912050859}$, the kernel bimonoid $Ker(\xi^\prime)$ has a normalized integral and the cokernel bimonoid $Cok ( \xi)$ has a normalized cointegral since we assume that $\xi,\xi^\prime$ are weakly pre-Fredholm.

\section{Induced bismallness}

In this section, we assume that every idempotent in a symmetric monoidal category $\mathcal{C}$ is a split idempotent.


\subsection{Bismallness of (co)kernels}
\label{201908041309}

In this subsection, we give some conditions where $Ker(\xi)$, $Cok (\xi)$ inherits a (co)smallness from that of the domain and the target of $\xi$.

\begin{prop}
\label{condition_for_ker_cok_small}
Let $\xi : A \to B$ be a bimonoid homomorphism.
Suppose that $A$ is small, $B$ is cosmall.
If $\xi$ is normal, then $Cok (\xi)$ is cosmall.
If $\xi$ is conormal, then $Ker(\xi)$ is small.
\end{prop}
\begin{proof}
We only prove the first claim.
Let $\xi$ be normal.
We have $Cok(\xi) = \alpha^{\to}_\xi \backslash B$.
There exists a normalized cointegral of $B$ since $B$ is cosmall by Corollary \ref{201912021909}.
We denote it by $\sigma^B : B \to \mathds{1}$.
Put $\sigma = \sigma^B \circ \tilde{\mu}_{cok(\xi)} : Cok(\xi) = \alpha^{\to}_\xi \backslash B \to \mathds{1}$.
Note that $\sigma \in Int_r (\eta_{\alpha^{\to}_\xi \backslash B})$ due to Proposition \ref{201907311055}.
In other words, $\sigma$ is a right cointegral of $Cok(\xi) = \alpha^{\to}_\xi \backslash B$.

We prove that $\sigma$ is normalized.
Let $\pi : B \to \alpha^{\to}_\xi \backslash B$ be the canonical morphism.
We have $\sigma \circ \eta_{\alpha^{\to}_\xi \backslash B} = \sigma^B \circ \tilde{\mu}_{cok(\xi)} \circ \eta_{\alpha^{\to}_\xi \backslash B} = \sigma^B \circ \tilde{\mu}_{cok(\xi)} \circ \pi \circ \eta_B$.
$\sigma \circ \eta_{\alpha^{\to}_\xi \backslash B} = id_{\mathds{1}}$ follows from $\tilde{\mu}_{cok(\xi)} \circ \pi =  L_{\alpha^{\to}_\xi} ( \sigma_A )$ in Lemma \ref{ker_cok_integral_exist} (1), and $\epsilon_A \circ \sigma_A = id_{\mathds{1}}$.
Hence, $\sigma$ is a normalized right cointegral of $\alpha^{\to}_\xi \backslash B = Cok (\xi)$.

Analogously, we use $Cok(\xi) = B/ \alpha^{\leftarrow}_\xi$ to verify an existence of a normalized left cointegral of $Cok(\xi)$.
By Proposition \ref{norm_r_l_integral_is_integral}, the cokernel $Cok(\xi)$ has a normalized cointegral.
By Corollary \ref{201912021909}, the cokernel bimonoid $Cok (\xi)$ is cosmall.
\end{proof}

\begin{prop}
\label{201907312056}
Let $A,B$ be bimonoids.
Let $\xi : A \to B$ be a bimonoid homomorphism.
If $A$, $B$ are small and $\xi$ is normal, then $Cok (\xi)$ is small.
If $A$, $B$ are cosmall and $\xi$ is conormal, then $Ker (\xi)$ is cosmall.
\end{prop}
\begin{proof}
We only prove the first claim.
The small bimonoid $B$ has a unique normalized integral $\sigma_B : \mathds{1} \to B$ by Corollary \ref{201912021909}.
By Definition \ref{defn_integral_cok_ker}, a normalized integral $\tilde{\mu}_{cok(\xi)} \in Int ( cok (\xi))$ exists.
By Lemma \ref{201907311900}, $\tilde{\mu}_{cok(\xi)}$ is a section of $cok (\xi)$ in $\mathcal{C}$.
By Lemma \ref{201908020749}, $cok (\xi) \circ \sigma_B$ is a normalized integral of $Cok (\xi)$.
By Corollary \ref{201912021909}, $Cok (\xi)$ is small.
\end{proof}

\begin{Corollary}
\label{201908011522}
Let $A,B$ be bimonoids.
Let $\xi : A \to B$ be a well-decomposable homomorphism.
If $A$ is small and $B$ is cosmall, then the homomorphism $\xi$ is weakly pre-Fredholm.
If both of $A$, $B$ are bismall, then the homomorphism $\xi$ is pre-Fredholm.
\end{Corollary}
\begin{proof}
Suppose that $A$ is a small bimonoid and $B$ is a cosmall bimonoid.
Since $\xi$ is well-decomposable, the cokernel bimonoid $Cok (\xi)$ is cosmall and the kernel biomonoid $Ker(\xi)$ is small by Proposition \ref{condition_for_ker_cok_small}.

Suppose that both of $A,B$ are bismall bimonoids.
Then the homomorphism $\xi$ is weakly pre-Fredholm by the above discussion.
Moreover, the cokernel bimonoid $Cok (\xi)$ is small and kernel bimonoid $Ker(\xi)$ is cosmall by Proposition \ref{201907312056}.
\end{proof}

\subsection{Bismallness of bimonoids in exact sequences}
\label{201908051618}

In this subsection, we discuss some conditions for (co)smallness of a bimonoid to be inherited from an exact sequence.

\begin{Lemma}
\label{201907311220}
Let $A,B,C$ be bimonoids.
Let $\iota : B \to A$ be a normal homomorphism and $\pi :A \to C$ be a homomorphism.
Suppose that the following sequence is exact :
\begin{align}
B \stackrel{\iota}{\to} A \stackrel{\pi}{\to} C \to \mathds{1}
\end{align}
Here, the exactness means that $\pi \circ \iota$ is trivial and the induced homomorphism $Cok (\iota) \to C$ is an isomorphism.
If the bimonoids $B$, $C$ are small, then $A$ is small.
\end{Lemma}
\begin{proof}
It suffices to prove that $A$ has a normalized integral by Corollary \ref{201912021909}.
We denote by $\sigma_C$ the normalized integral of $C$.
Since $B$ is small and $\iota$ is normal, we have a normalized integral $\tilde{\mu}_{cok(\iota)}$ along $cok(\iota)$ (see Definition \ref{defn_integral_cok_ker}).
Since the induced homomorphism $Cok (\iota) \to C$ is isomorphism by the assumption, we have a normalized integral $\tilde{\mu}_{\pi}$ along $\pi$.
Then the composition $\tilde{\mu}_{\pi} \circ \sigma_C : \mathds{1} \to A$ gives an integral of $A$ by Proposition \ref{201907311055}.
Moreover $\tilde{\mu}_{\pi} \circ \sigma_C $ is normalized since $\epsilon_A \circ \tilde{\mu}_{\pi} \circ \sigma_C = \epsilon_C \circ \pi \circ \tilde{\mu}_{\pi} \circ \sigma_C = \epsilon_C \circ \sigma_C = id_{\mathds{1}}$ by Lemma \ref{ker_cok_integral_exist}.
It completes the proof.
\end{proof}

\begin{prop}
\label{201907311231}
Let $A,B,C,C^\prime$ be bimonoids.
Let $\iota : B \to A$ be a normal homomorphism, $\pi^\prime : C \to C^\prime$ be a conormal homomorphism and $\pi : A \to C$ be a homomorphism.
Suppose that the following sequence is exact :
\begin{align}
B \stackrel{\iota}{\to} A \stackrel{\pi}{\to} C \stackrel{\pi^\prime}{\to} C^\prime 
\end{align}
Suppose that $Cok ( \iota) \to Ker (\pi^\prime)$ is an isomorphism.
If the bimonoids $B,C$ are small and the bimonoid $C^\prime$ is cosmall, then the bimonoid $A$ is small.
\end{prop}
\begin{proof}
By the assumption, we obtain an exact sequence in the sense of Lemma \ref{201907311220},
\begin{align}
B \stackrel{\iota}{\to} A \stackrel{\bar{\pi}}{\to} Ker ( \pi^\prime ) \to \mathds{1} .
\end{align}
Note that $Ker (\pi^\prime)$ is small by Proposition \ref{condition_for_ker_cok_small}.
Since $\iota$ is normal and $B, Ker (\pi^\prime)$ are small, the bimonoid $A$ is small due to Lemma \ref{201907311220}.
\end{proof}

We have dual statements as follows.
For convenience of the readers, we give them without proof.

\begin{Lemma}
Let $A,B,C$ be bimonoids.
Let $\iota : B \to A$ be a homomorphism and $\pi :A \to C$ be a conormal homomorphism.
Suppose that the following sequence is exact.
\begin{align}
\mathds{1} \to B \stackrel{\iota}{\to} A \stackrel{\pi}{\to} C 
\end{align}
Here, the exactness means that $\pi \circ \iota$ is trivial and the induced morphism $B \to Ker (\xi)$ is an isomorphism.
If $\pi$ is conormal and the bimonoids $B$, $C$ are cosmall, then $A$ is cosmall.
\end{Lemma}

\begin{prop}
\label{201907311230}
Let $A,B,B^\prime,C$ be bimonoids.
Let $\iota^\prime : B^\prime \to B$ be a normal homomorphism, $\pi :A \to C$ be a conormal homomorphism, and $\iota : B \to A$ be a homomorphism.
Suppose that the following sequence is exact.
\begin{align}
B^\prime \stackrel{\iota^\prime}{\to} B \stackrel{\iota}{\to} A \stackrel{\pi}{\to} C 
\end{align}
Suppose that $Cok ( \iota^\prime ) \to Ker (\pi)$ is an isomorphism.
If the bimonoid $B^\prime$ are small and the bimonoids $B,C$ is cosmall, then the bimonoid $A$ is cosmall.
\end{prop}

\begin{theorem}
\label{201907312238}
Let $A,C$ be bismall bicommutative Hopf monoids in $\mathcal{C}$ and $B$ be an arbitrary bicommutative bimonoid.
If there exists an exact sequence $\mathds{1} \to A \to B \to C \to \mathds{1}$ of bicommutative Hopf monoids, then $B$ is bismall.
\end{theorem}
\begin{proof}
Consider an exact sequence in $\mathsf{Hopf}^\mathsf{bc} (\mathcal{C})$ where $B^\prime = \mathds{1} = C^\prime$.
\begin{align}
B^\prime \stackrel{\iota^\prime}{\to} B \stackrel{\iota}{\to} A \stackrel{\pi}{\to} C \stackrel{\pi^\prime}{\to} C^\prime
\end{align}
By Proposition \ref{stab_gives_cokernel}, any morphism in $\mathsf{Hopf}^\mathsf{bc} ( \mathcal{C})$ is binormal.
By Corollary \ref{201908040952}, a cokernel (kernel, resp.) as a bimonoid is a cokernel (cokernel, resp.) as a bicommutative Hopf monoid.
Hence, the assumptions in Proposition \ref{201907311231}, \ref{201907311230} are deduced from the assumption in the statement.
By Proposition \ref{201907311231}, \ref{201907311230}, we obtain the result.
\end{proof}

\section{Volume on abelian category}
\label{201912022325}

In this section, we introduce and study a notion of {\it volume on an abelian category}.

\subsection{Basic properties}

\begin{Defn}
\label{201911231935}
\rm
For an abelian monoid $M$\footnote{The reason that we consider a monoid $M$, not a group is that we deal with {\it infinite dimension} or {\it infinite order} uniformly.}, {\it an $M$-valued volume on the abelian category $\mathcal{A}$} is an assignment of $v(A) \in M$ to an object $A$ of $\mathcal{A}$ which satisfies
\begin{enumerate}
\item
For a zero object $0$ of $\mathcal{A}$, the corresponding element $v(0) \in M$ is the unit $1$ of the abelian monoid $M$.
\item
For an exact sequence $0 \to A \to B \to C \to 0$ in $\mathcal{A}$, we have $v(B) = v(A) \cdot v(C)$.
\end{enumerate}
\end{Defn}

\begin{prop}
\label{201911231957}
An $M$-valued volume $v$ on an abelian category $\mathcal{A}$ is an isomorphism invariant.
In other words, if objects $A, B$ of $\mathcal{A}$ are isomorphic to each other, then we have $v(A) = v(B)$.
\end{prop}
\begin{proof}
If we choose an isomorphism between $A$ and $B$, then we obtain an exact sequence $0 \to A \to B \to 0 \to 0$.
By the second axiom in Definition \ref{201911231935}, we obtain $v(B) = v(A) \cdot v(0)$.
Since $v(0) = 1$ by the first axiom in Definition \ref{201911231935}, we obtain $v(A) = v(B)$.
\end{proof}

\begin{prop}
An $M$-valued volume $v$ on an abelian category $\mathcal{A}$ is compatible with the direct sum $\oplus$ on the abelian category $\mathcal{A}$.
In other words, for objects $A,B$ of $\mathcal{A}$, we have $v(A \oplus B) = v(A) \cdot v(B)$.
\end{prop}
\begin{proof}
Note that we have an exact sequence $0 \to A \to A \oplus B \to B \to 0$.
By the second axiom in Definition \ref{201911231935}, we obtain $v(A \oplus B) = v(A) \cdot v(B)$.
\end{proof}

\subsection{Fredholm index}

In this subsection, we introduce a notion of {\it index} of morphisms in an abelian category.
By regarding objects of $\mathcal{A}$ with invertible volume as ``finite-dimensional objects'', we define a notion of  Fredholm morphisms in $\mathcal{B}$ and its index which is an invariant respecting compositions and robust to finite perturbations (see Definition \ref{201911232050}).
It generalizes the Fredholm index of Fredholm operator in the algebraic sense.

\begin{Defn}
\label{201912022112}
\rm
Let $\mathcal{B}$ be an abelian category and $\mathcal{A}$ be a abelian subcategory.
The abelian subcategory $\mathcal{A}$ is {\it closed under short exact sequences} if $A,C$ are objects of $\mathcal{A}$ and $B$ is an object of $\mathcal{B}$ for a short exact sequence $0 \to A \to B \to C \to 0$ in $\mathcal{B}$, then $B$ is an object of $\mathcal{A}$.
\end{Defn}

\begin{Defn}
\label{201911232050}
\rm
Let $\mathcal{B}$ be an abelian category and $\mathcal{A}$ be its abelian subcategory closed under short exact sequences.
Let $M$ be an abelian monoid and $v$ be an $M$-valued volume on $\mathcal{A}$.
For two objects $A,B$ of $\mathcal{B}$, a morphism $f : A \to B$ is {\it Fredholm with respect to the volume $v$} if $Ker(f)$ and $Cok(f)$ are essentially objects of $\mathcal{A}$ and the volumes $v(Ker(f)), v(Cok(f)) \in M$ are invertible.
For a Fredholm morphism $f : A \to B$, we define its {\it Fredholm index} by 
\begin{align}
Ind_{\mathcal{B},\mathcal{A},v} (f) \stackrel{\mathrm{def.}}{=} v(Cok(f)) \cdot v(Ker(f))^{-1} \in M .
\end{align}
\end{Defn}

\begin{Lemma}
\label{201911232046}
Let $A$ be an object of $\mathcal{B}$.
The identity $Id_A$ on $A$ is Fredholm.
We have $Ind_{\mathcal{B},\mathcal{A},v} (Id_A) =1 \in M$.
\end{Lemma}
\begin{proof}
It follows from the fact that $Ker (Id_A) = 0 = Cok (Id_A)$ whose volume is the unit $1 \in M$.
\end{proof}

\begin{Lemma}
\label{201911232047}
Let $f : A \to B$ and $g : B \to C$ be morphisms in $\mathcal{B}$.
If the morphisms $f,g$ are Fredholm, then the composition $g \circ f$ is Fredholm.
We have $Ind_{\mathcal{B},\mathcal{A},v} (g \circ f) = Ind_{\mathcal{B},\mathcal{A},v} (g) \cdot Ind_{\mathcal{B},\mathcal{A},v} (f) \in M$.
\end{Lemma}
\begin{proof}
We use the exact sequence $0 \to Ker (f) \to Ker (g\circ f) \to Ker (g) \stackrel{cok(f)\circ ker(g)}{\to} Cok (f) \to Cok (g\circ f) \to Cok (g) \to 0$.
Since $v( Ker (g)) \in M$ is invertible, any subobject of $Ker (g)$ has an invertible volume.
The volume $v(Ker ( cok(f) \circ ker(g))) \in M$ is invertible.
By the induced exact sequence $0 \to Ker (f) \to Ker (g\circ f) \to Ker ( cok(f) \circ ker(g)) \to 0$, we see that $v(Ker (g\circ f)) \in M$ is invertible.
Likewise, $v(Cok(g \circ f))$ is invertible.
Hence, the composition $g \circ f$ is Fredholm with respect to the volume $v$.
By repeating the second axiom of volumes in Definition \ref{201911231935}, we obtain 
\begin{align}
v(Ker (f)) \cdot v(Ker (g)) \cdot v(Cok (g\circ f)) = v(Ker (g\circ f)) \cdot v(Cok (f)) \cdot v(Cok (g)) .
\end{align}
It proves that $Ind_{\mathcal{B},\mathcal{A},v} (g \circ f) = Ind_{\mathcal{B},\mathcal{A},v} (g) \cdot Ind_{\mathcal{B},\mathcal{A},v} (f) \in M$.
\end{proof}

\begin{Defn}
\label{201912022319}
\rm
Let $\mathcal{B}$ be an abelian category and $\mathcal{A}$ be an abelian subcategory which is closed under short exact sequences.
Let $v$ be an $M$-valued volume on $\mathcal{A}$.
We define a category $\mathcal{A}^{Fr}$ as a subcategory of $\mathcal{B}$ formed by all Fredholm homomorphisms.
It is a well-defined category due to Lemma \ref{201911232046}, \ref{201911232047}.
\end{Defn}

\begin{prop}
\label{201911232139}
Every morphism $f : A \to B$ between objects with invertible volumes is Fredholm.
Then we have
\begin{align}
Ind_{\mathcal{B},\mathcal{A},v} ( f ) = v (B) \circ v ( A)^{-1} \in M . 
\end{align}
\end{prop}
\begin{proof}
If objects $A,B$ of $\mathcal{A}$ have invertible volumes, then for a morphism $f : A \to B$ its kernel and cokernel have invertible volumes due to the second axiom in Definition \ref{201911231935}.

By the exact sequence $0 \to Ker(f) \to A \stackrel{f}{\to} B \to Cok (f) \to 0$, we have $v(B) \cdot v(Ker(f)) = v(A) \cdot v(Cok(f))$.
We obtain $Ind_{\mathcal{B},\mathcal{A},v} ( f ) = v (B) \circ v ( A)^{-1}$.
\end{proof}

\subsection{Finite perturbation}

Consider an abelian category $\mathcal{B}$ and its abelian subcategory $\mathcal{A}$ closed under short exact sequences.
Let $v$ be an $M$-valued volume on the abelian category $\mathcal{A}$, not necessarily on $\mathcal{B}$ where $M$ is an abelian monoid.

\begin{Defn}
\label{201912041441}
\rm
Let $f$ be a morphism in $\mathcal{B}$.
A morphism $f$ in $\mathcal{B}$ is {\it finite with respect to the volume $v$} if the value of the image of $f$ (equivalently, the coimage of $f$) by $v$ is invertible in $M$.
In other words, the image $Im ( f )$ is essentially an object of $\mathcal{A}$ and the volume $v( Im (f)) \in M$ is invertible.
\end{Defn}

\begin{prop}[Invariance of index under finite perturbations]
\label{201912021443}
Let $f,k : A \to B$ be morphisms in $\mathcal{B}$.
If the morphism $f$ is Fredholm and the morphism $k$ is finite with respect to the volume $v$, then the morphism $(f + k) : A \to B$ is Fredholm with respect to the volume $v$.
Moreover, we have
\begin{align}
Ind_{\mathcal{B},\mathcal{A},v} ( f + k )
=
Ind_{\mathcal{B},\mathcal{A},v} ( f  ) \in M .
\end{align}
\end{prop}
\begin{proof}
Denote by $C$ the (co)image of the morphism $k : A \to B$.
Note that $(f+k)$ is decomposed into following morphisms :
\begin{align}
A
\stackrel{(id_A \oplus coim(k))\circ \Delta_A}{\longrightarrow}
A \oplus C
\stackrel{f \oplus id_C}{\longrightarrow}
B \oplus C
\stackrel{\nabla_B \circ ( id_B \oplus im(k))}{\longrightarrow}
B
.
\end{align}
Since the volume $v(C) \in M$ is invertible, the morphisms $(id_A \oplus coim(k))\circ \Delta_A$ and $\nabla_B \circ ( id_B \oplus im(k))$ are Fredholm with respect to the volume $v$.
Since the morphism $f$ is Fredholm with respect to the volume $v$, so the morphism $f \oplus id_C$ is.
By Lemma \ref{201911232047}, $(f+k)$ is Fredholm and,
\begin{align}
&Ind_{\mathcal{B},\mathcal{A},v} ( f + k ) \\
&=
Ind_{\mathcal{B},\mathcal{A},v} ( \nabla_B \circ ( id_B \oplus im(k)) )
\cdot
Ind_{\mathcal{B},\mathcal{A},v} ( f \oplus id_C )
\cdot
Ind_{\mathcal{B},\mathcal{A},v} ( (id_A \oplus coim(k))\circ \Delta_A ) .
\end{align}
Note that $Ind_{\mathcal{B},\mathcal{A},v} ( f \oplus id_C ) = Ind_{\mathcal{B},\mathcal{A},v} ( f  )$.
Moreover we have $Ind_{\mathcal{B},\mathcal{A},v} ( \nabla_B \circ ( id_B \oplus im(k)) ) \cdot Ind_{\mathcal{B},\mathcal{A},v} ( (id_A \oplus coim(k))\circ \Delta_A ) = v ( C)^{-1} \cdot v(C) = 1$ by definitions.
It completes the proof.
\end{proof}


\section{Applications to the category $\mathsf{Hopf}^\mathsf{bc}(\mathcal{C})$}
\label{202002201415}

In this section, we give an application of the previous results to the category of bicommutative Hopf monoids $\mathsf{Hopf}^\mathsf{bc}(\mathcal{C})$.
We show that the inverse volume gives a volume on some abelian category.
From now on, we consider the following assumptions on $\mathcal{C}$.
\begin{itemize}
\item
(Assumption 0)
The category $\mathcal{C}$ has any equalizer and coequalizer.
\item
(Assumption 1)
The monoidal structure of $\mathcal{C}$ is bistable.
\item
(Assumption 2)
The category $\mathsf{Hopf}^\mathsf{bc}(\mathcal{C})$ is an abelian category.
\end{itemize}
Here, (co, bi)stability of the monoidal structure of $\mathcal{C}$ is introduced in  Definition \ref{201907230933}.
See Definition \ref{201907311409} for the definition of inverse volume.

\begin{Example}
\label{202002271704}
Note that the assumptions on $\mathcal{C}$ automatically hold if $\mathcal{C} = \mathsf{Vec}^\otimes_{k}$, the category of (possibly, infinite-dimensional) vector spaces over a field ${k}$.
The (Assumption 1) follows from Proposition \ref{comp_monoidal_stab} (see Example \ref{201912022052}).
The (Assumption 2) follows from the fact that $\mathsf{Hopf}^\mathsf{bc} ( \mathsf{Vec}^\otimes_{k})$ is an abelian category by Corollary 4.16 in \cite{takeuchi1972correspondence} or Theorem 4.3 in \cite{newman1975correspondence}.
\end{Example}

\begin{remark}
We remark a relationship between the assumptions.
(Assumption 0,1) implies that the category $\mathsf{Hopf}^\mathsf{bc}(\mathcal{C})$ is an pre-abelian category i.e. an additive category with arbitrary kernel and cokernel.
Under (Assumption 0,1), (Assumption 2) is equivalent with the fundamental theorem on homomorphisms.
\end{remark}

\begin{remark}
We need those (Assumption 0,1,2) because we use the following properties :
\begin{enumerate}
\item
By (Assumption 0), every idempotent in $\mathcal{C}$ is a split idempotent due to Proposition \ref{201912021623}.
By Corollary \ref{201912021909}, a bimonoid $A$ in $\mathcal{C}$ is bismall if and only if $A$ has a normalized integral and a normalized cointegral.
By Corollary \ref{201907101753}, the full subcategory of bismall bimonoids in the symmetric monoidal category $\mathcal{C}$ gives a sub symmetric monoidal category of $\mathsf{Bimon}(\mathcal{C})$.
\item
We need (Assumption 1) to make use of Proposition \ref{stab_gives_cokernel}, i.e. every homomorphism in $\mathsf{Hopf}^\mathsf{bc}(\mathcal{C})$ is binormal.
\item
Recall Definition \ref{201908041300}.
Furthermore, due to (Assumption 0, 1), every homomorphism in $\mathsf{Hopf}^\mathsf{bc}(\mathcal{C})$ is well-decomposable by definition.
\item
From (Assumption 2), we obtain the following exact sequence :
For bicommutative Hopf monoids $A,B,C$ in $\mathcal{C}$ and homomorphisms $\xi : A \to B$, $\xi^\prime : B \to C$, we have an exact sequence,
\begin{align}
\label{201907312237}
\mathds{1} 
\to
Ker ( \xi )
\to
Ker ( \xi^\prime \circ \xi )
\to
Ker ( \xi^\prime )
\to
Cok (\xi)
\to
Cok ( \xi^\prime \circ \xi )
\to
Cok ( \xi^\prime )
\to
\mathds{1} 
\end{align}
Note that until this subsection, we use the notation $Ker(\xi), Cok(\xi)$ for the kernel and cokernel in $\mathsf{Bimon}(\mathcal{C})$.
See Definition \ref{201908040955}.
In (\ref{201907312237}), $Ker(\xi)$, $Cok(\xi)$ denote a kernel and a cokernel in $\mathsf{Hopf}^\mathsf{bc}(\mathcal{C})$.
In fact, these coincide with each other due to (Assumption 1) and Corollary \ref{201908040952}.
\end{enumerate}
\end{remark}


\subsection{A volume on $\mathsf{Hopf^{bc,bs}} ( \mathcal{C} )$}
\label{201911161328}

\begin{prop}
\label{201907312228}
Let $A,B,C$ be bicommutative Hopf monoids.
Let $\xi : A \to B$, $\xi^\prime : B \to C$ be bimonoid homomorphism.
If the bimonoid homomorphisms $\xi, \xi^\prime$ are pre-Fredholm, then the composition $\xi^\prime\circ\xi$ is pre-Fredholm.
Moreover we have,
\begin{align}
vol^{-1} ( Ker (\xi) ) \circ vol^{-1} ( Ker ( \xi^\prime)) = \langle cok (\xi ) \circ ker (\xi^\prime) \rangle \circ vol^{-1} ( Ker ( \xi^\prime \circ \xi ) ) ,  \\
vol^{-1} ( Cok (\xi) ) \circ vol^{-1} ( Cok ( \xi^\prime)) = \langle cok (\xi ) \circ ker (\xi^\prime) \rangle \circ vol^{-1} ( Cok ( \xi^\prime \circ \xi ) ) . 
\end{align}
\end{prop}
\begin{proof}
Recall that we have an exact sequence (\ref{201907312237}).
By Theorem \ref{201907312238}, the Hopf monoids $Cok (\xi^\prime \circ \xi)$, $Ker (\xi^\prime \circ \xi)$ are bismall since the Hopf monoids $Ker ( \xi), Ker ( \xi^\prime)$ and cokernels $Cok ( \xi), Cok ( \xi^\prime)$ are bismall.
Hence, the composition $\xi^\prime\circ\xi$ is pre-Fredholm.

We prove the first equation.
Denote by $\varphi = cok (\xi) \circ ker (\xi^\prime) : Ker (\xi^\prime) \to Cok (\xi)$.
From the exact sequence (\ref{201907312237}), we obtain an exact sequence,
\begin{align}
\mathds{1} 
\to
Ker ( \xi )
\to
Ker ( \xi^\prime \circ \xi )
\to
Ker ( \xi^\prime )
\to
Im ( \varphi )
\to
\mathds{1} 
\end{align}
We apply Corollary \ref{201907312239} by assuming $A,B,\xi$ in Corollary \ref{201907312239} are $Ker ( \xi^\prime \circ \xi ), Ker ( \xi^\prime)$ and the homomorphism $Ker ( \xi^\prime \circ \xi) \to Ker ( \xi^\prime)$.
In fact, the first assumption in Corollary \ref{201907312239} follows from (Assumption 1).
The second and fourth assumptions in Corollary \ref{201907312239} follows from Theorem \ref{201907292156}.
The third assumption is already proved as before.
Then we obtain,
\begin{align}
vol^{-1} ( Ker (\xi) ) \circ vol^{-1} ( Ker ( \xi^\prime)) = vol^{-1} ( Im ( \varphi ))  \circ vol^{-1} ( Ker ( \xi^\prime \circ \xi ) ) .
\end{align}
By Proposition \ref{201907312242}, we have $\langle \varphi \rangle = vol^{-1} ( Im (\varphi))$ so that it completes the first equation.
The second equation is proved analogously.
\end{proof}

\begin{prop}
\label{202002201252}
The subcategory $\mathsf{Hopf}^\mathsf{bc,bs} ( \mathcal{C} )$ is an abelian subcategory of the abelian category $ \mathsf{Hopf}^\mathsf{bc} ( \mathcal{C} )$.
\end{prop}
\begin{proof}
Let $A,B$ be bicommutative bismall Hopf monoids.
Let $\xi : A \to B$ be a bimonoid homomorphism, i.e. a morphism in $ \mathsf{Hopf}^\mathsf{bc} ( \mathcal{C} )$.
We have an exact sequence,
\begin{align}
\mathds{1}
\to
\mathds{1}
\to
Ker ( \xi )
\stackrel{ker(\xi)}{\to}
A 
\stackrel{\xi}{\to}
B
.
\end{align}
Due to (Assumption 1) and (Assumption 2), we can apply Theorem \ref{201907312238}.
By Theorem \ref{201907312238}, the kernel Hopf monoid $Ker ( \xi)$ is bismall.
Analogously, the cokernel Hopf monoid $Cok ( \xi )$ is bismall.
It completes the proof.
\end{proof}

\begin{Defn}
\label{201912022344}
\rm
Let $End_{\mathcal{C}} ( \mathds{1} )$ be the set of endomorphism on the unit object $\mathds{1}$.
Note that the composition induces an abelian monoid structure on the set $End_{\mathcal{C}} ( \mathds{1} )$.
We denote by $M_\mathcal{C}$ the smallest submonoid of $End_\mathcal{C} ( \mathds{1} )$ containing all $f \in End_\mathcal{C} ( \mathds{1})$ such that $f = vol^{-1} ( A)$ or $f \circ vol^{-1} (A) = id_{\mathds{1}} = vol^{-1} (A) \circ f$ for some bicommutative bismall Hopf monoid $A$.
Denote by $M^{-1}_\mathcal{C}$ the submonoid consisting of invertible elements in the monoid $M_\mathcal{C}$, i.e. $M^{-1}_\mathcal{C} = M_\mathcal{C} \cap Aut_\mathcal{C} ( \mathds{1} )$.
\end{Defn}

\begin{theorem}
\label{201912022324}
The assignment $vol^{-1}$ of inverse volumes is a $M_\mathcal{C}$-valued volume on the abelian category $\mathsf{Hopf}^\mathsf{bc,bs} ( \mathcal{C} )$.
\end{theorem}
\begin{proof}
Put $v = vol^{-1}$.
The unit Hopf monoid $\mathds{1}$ is a zero object of $\mathsf{Hopf}^\mathsf{bc,bs}(\mathcal{C})$.
By the first part of Proposition \ref{201908010934}, we have $v(\mathds{1}) = vol^{-1} ( \mathds{1}) \in M_\mathcal{C}$ is the unit of $M_\mathcal{C}$.

Let $\mathds{1} \to A \stackrel{f}{\to} B \stackrel{g}{\to} C \to \mathds{1}$ be an exact sequence in the abelian category $\mathcal{A} = \mathsf{Hopf}^\mathsf{bc,bs} ( \mathcal{C} )$.
We apply the first equation in Theorem \ref{201907312228} by considering $\xi = g$ and $\xi^\prime = \epsilon_C$.
In fact, $B,C,\mathds{1}$ are bismall bimonoids, the homomorphisms $g$ and $\epsilon_C$ are pre-Fredholm.
We obtain
\begin{align}
vol^{-1} ( Ker (g)) \circ vol^{-1} (Ker (\epsilon_C)) = \langle cok ( g) \circ ker ( \epsilon_C) \rangle \circ vol^{-1} (Ker (\epsilon_B)) .
\end{align}
By the exactness, we have $A \cong Ker ( g)$ and $Cok ( g) \cong \mathds{1}$.
Moreover we have $Ker (\epsilon_C) \cong C$ and $Ker (\epsilon_B) \cong B$.
Hence, we obtain $\langle cok ( g) \circ ker ( \epsilon_C) \rangle = id_{\mathds{1}}$ so that $vol^{-1} ( A ) \cdot vol^{-1} (C) = vol^{-1} ( B)$.
It completes the proof.
\end{proof}


\subsection{Functorial integral}
\label{201908010930}

\begin{Defn}
\rm
\begin{enumerate}
\item
Recall Definition \ref{201911232050}.
For two bicommutative Hopf monoids $A,B$, a bimonoid homomorphism $\xi : A \to B$ is {\it Fredholm} if it is Fredholm with respect to the inverse volume $vol^{-1}$.
In other words, the homomorphism $\xi$ is pre-Fredholm, and its kernel Hopf monoid and cokernel Hopf monoid have finite volumes.
For a Fredholm homomorphism $\xi : A \to B$ between bicommutative Hopf monoids, we denote by $Ind (\xi ) \stackrel{\mathrm{def.}}{=} Ind_{\mathcal{B},\mathcal{A},v} (\xi)$ for $\mathcal{B} = \mathsf{Hopf}^\mathsf{bc}(\mathcal{C})$, $\mathcal{A} = \mathsf{Hopf}^\mathsf{bc,bs}(\mathcal{C})$, $M= M_\mathcal{C}$ and $v = vol^{-1}$.
\item
Denote by $\mathsf{Hopf}^\mathsf{bc,Fr} ( \mathcal{C} )$ the category consisting of Fredholm homomorphisms between bicommutative Hopf monoids.
Recall Definition \ref{201912022319}.
For $\mathcal{B} = \mathsf{Hopf}^\mathsf{bc}(\mathcal{C})$, $\mathcal{A} = \mathsf{Hopf}^\mathsf{bc,bs}(\mathcal{C})$, $M= M_\mathcal{C}$ and $v = vol^{-1}$, the category $\mathsf{Hopf}^\mathsf{bc,Fr} ( \mathcal{C} )$ is defined by $\mathsf{Hopf}^\mathsf{bc,Fr} ( \mathcal{C} ) \stackrel{\mathrm{def.}}{=} \mathcal{A}^{Fr}$.
\item
Let $\xi :A \to B$ be a homomorphism between bicommutative Hopf monoids.
The homomorphism $\xi$ is {\it finite} if the morphism $\xi$ in $\mathsf{Hopf}^\mathsf{bc}$ is finite with respect to the volume $vol^{-1}$.
See Definition \ref{201912041441}.
\end{enumerate}
\end{Defn}

\begin{prop}
\label{202002271226}
\begin{enumerate}
\item
For a bicommutative Hopf monoid $A$, the identity $id_A$ is Fredholm and we have
$Ind ( id_A ) = id_{\mathds{1}} \in M^{-1}_{\mathcal{C}}$.
\item
For Fredholm homomorphisms $\xi : A \to B$ and $\xi^\prime : B \to C$ between bicommutative Hopf monoids, the composition $\xi^\prime\circ\xi$ is Fredholm and we have
$Ind (\xi^\prime \circ \xi ) = Ind ( \xi^\prime ) \circ Ind ( \xi ) \in M^{-1}_{\mathcal{C}}$.
\item
For a Fredholm homomorphism $\xi : A \to B$ and a finite homomorphism $\epsilon : A \to B$, the convolution $\xi \ast \epsilon$ is Fredholm and we have
$Ind ( \xi \ast \epsilon ) = Ind ( \xi )\in M^{-1}_{\mathcal{C}}$.
\end{enumerate}
\end{prop}
\begin{proof}
The first part follows from Lemma \ref{201911232046}.
The second part follows from Lemma \ref{201911232047}.
The third part follows from Proposition \ref{201912021443}.
\end{proof}

\begin{Defn}
\rm
We define a 2-cochain $\omega_\mathcal{C}$ of the symmetric monoidal category $\mathsf{Hopf}^\mathsf{bc,Fr} ( \mathcal{C} )$ with coefficients in the abelian group $M^{-1}_\mathcal{C}$.
Let $\xi : A \to B,\xi^\prime : B \to C$ be composable Fredholm homomorphisms between bicommutative Hopf monoids.
We define 
\begin{align}
\omega_\mathcal{C} ( \xi , \xi^\prime ) \stackrel{\mathrm{def.}}{=} \langle cok (\xi) \circ ker ( \xi^\prime ) \rangle \in M^{-1}_\mathcal{C}  . 
\end{align}
\end{Defn}

\begin{prop}
The 2-cochain $\omega_\mathcal{C}$ is a 2-cocycle.
\end{prop}
\begin{proof}
The 2-cocycle condition is immediate from the associativity of compositions.
In fact, $\mu_{  \xi^{\prime\prime} } 
\circ
\left( \mu_{  \xi^\prime } \circ  \mu_{   \xi } \right)
=
\left( \mu_{   \xi^{\prime\prime} } \circ  \mu_{   \xi^{\prime} } \right) \circ \mu_{   \xi }$
implies, 
\begin{align}
( \omega_\mathcal{C} ( \xi, \xi^\prime) \circ \omega_\mathcal{C} ( \xi^\prime \circ \xi, \xi^{\prime\prime} ) ) \cdot \mu_{\xi^{\prime\prime} \circ \xi^\prime \circ \xi} = ( \omega_\mathcal{C} (\xi^\prime,  \xi^{\prime\prime} ) \circ \omega_\mathcal{C} ( \xi, \xi^{\prime\prime} \circ \xi^\prime ) ) \cdot \mu_{\xi^{\prime\prime} \circ \xi^\prime \circ \xi}. 
\end{align}
Here, we use Theorem \ref{201907311424} where the assumptions in Theorem are deduced from (Assumption 0, 1).
By Theorem \ref{201906281559}, we obtain 
\begin{align}
\omega_\mathcal{C} ( \xi, \xi^\prime) \circ \omega_\mathcal{C} ( \xi^\prime \circ \xi, \xi^{\prime\prime} )  = \omega_\mathcal{C} (\xi^\prime,  \xi^{\prime\prime} ) \circ \omega_\mathcal{C} ( \xi , \xi^{\prime\prime} \circ \xi^\prime )  .
\end{align}
It proves that the 2-cochain $\omega_\mathcal{C}$ is a 2-cocycle.

Moreover we have $\omega_\mathcal{C} ( id_B , \xi ) = id_\mathds{1} = \omega_\mathcal{C} ( \xi , id_A ) $ by definitions.
Hence, the 2-cocycle $\omega_\mathcal{C}$ is normalized.
It completes the proof.
\end{proof}

\begin{Defn}
\rm
We define a 2-cohomology class $o_\mathcal{C} \in H^2_{nor} ( \mathsf{Hopf}^\mathsf{bc,Fr} ( \mathcal{C} ) ; M^{-1}_\mathcal{C})$ by the class of the 2-cocycle $\omega_\mathcal{C}$.
\end{Defn}

\begin{prop}
\label{201908020848}
We have $o_\mathcal{C} = 1 \in H^2_{nor} ( \mathsf{Hopf}^\mathsf{bc,Fr} ( \mathcal{C} ) ; M^{-1}_\mathcal{C})$.
In particular, the induced 2-cohomology class $o_\mathcal{C} \in H^2_{nor} ( \mathsf{Hopf}^\mathsf{bc,Fr} ( \mathcal{C} ) ; Aut_\mathcal{C} (\mathds{1}) )$ by $M^{-1}_\mathcal{C} \subset Aut_\mathcal{C} ( \mathds{1} )$ is trivial.
\end{prop}
\begin{proof}
Choose $\upsilon$ defined by $\upsilon ( \xi ) = vol^{-1} ( Ker ( \xi ))$.
Then the first equation in Theorem \ref{201907312228} proves the claim.
\end{proof}

\begin{Defn}[Functorial integral]
\label{201907312157}
\rm
Let $\upsilon$ be a normalized 1-cochain with coefficients in the abelian group $Aut_\mathcal{C}(\mathds{1})$ such that $\delta^1 \upsilon = \omega_\mathcal{C}$.
Let $\xi : A \to B$ be a Fredholm bimonoid homomorphism between bicommutative Hopf monoids.
Recall $\mu_\xi$ in Definition \ref{201907312115}.
We define a morphism $\xi_! : B \to A$ in $\mathcal{C}$ by
\begin{align}
\xi_! \stackrel{\mathrm{def.}}{=} \upsilon ( \xi )^{-1} \cdot \mu_\xi . 
\end{align}
\end{Defn}

\begin{prop}
Let $A$ be a bicommutative Hopf monoid.
Note that the identity $id_A$ is Fredholm.
We have,
\begin{align}
(id_A)_! = id_A . 
\end{align}
\end{prop}
\begin{proof}
It follows from $\upsilon (id_A ) = id_{\mathds{1}}$.
\end{proof}

\begin{prop}
\label{201908011125}
Let $A,B,C$ be bicommutative Hopf monoids.
Let $\xi : A \to B, \xi^\prime : B \to C$ be bimonoid homomorphisms.
If $\xi, \xi^\prime$ are Fredholm, then the composition $\xi^\prime \circ \xi$ is Fredholm and we have
\begin{align}
\label{201907312213}
(\xi^\prime \circ \xi )_!  = \xi_! \circ \xi^\prime_! . 
\end{align}
\end{prop}
\begin{proof}
By Theorem \ref{201907312228}, we have
\begin{align}
(\xi^\prime \circ \xi )_! 
&= 
\upsilon ( \xi^\prime \circ \xi )^{-1} \cdot \mu_{\xi^\prime \circ \xi}  \\
&= 
\left( \upsilon ( \xi^\prime \circ \xi )^{-1} \circ \omega (\xi^\prime , \xi )^{-1} \right) \cdot \left( \mu_\xi \circ \mu_{\xi^\prime} \right)  \\
&=
\left( \upsilon ( \xi )^{-1} \circ \upsilon ( \xi^\prime )^{-1} \right) \cdot \left( \mu_\xi \circ \mu_{\xi^\prime} \right)  \\
&=
\xi_! \circ \xi^\prime_! . 
\end{align}
\end{proof}

\begin{Defn}
\rm
We define a normalized 1-cochain $\upsilon_0$ with coefficients in $M^{-1}_\mathcal{C}$.
For a Fredholm homomorphism $\xi$, we define $\upsilon_0 ( \xi ) \stackrel{\mathrm{def.}}{=} vol^{-1} ( Ker ( \xi ) )$.
We define another normalized 1-cochain $\upsilon_1$ with coefficients in $M^{-1}_\mathcal{C}$ by $\upsilon_1 ( \xi ) \stackrel{\mathrm{def.}}{=} vol^{-1} ( Cok ( \xi ) )$.
They satisfy $\delta^1 \upsilon_0 = \omega_\mathcal{C} = \delta^1 \upsilon_1$.
\end{Defn}

\begin{theorem}
\label{201908011123}
Consider $\upsilon = \upsilon_0$ ($\upsilon = \upsilon_1$, resp.) in Definition \ref{201907312157}.
Let $A,B, C, D$ be bicommutative Hopf monoids.
Consider a commutative diagram of Fredholm bimonoid homomorphisms.
Suppose that 
\begin{itemize}
\item
the induced bimonoid homomorphism $Ker ( \varphi ) \to Ker ( \psi )$ is an isomorphism (an epimorphism resp.) in $\mathsf{Hopf}^\mathsf{bc} ( \mathcal{C} )$.
\item
the induced bimonoid homomorphism $Cok ( \varphi ) \to Cok ( \psi )$ is a monomorphism (an isomorphism, resp.) in $\mathsf{Hopf}^\mathsf{bc} ( \mathcal{C} )$.
\end{itemize}
Then we have $\varphi^\prime \circ \varphi_! = \psi_! \circ \psi^\prime$.
$$
\begin{tikzcd}
A \ar[r, "\varphi^\prime"] \ar[d, "\varphi"] & C \ar[d, "\psi"] \\
B \ar[r, "\psi^\prime"] & D
\end{tikzcd}
$$
\end{theorem}
\begin{proof}
We prove the case $\upsilon = \upsilon_0$ and leave to the readers the case $\upsilon = \upsilon_1$.
Note that there exists a section of the induced bimonoid homomorphism $\varphi^{\prime\prime} : Ker ( \varphi ) \to Ker ( \psi )$ in $\mathcal{C}$ since $\varphi^{\prime\prime}$ is an isomorphism in $\mathsf{Hopf}^\mathsf{bc} (\mathcal{C})$, in particular in $\mathcal{C}$.
Moreover, the induced morphism $\psi^{\prime\prime} : Cok ( \varphi) \to Cok ( \psi)$ has a retract in $\mathcal{C}$.
In fact, since $\psi^{\prime\prime}$ is a monomorphism, there exists a morphism $\xi$ in $\mathsf{Hopf}^\mathsf{bc}(\mathcal{C})$ such that $ker ( \xi ) = \psi^{\prime\prime}$.
By Lemma \ref{201907311900}, $\tilde{\mu}_{ker(\xi)} \circ \psi^{\prime\prime} = id_{Cok (\varphi )}$.

By Theorem \ref{key_theorem_existence_integral}, we have $\mu_\psi \circ \psi^\prime = \varphi^\prime \circ \mu_\varphi$.
Since $\upsilon_0 ( \varphi ) = vol^{-1} (Ker (\varphi))$, $\upsilon_0 ( \psi ) = vol^{-1} (Ker (\psi))$ and $\varphi^{\prime\prime}$ is an isomorphism, we have $\upsilon_0 ( \varphi ) = \upsilon_0 ( \psi )$.
By definitions, we obtain $ \psi_! \circ \psi^\prime = \varphi^\prime \circ \varphi_!$.
\end{proof}


\bibliography{Note_on_integrals}{}

\begin{thebibliography}{10}

\bibitem{MS}
Marcelo Aguiar and Swapneel~Arvind Mahajan.
\newblock {\em Monoidal functors, species and {H}opf algebras}, volume~29.
\newblock American Mathematical Society Providence, RI, 2010.

\bibitem{BalKir}
Benjamin Balsam and Alexander Kirillov~Jr.
\newblock Kitaev's lattice model and {T}uraev-{V}iro {TQFT}s.
\newblock {\em arXiv preprint arXiv:1206.2308}, 2012.

\bibitem{BW}
John Barrett and Bruce Westbury.
\newblock Invariants of piecewise-linear 3-manifolds.
\newblock {\em Transactions of the American Mathematical Society},
  348(10):3997--4022, 1996.

\bibitem{bespalov2000integrals}
Yuri Bespalov, Thomas Kerler, Volodymyr Lyubashenko, and Vladimir Turaev.
\newblock Integrals for braided {H}opf algebras.
\newblock {\em Journal of Pure and Applied Algebra}, 148(2):113--164, 2000.

\bibitem{BMCA}
Oliver Buerschaper, Juan~Mart{\'\i}n Mombelli, Matthias Christandl, and Miguel
  Aguado.
\newblock A hierarchy of topological tensor network states.
\newblock {\em Journal of Mathematical Physics}, 54(1):012201, 2013.

\bibitem{DW}
Robbert Dijkgraaf and Edward Witten.
\newblock Topological gauge theories and group cohomology.
\newblock {\em Communications in Mathematical Physics}, 129(2):393--429, 1990.

\bibitem{FQ}
Daniel~S Freed and Frank Quinn.
\newblock Chern-{S}imons theory with finite gauge group.
\newblock {\em Communications in Mathematical Physics}, 156(3):435--472, 1993.

\bibitem{kato2013perturbation}
Tosio Kato.
\newblock {\em Perturbation theory for linear operators}, volume 132.
\newblock Springer Science \& Business Media, 2013.

\bibitem{KTLV}
Thomas Kerler and Volodymyr~V Lyubashenko.
\newblock {\em Non-semisimple topological quantum field theories for
  3-manifolds with corners}, volume 1765.
\newblock Springer Science \& Business Media, 2001.

\bibitem{kim2020family}
Minkyu Kim.
\newblock A family of {TQFT}'s associated with homology theory.
\newblock {\em arXiv preprint arXiv:2006.10438}, 2020.

\bibitem{kuperberg1991involutory}
Greg Kuperberg.
\newblock Involutory {H}opf algebras and 3-manifold invariants.
\newblock {\em International Journal of Mathematics}, 2(01):41--66, 1991.

\bibitem{kuperberg1997non}
Greg Kuperberg.
\newblock Non-involutory {H}opf algebras and 3-manifold invariants.
\newblock {\em arXiv preprint q-alg/9712047}, 1997.

\bibitem{LarSwe}
Richard~Gustavus Larson and Moss~Eisenberg Sweedler.
\newblock An associative orthogonal bilinear form for {H}opf algebras.
\newblock {\em American Journal of Mathematics}, 91(1):75--94, 1969.

\bibitem{mac}
Saunders Mac~Lane.
\newblock {\em Categories for the working mathematician}, volume~5.
\newblock Springer Science \& Business Media, 2013.

\bibitem{meusburger}
Catherine Meusburger.
\newblock Kitaev lattice models as a {H}opf algebra gauge theory.
\newblock {\em Communications in Mathematical Physics}, 353(1):413--468, 2017.

\bibitem{Mil}
John~W Milnor and John~C Moore.
\newblock On the structure of {H}opf algebras.
\newblock {\em Annals of Mathematics}, pages 211--264, 1965.

\bibitem{newman1975correspondence}
Kenneth Newman.
\newblock A correspondence between bi-ideals and sub-{H}opf algebras in
  cocommutative {H}opf algebras.
\newblock {\em Journal of Algebra}, 36(1):1--15, 1975.

\bibitem{radford1976order}
David~E Radford.
\newblock The order of the antipode of a finite dimensional {H}opf algebra is
  finite.
\newblock {\em American Journal of Mathematics}, pages 333--355, 1976.

\bibitem{sweedler1969integrals}
Moss~Eisenberg Sweedler.
\newblock Integrals for {H}opf algebras.
\newblock {\em Annals of Mathematics}, pages 323--335, 1969.

\bibitem{takeuchi1972correspondence}
Mitsuhiro Takeuchi.
\newblock A correspondence between {H}opf ideals and sub-{H}opf algebras.
\newblock {\em manuscripta mathematica}, 7(3):251--270, 1972.

\bibitem{TV}
Vladimir~G Turaev and Oleg~Ya Viro.
\newblock State sum invariants of 3-manifolds and quantum 6j-symbols.
\newblock {\em Topology}, 31(4):865--902, 1992.

\bibitem{Wakui}
Michihisa Wakui.
\newblock On {D}ijkgraaf-{W}itten invariant for 3-manifolds.
\newblock {\em Osaka Journal of Mathematics}, 29(4):675--696, 1992.

\end{thebibliography}
\bibliographystyle{plain}

\end{document}